\documentclass[10pt,reqno]{amsart}

\usepackage{amsthm,amsmath,amstext,amssymb,amscd,euscript, mathrsfs, dsfont,multicol,times,enumerate,subfig,sidecap}

\usepackage{comment}
\excludecomment{comm}

\usepackage{enumerate}
\usepackage{amsmath, amssymb, amsthm}
\usepackage{mathrsfs}
\usepackage{esint}
\usepackage{xcolor}
\usepackage[normalem]{ulem}
\usepackage{mathtools}
\usepackage{bm}
\usepackage{bbm}

\usepackage[hidelinks, pdfstartview=FitV]{hyperref} 

\usepackage[left=2.5cm,right=2.5cm,top=3cm,bottom=3cm]{geometry}

\newtheorem{thm}{Theorem}[section]
\newtheorem{corollary}[thm]{Corollary}
\newtheorem{lem}[thm]{Lemma}
\newtheorem{prop}[thm]{Proposition}
\theoremstyle{definition}
\newtheorem{defn}[thm]{Definition}

\newtheorem{assumption}[thm]{Assumption}
\newtheorem{rem}[thm]{Remark}

\newcommand\bN{\mathbb{N}}

\newcommand\bR{\mathbb{R}}

\newcommand\cB{\mathcal{B}}

\newcommand\cD{\mathcal{D}}
\newcommand\cF{\mathcal{F}}

\newcommand\cH{\mathcal{H}}

\newcommand\cS{\mathcal{S}}

\newcommand{\p}{\partial}
\DeclareMathOperator*{\esssup}{ess\,sup}

\newcommand{\norm}[1]{\left\Vert#1\right\Vert}
\newcommand\cpar{\text{$[$\kern-.38em$[$}}
\newcommand\cbrk{\text{$]$\kern-.15em$]$}}
\newcommand\opar{\text{\,\raise.2ex\hbox{${\scriptstyle
|}$}\kern-.34em$($}}
\newcommand\obrk{\text{$)$\kern-.34em\raise.2ex\hbox{${\scriptstyle |}$}}\,}

\def\XXint#1#2#3{{\setbox0=\hbox{$#1{#2#3}{\int}$ }
\vcenter{\hbox{$#2#3$ }}\kern-.58\wd0}}

\makeatletter
\def\@tocline#1#2#3#4#5#6#7{\relax
  \ifnum #1>\c@tocdepth 
  \else
    \par \addpenalty\@secpenalty\addvspace{#2}%
    \begingroup \hyphenpenalty\@M
    \@ifempty{#4}{%
      \@tempdima\csname r@tocindent\number#1\endcsname\relax
    }{%
      \@tempdima#4\relax
    }%
    \parindent\z@ \leftskip#3\relax \advance\leftskip\@tempdima\relax
    \rightskip\@pnumwidth plus4em \parfillskip-\@pnumwidth
    #5\leavevmode\hskip-\@tempdima
      \ifcase #1
       \or\or \hskip 1em \or \hskip 2em \else \hskip 3em \fi%
      #6\nobreak\relax
    \dotfill\hbox to\@pnumwidth{\@tocpagenum{#7}}\par
    \nobreak
    \endgroup
  \fi}
\makeatother

\newcommand{\nrm}[1]{\Vert#1\Vert}
\newcommand{\abs}[1]{\left\vert#1\right\vert}
\newcommand{\set}[1]{\left\{#1\right\}}

\newcommand{\tld}[1]{\widetilde{#1}}

\newcommand{\ift}{\infty}
\newcommand{\alp}{\alpha}
\newcommand{\bt}{\beta}
\newcommand{\gmm}{\gamma}

\newcommand{\lmb}{\lambda}

\newcommand{\tht}{\theta}

\newcommand{\bfb}{{\bf b}}

\newcommand{\bff}{{\bf f}}
\newcommand{\bfg}{{\bf g}}

\newcommand{\bfB}{{\bf B}}

\newcommand{\embed}{\hookrightarrow}

\newcommand{\f}[2]{\frac{#1}{#2}}       

\newcommand{\q}{\text{ }}       
\newcommand{\qd}{\quad }

\newcommand{\bfone}{\mathbf{1}}

\newcommand{\paren}[1]{\left(#1\right)}

\newcommand{\domain}{\mathbb{R}^{d+1}}
\newcommand{\dd}{\mathrm{d} }

\begin{document}

\title[Fourier--Barron Regularity for Fractional Parabolic PDEs]{Fractional Parabolic Partial Differential Equations in Anisotropic Spectral Barron Spaces: Regularity and Neural Approximation}

\author[J.-H. Choi]{Jae-Hwan Choi}
\address{Department of Mathematics, The Catholic University of Korea, 43 Jibong-ro, Bucheon-si, Gyeonggi-do, 14662, Republic of Korea}
\email{mathjhchoi@catholic.ac.kr}
\thanks{}

\author[H. Lim]{Hyojae Lim}
\address{Center for Artificial Intelligence and Natural Sciences,
Korea Institute for Advanced Study,\newline
85 Hoegi-ro, Dongdaemun-gu, Seoul, 02455, Republic of Korea}
\email{hlim@kias.re.kr}
\thanks{}

\author[J. Seo]{Jinsol Seo}
\address{School of Mathematics, Korea Institute for Advanced Study, 85 Hoegi-ro, Dongdaemun-gu, Seoul, 02455, Republic of Korea}
\email{seo9401@kias.re.kr}
\thanks{}

\author[Y.-J. Sim]{Young-Jin Sim}
\address{School of Mathematics, Korea Institute for Advanced Study, 85 Hoegi-ro, Dongdaemun-gu, Seoul, 02455, Republic of Korea}
\email{yjsim@kias.re.kr}
\thanks{}

\author[C. Song]{Changhoon Song}
\address{Research Institute of Mathematics, Seoul National University, 1 Gwanak-ro, Gwanak-gu, Seoul, 08826, South Korea}
\email{changhoon.song93@snu.ac.kr}
\thanks{}



\begin{abstract}
We study fractional parabolic initial-value problems with lower-order drift and potential terms in anisotropic spectral Barron spaces, defined by weighted space--time Fourier $L^1$ norms adapted to parabolic scaling.
We prove existence, uniqueness, and maximal regularity with a gain of one derivative in time and $\gamma$ derivatives in space, where $\gamma>0$ is the order of the fractional Laplacian.
A key difficulty is that the evolution is defined only for $t\geq0$, whereas the finite-time norm requires a global space--time extension with sufficient temporal Fourier decay.
We construct a finite reflected semigroup extension using a Vandermonde system to match derivatives at $t=0$, obtaining temporal Fourier estimates uniform in the semigroup parameter.
Combined with Fourier multiplier estimates for the damped principal operator, this extension yields maximal regularity for the principal equation.
Dimension-independent multiplication estimates support a finite regularity bootstrap, while interpolation and a sufficiently large damping parameter allow the lower-order terms to be absorbed in the base estimate.
The resulting a priori estimate allows us to apply the method of continuity and obtain the desired maximal regularity theorem without smallness assumptions on the lower-order coefficients.
We also construct a frequency-localized counterexample showing that a uniform-in-time spatial Barron bound on the forcing does not imply the corresponding two-derivative bound on the solution, even for the one-dimensional heat equation.

Combining this regularity with Fourier sampling, we obtain $n^{-1/2}$ approximation rates for the solution in mixed space--time Sobolev norms using shallow neural networks with suitable activations.
Finally, sampling in a product Hilbert space yields a population-level PINN consistency estimate for shallow cosine networks on a bounded observation cylinder.
There exists a single width-$n$ network for which the sum of the squared mixed-Sobolev solution error, the squared $L^2$-norm of the residual for the whole-space fractional equation, and the squared initial-data error is $O(n^{-1})$.
\end{abstract}

\maketitle
\pagestyle{headings}

\section{Introduction}

\subsection{Problem and analytic motivation}

We study the fractional parabolic initial-value problem
\begin{equation}\label{eq: intro target PDE}
    \begin{cases}
        \partial_t v
        =
        \Delta^{\gamma/2}v
        +
        \paren{\bfb\cdot\nabla v}
        \bfone_{\set{\gamma>1}}
        +
        cv+f,
        & \text{in }(0,T)\times\mathbb R^d,
        \\
        v(0,\cdot)=v_0,
        & \text{in }\mathbb R^d,
    \end{cases}
\end{equation}
where $T,\gamma>0$, $d\geq1$, and $\bfb=(b_1,\ldots,b_d)$. Here, $\Delta^{\gamma/2}$ denotes the fractional Laplacian of order $\gamma$, defined by
$$
    \widehat{\Delta^{\gamma/2}v}(\xi)
    =
    -|\xi|^\gamma\widehat v(\xi).    
$$
The drift term is included only when $\gamma>1$, so that it is strictly lower order than the principal part.

The principal operator $\partial_t-\Delta^{\gamma/2}$ has space--time Fourier symbol $\mathrm{i}\tau+|\xi|^\gamma$. Thus one temporal derivative is naturally paired with $\gamma$ spatial derivatives, which motivates the additive Fourier weight $1+|\tau|^\alpha+|\xi|^{\gamma\alpha}$, with $\alpha\geq0$. The resulting spaces are anisotropic weighted Fourier--Lebesgue spaces with integrability exponent one. Following the terminology used in neural approximation, we refer to them as \emph{anisotropic spectral Barron spaces}. 
From the approximation viewpoint, the same Fourier integrability controls quantitative approximation by shallow-wide neural networks.

The central PDE question is whether the solution operator associated with \eqref{eq: intro target PDE} admits a maximal-regularity estimate in this anisotropic scale. For
$s\geq0$, we ask whether the following regularity implication holds:
$$
v_0\in
    \mathcal B^{\gamma+s}(\mathbb R^d),\qd
    b_1,\ldots,b_d, c,f
    \in
    \mathcal B^{s/\gamma,s}(T),
\qquad\Longrightarrow\qquad
    v
    \in
    \mathcal B^{1+s/\gamma,\gamma+s}(T).
$$
The indices express the full parabolic regularity gain:
$$
    \left(\frac{s}{\gamma},s\right)
    +(1,\gamma)
    =
    \left(1+\frac{s}{\gamma},\gamma+s\right).
$$
This is the weighted Fourier-$L^1$ counterpart of the regularity gain expected from parabolic Schauder or maximal-regularity theory.

A first difficulty is caused by the initial-time boundary. Although the solution is defined only for $t\geq0$, its finite-time norm is defined through global space--time extensions with sufficient temporal Fourier decay.
Since zero extension generally introduces a jump at $t=0$, we must construct an extension that preserves the required temporal regularity. A second difficulty comes from the lower-order terms. 
The products $\bfb\cdot\nabla v$ and $cv$ must be controlled in the source space without introducing additional dependence on the spatial dimension.
At the same time, the drift and potential are not assumed to be small, so their treatment cannot rely on a direct small-perturbation argument.

These PDE regularity estimates also provide a basis for quantitative neural approximation.
Once the exact solution is shown to belong to the anisotropic spectral Barron class, Fourier sampling can be used to construct shallow neural networks with quantitative mixed space--time Sobolev error bounds.
Neural-network methods for solving PDEs often minimize losses involving the PDE residual.
In the nonlocal fractional setting, however, Sobolev approximation on a bounded domain does not by itself ensure a small PDE residual there, since the whole-space fractional Laplacian also depends on values outside that domain.
At the population level, one may further ask whether the same construction can simultaneously control the solution error, the PDE residual, and the initial-data error.

\subsection{Main results}
\label{sec:main}

The central PDE result is Theorem~\ref{thm: intro PDE estimate}.
We first present two semigroup estimates used in its proof, then state the theorem and a counterexample concerning uniform-in-time spatial regularity.
The subsequent approximation results apply to general anisotropic spectral Barron functions and, in particular, to the PDE solutions.
A further sampling argument yields simultaneous control of the solution error, PDE residual, and initial-data error.

We first introduce the notation needed to state the results precisely.
Throughout, $C(a_1,\ldots,a_m)$ denotes a positive constant depending only on the indicated parameters, whose value may change from line to line.
We write $A\lesssim_{a_1,\ldots,a_m}B$ if $A\leq C(a_1,\ldots,a_m)B$.
For $N\geq1$, we use the Fourier transform
$$
\mathcal{F}_N[u](\zeta)
:=
\int_{\mathbb{R}^N}
\mathrm{e}^{-\mathrm{i} y\cdot\zeta}
u(y)\,\mathrm{d}y,
$$
extended to tempered distributions $\mathcal{S}'(\mathbb{R}^N)$ by duality.

For $r\geq0$, the spectral Barron space on $\mathbb{R}^N$ is
\begin{equation}
\label{eq:intro-spatial-barron}
\mathcal{B}^r(\mathbb{R}^N)
:=
\left\{
u\in\mathcal{S}'(\mathbb{R}^N):
\mathcal{F}_N[u]\in L^1_{\mathrm{loc}}(\mathbb{R}^N),\
\|u\|_{\mathcal{B}^r(\mathbb R^N)}<\infty
\right\},
\end{equation}
where
$$
\|u\|_{\mathcal{B}^r(\mathbb{R}^N)}
:=
\int_{\mathbb{R}^N}
(1+|\zeta|^r)
|\mathcal{F}_N[u](\zeta)|\,\mathrm{d}\zeta.
$$

For $\alpha,\beta\geq0$ and the space--time domain $\bR\times\bR^d=\bR^{d+1}$, the global anisotropic spectral Barron space is defined by
\begin{align}
\mathcal B^{\alpha,\beta}(\mathbb R^{d+1})
:=
\biggl\{
U\in\mathcal S'(\mathbb R^{d+1}):
\mathcal F_{d+1}[U]\in L^1_{\mathrm{loc}}(\mathbb R^{d+1}),
\|U\|_{\mathcal B^{\alpha,\beta}(\mathbb R^{d+1})}<\infty
\biggr\},
\label{eq:intro-anisotropic-barron}
\end{align}
with
$$
\|U\|_{\mathcal{B}^{\alpha,\beta}(\mathbb{R}^{d+1})}
:=
\int_{\mathbb{R}}\int_{\mathbb{R}^d}
\bigl(1+|\tau|^\alpha+|\xi|^\beta\bigr)
|\mathcal{F}_{d+1}[U](\tau,\xi)|
\,\mathrm{d}\xi\,\mathrm{d}\tau.
$$

The corresponding finite-time space is the restriction space
\begin{equation}
\label{eq:intro-anisotropic-restriction}
\mathcal{B}^{\alpha,\beta}(T)
:=
\left\{
U|_{[0,T]\times\mathbb{R}^d}:
U\in\mathcal{B}^{\alpha,\beta}(\mathbb{R}^{d+1})
\right\},
\end{equation}
equipped with
$$
\|u\|_{\mathcal{B}^{\alpha,\beta}(T)}
:=
\inf_{\substack{
U|_{[0,T]\times\mathbb R^d}=u
}}
\|U\|_{\mathcal{B}^{\alpha,\beta}(\mathbb{R}^{d+1})}.
$$
The structural properties of these spaces, including completeness, embeddings, interpolation, and multiplication estimates, are established in Section~\ref{sec:preliminaries}.

We now state the main results on anisotropic spectral Barron regularity, quantitative neural approximation, and population-level residual estimates.

\subsubsection{\textbf{Maximal regularity of fractional parabolic PDEs.}}
In preparation for the maximal-regularity theorem, we first record two estimates for the forward-in-time semigroup.
The reflection estimate in Theorem~\ref{lem:fourier_bound} controls extension across $t=0$, and applying it at each spatial frequency yields the homogeneous fractional heat estimate in Corollary~\ref{thm:homogeneous_solution}.

\begin{thm}[Vandermonde reflection estimate]
\label{lem:fourier_bound}
Let $\alpha,\beta\geq0$. For an integer $N>\alpha$, define the coefficients by
$$
(c_1,\ldots,c_N)^T
:=V_N^{-1}(1,-1,\ldots,(-1)^{N-1})^T,
\qquad
V_N:=(k^{j-1})_{j,k=1}^N,
$$
where $V_N$ is the Vandermonde matrix with nodes $1,\ldots,N$ (see Lemma~\ref{lem:reflection_coeffs} for the matching identities). Set
$$
h_N(t)
:=
\begin{cases}
\mathrm{e}^{-t},
& t\geq0,
\\[1mm]
\displaystyle
\sum_{k=1}^N c_k\mathrm{e}^{kt},
& t<0.
\end{cases}
$$
Then $h_N\in C^{N-1}(\mathbb{R})$ and
$$
|\mathcal F_1[h_N](\tau)|
\leq
C(N)(1+|\tau|)^{-N-1},
\qquad
\tau\in\mathbb{R}.
$$
Moreover, for every $\lambda>0$ and $\chi\in C_c^\infty(\mathbb{R})$,
\begin{equation}
\label{260410549}
\int_{\mathbb{R}}
\bigl(1+|\tau|^\alpha+\lambda^\beta\bigr)
|\mathcal{F}_1[\chi h_N(\lambda\cdot)](\tau)|
\,\mathrm{d}\tau
\leq
C(\alpha,N)
\|\chi\|_{\mathcal{B}^\alpha(\mathbb{R})}
\bigl(1+\lambda^\alpha+\lambda^\beta\bigr).
\end{equation}
\end{thm}

The homogeneous semigroup has the spatial Fourier representation
$$
\mathcal F_d[\mathrm e^{t(\Delta^{\gamma/2}-\lambda)}v_0](\xi)
=\mathrm e^{-t(\lambda+|\xi|^\gamma)}\mathcal F_d[v_0](\xi).
$$
Thus, choosing $N>s/\gamma$ and a cutoff $\chi\in C_c^\infty(\mathbb R)$ with $\chi=1$ on $[0,T]$, we obtain a global extension by setting
$$
\mathcal F_d[V(t,\cdot)](\xi)
:=\chi(t)h_N\bigl(t(\lambda+|\xi|^\gamma)\bigr)\mathcal F_d[v_0](\xi).
$$
Applying \eqref{260410549} with $\alpha=\beta=s/\gamma$ and the spectral parameter $\lambda+|\xi|^\gamma$, then integrating in $\xi$, gives the basic estimate in the following corollary. Its $\theta$-dependent form follows by interpolation.

\begin{corollary}[Homogeneous fractional heat evolution]
\label{thm:homogeneous_solution}
Let $d\geq1$, $T,\gamma,\lambda>0$, and $s\geq0$. 
For every $v_0\in\mathcal B^s(\mathbb{R}^d)$, the homogeneous solution $v(t)=\mathrm{e}^{t(\Delta^{\gamma/2}-\lambda)}v_0$ belongs to $\mathcal{B}^{s/\gamma,s}(T)$. 
More precisely, for every $\theta\in[0,1]$,
\begin{align}
\lambda^{(1-\theta)s/\gamma}
\|v\|_{\mathcal{B}^{\theta s/\gamma,\theta s}(T)}
\leq
C(T,s/\gamma,\theta)
\left(
\|v_0\|_{\mathcal{B}^s(\mathbb{R}^d)}
+
\lambda^{s/\gamma}
\|v_0\|_{\mathcal{B}^0(\mathbb{R}^d)}
\right).
\label{2607201155}
\end{align}
In particular,
\begin{align}
\|v\|_{\mathcal{B}^{s/\gamma,s}(T)}
+
\lambda^{s/\gamma}
\|v\|_{\mathcal{B}^{0,0}(T)}
\leq
C(T,s/\gamma)
\left(
\|v_0\|_{\mathcal{B}^s(\mathbb{R}^d)}
+
\lambda^{s/\gamma}
\|v_0\|_{\mathcal{B}^0(\mathbb{R}^d)}
\right).
\label{260510326}
\end{align}
\end{corollary}

The preceding corollary provides the homogeneous correction used to impose the initial condition in the inhomogeneous problem.
We now state the full PDE theorem, which also incorporates the drift and potential terms.

\begin{thm}[Anisotropic spectral Barron maximal regularity]
\label{thm: intro PDE estimate}
Let $d\geq1$, $T,\gamma>0$, and $s\geq0$. Assume that
$$
v_0\in\mathcal B^{\gamma+s}(\mathbb R^d),
\qquad
\bfb\in
\bigl(\mathcal B^{s/\gamma,s}(T)\bigr)^d,
\qquad
c,f\in\mathcal B^{s/\gamma,s}(T),
$$
where $\bfb=(b_1,\ldots,b_d)$. Then
\eqref{eq: intro target PDE}, with the drift term present only when
$\gamma>1$, has a unique solution $v\in\mathcal{B}^{1+s/\gamma,\gamma+s}(T)$.
Moreover,
\begin{equation}
\label{eq:main-regularity-estimate}
\|v\|_{\mathcal{B}^{1+s/\gamma,\gamma+s}(T)}
\leq
C(T,s,\gamma,\|\bfb\|,\|c\|)
\left(
\|v_0\|_{\mathcal{B}^{\gamma+s}(\mathbb{R}^d)}
+
\|f\|_{\mathcal{B}^{s/\gamma,s}(T)}
\right),
\end{equation}
where
$$
\|\bfb\|
:=
\left(
\sum_{k=1}^d
\|b_k\|_{\mathcal B^{s/\gamma,s}(T)}^2
\right)^{1/2},
\qquad
\|c\|
:=
\|c\|_{\mathcal B^{s/\gamma,s}(T)}.
$$
The constant $C$ in \eqref{eq:main-regularity-estimate} has no additional dependence on the spatial dimension.
More precisely, if $\norm{\bfb}\leq M_b$ and $\norm{c}\leq M_c$ uniformly in $d$, then the same constant $C(T,s,\gamma,M_b,M_c)$ works in every dimension.
For fixed $\bfb$ and $c$, the solution map $\mathrm{S}_{\bfb,c}:(v_0,f)\mapsto v$ therefore defines a bounded linear operator that maps $\mathcal B^{\gamma+s}(\mathbb R^d)
\times
\mathcal B^{s/\gamma,s}(T)$ to $
\mathcal B^{1+s/\gamma,\gamma+s}(T)$.
\end{thm}

The theorem derives the solution's Barron regularity from separate assumptions on the data and coefficients, without a smallness assumption on the lower-order terms.
The estimate is uniform in dimension whenever the stated coefficient norms are uniformly bounded.

The joint space--time norm in the theorem cannot in general be replaced by the corresponding uniform-in-time spatial norm.

\begin{prop}[Failure of the uniform-in-time spatial gain]
\label{prop:failure-uniform-time}
Let $T>0$. 
There is no constant $C(T)$ such that, for every $f\in C_c^\infty((0,T);\mathcal{S}(\mathbb R))$, the mild solution
\begin{equation}
\label{260713405}
u(t,x)
:=
\int_0^t
\mathrm{e}^{(t-r)\Delta}
f(r,\cdot)(x)\,\mathrm{d}r
\end{equation}
of
$$
\partial_tu=\Delta u+f,
\qquad
u(0,\cdot)=0,
$$
satisfies
\begin{equation}
\label{eq:false-endpoint-barron}
\|u\|_{L^\infty((0,T);\mathcal B^2(\mathbb{R}))}
\leq
C(T)
\|f\|_{L^\infty((0,T);\mathcal B^0(\mathbb{R}))}.
\end{equation}
\end{prop}

The counterexample uses forcing packets supported on disjoint time
intervals and separated spatial-frequency shells. At each time only one
packet is active, while at a later observation time the contributions of
all packets accumulate. Thus Proposition~\ref{prop:failure-uniform-time}
shows that a uniform-in-time $\mathcal B^0$ bound on the forcing alone does
not yield a uniform-in-time gain of two spatial Barron derivatives, as in
\eqref{eq:false-endpoint-barron}.
More generally, for every $s\geq0$ and $\gamma>0$, the analogous estimate
from $L^\infty((0,T);\mathcal B^s(\mathbb R))$ to
$L^\infty((0,T);\mathcal B^{s+\gamma}(\mathbb R))$ fails for the fractional
heat equation with zero initial data. The extension follows by adjusting
the time scale and applying a Bessel potential; see
Section~\ref{subsec:uniform-time-counterexample}.

\subsubsection{\textbf{Neural-network approximation.}}
We next state an approximation result for general functions in anisotropic spectral Barron spaces and then apply it to the PDE solution.

Let $T>0$, and let $\Omega\subset\mathbb{R}^d$ be a bounded Lipschitz domain.
For $\alpha,\beta\geq0$, set
\begin{equation}
\label{eq:intro-mixed-sobolev}
\mathcal{H}_{\alpha,\beta}(\Omega)
:=
H^\alpha((0,T);L^2(\Omega))
\cap
L^2((0,T);H^\beta(\Omega)),
\end{equation}
with
$$
\|u\|_{\mathcal{H}_{\alpha,\beta}(\Omega)}^2
:=
\|u\|_{H^\alpha((0,T);L^2(\Omega))}^2
+
\|u\|_{L^2((0,T);H^\beta(\Omega))}^2.
$$
Here the Sobolev spaces are understood as Bessel-potential restriction spaces; the precise convention is given in Section~\ref{sec: approximation}.

For an activation $\sigma:\mathbb R\to\mathbb R$ and $n\in\mathbb N$,
define
\begin{equation}
\label{eq:intro-network-class}
\Sigma_n^\sigma
:=
\left\{
\Phi(t,x)
=
\sum_{j=1}^n
a_j
\sigma(\omega_jt+\xi_j\cdot x+\theta_j):
a_j,\omega_j,\theta_j\in\mathbb R,\
\xi_j\in\mathbb R^d
\right\}.
\end{equation}

\begin{thm}[Approximation with periodic activations]
\label{thm:intro-periodic-approximation}
Let $T>0$ and $\alpha,\beta\geq0$.
Assume that $\sigma$ is a nonconstant periodic function in
$W^{\lceil\max\{\alpha,\beta\}\rceil,\infty}(\mathbb{R})$.
Then, for every real-valued $u\in\mathcal B^{\alpha,\beta}(T)$, every $n\in\mathbb{N}$, and every bounded Lipschitz $\Omega\subset\mathbb{R}^d$,  there exists $\Phi_n\in\Sigma_n^\sigma$ such that
\begin{equation}
\label{eq:intro-periodic-approximation}
\|u-\Phi_n\|_{\mathcal{H}_{\alpha,\beta}(\Omega)}
\leq
C(T,\Omega,\alpha,\beta,\sigma)
n^{-1/2}
\|u\|_{\mathcal B^{\alpha,\beta}(T)}.
\end{equation}
\end{thm}

For these periodic activations, the Sobolev error is controlled by the Barron norm with the same regularity indices, without an additional regularity shift.

In particular, for real-valued data and coefficients satisfying the assumptions
of Theorem~\ref{thm: intro PDE estimate}, let
$v=\mathrm S_{\bfb,c}(v_0,f)$ denote the corresponding solution.
Combining its regularity estimate with
Theorem~\ref{thm:intro-periodic-approximation} at $\alpha=1+s/\gamma$ and
$\beta=\gamma+s$ gives
\begin{equation}
\label{eq:intro-pde-network-approximation}
\begin{aligned}
&\inf_{\Phi\in\Sigma_n^\sigma}
\|v-\Phi\|_{\mathcal H_{1+\frac{s}{\gamma},\gamma+s}(\Omega)}\leq
C(T,\Omega,s,\gamma,\sigma,\|\bfb\|,\|c\|)n^{-1/2}
\left(
\|v_0\|_{\mathcal{B}^{\gamma+s}(\mathbb R^d)}
+
\|f\|_{\mathcal{B}^{s/\gamma,s}(T)}
\right).
\end{aligned}
\end{equation}

For nonperiodic activations, the same width exponent holds under an additional anisotropic regularity shift. 
More precisely, suppose that $\sigma$ satisfies Assumption~\ref{26.05.17.19.46} $(\left\lceil\max\{\alpha,\alpha\gamma\}\right\rceil)$.
Then, for every real-valued 
$$
u\in \mathcal{B}^{\alpha+\max\{1,\gamma^{-1}\},\alpha\gamma+\max\{1,\gamma\}}(T),
$$ 
there exists $\Phi_{\ell n}\in\Sigma_{\ell n}^{\sigma}$ such that
\begin{equation}
\label{eq:intro-nonperiodic-approximation}
\begin{aligned}
&\|u-\Phi_{\ell n}\|_{\mathcal H_{\alpha,\alpha\gamma}(\Omega)} \leq
C(d,T,\Omega,\alpha,\gamma,\sigma)n^{-1/2}
\|u\|_{
\mathcal B^{\alpha+\max\{1,\gamma^{-1}\},\alpha\gamma+\max\{1,\gamma\}}(T)
}.
\end{aligned}
\end{equation}
Here $\ell$ is the fixed number of affine copies appearing in Assumption~\ref{26.05.17.19.46}.

\subsubsection{\textbf{A population PINN consequence.}}
We finally extend the approximation result to simultaneous control of the solution error, PDE residual, and initial-data error for a single cosine network.
This requires a further Fourier-sampling argument in a product Hilbert space.

For $R>0$, define the data ball
\begin{equation}
\label{eq:pinn-data-ball}
\mathcal{D}_R
:=
\left\{
(v_0,f)\in \mathcal{B}^{\gamma+s}(\mathbb{R}^d)\times\mathcal{B}^{s/\gamma,s}(T):
\begin{aligned}
&
v_0,f\ \text{are real-valued},
\\
&
\|v_0\|_{\mathcal B^{\gamma+s}(\mathbb{R}^d)}
+
\|f\|_{\mathcal B^{s/\gamma,s}(T)}
\leq R
\end{aligned}
\right\}.
\end{equation}

We now consider approximation by networks in $\Sigma_n^{\cos}$. 
For fixed real-valued $\bfb\in
\bigl(\mathcal{B}^{s/\gamma,s}(T)\bigr)^d$, $c\in\mathcal{B}^{s/\gamma,s}(T)$, and $q\in[0,\gamma+s]$, define
\begin{align}
\mathcal J^{(q)}_{\Omega;v_0,f}(w)^2
&:=
\left\|
\partial_tw-\Delta^{\gamma/2}w
-
(\bfb\cdot\nabla w)\bfone_{\set{\gamma>1}}
-cw-f
\right\|_{L^2((0,T)\times\Omega)}^2+
\|w(0,\cdot)-v_0\|_{H^q(\Omega)}^2.
\label{eq:population-pinn-loss-q}
\end{align}
Here $\Delta^{\gamma/2}$ is the whole-space Fourier multiplier in \eqref{eq: intro target PDE}; it is applied before restriction to $(0,T)\times\Omega$.

\begin{corollary}
[Population PINN estimate with higher-order initial control]
\label{cor:pde-solution-class}
Let $T>0$, $q\in[0,\gamma+s]$, and $\Omega\subset\mathbb{R}^d$ be bounded and Lipschitz. 
Under the preceding assumptions, for every $R>0$ and $n\in\mathbb{N}$,
\begin{equation}
\label{eq:pde-solution-class}
\begin{aligned}
&\sup_{(v_0,f)\in\mathcal D_R}
\inf_{\Phi\in\Sigma_n^{\cos}}
\left\{
\|\mathrm{S}_{\bfb,c}(v_0,f)-\Phi\|_{\mathcal{H}_{1+\frac{s}{\gamma},\gamma+s}(\Omega)}^2
+
\mathcal J^{(q)}_{\Omega;v_0,f}(\Phi)^2
\right\}\leq
C(T,\Omega,d,s,\gamma,q,\|\bfb\|,\|c\|)n^{-1}R^2.
\end{aligned}
\end{equation}
\end{corollary}

Both the exact solution and the trial networks in Corollary~\ref{cor:pde-solution-class} are defined on $[0,T]\times\mathbb{R}^d$, while the population loss is observed only on $[0,T]\times\Omega$.
The infimum is taken separately for each $(v_0,f)\in\mathcal D_R$, so the approximating network may depend on the datum even though the bound is uniform over the data ball.
Thus, \eqref{eq:pde-solution-class} is a quantitative best-in-class population estimate.
It does not address operator learning, empirical generalization, optimization, or training convergence.

The remarks below explain the choice of initial norm, the role of the whole-space residual, and the dimension dependence.

\begin{rem}[Natural trace regularity and stronger initial control]
The choice $q=0$ gives the standard $L^2$ initial loss.
The anisotropic trace theorem
(see, \textit{e.g.}, \cite[Theorem~1.2]{agresti2023trace}) yields
$$
\mathcal{H}_{1+\frac{s}{\gamma},\gamma+s}(\Omega)
\hookrightarrow
C\left([0,T];H^{s+\gamma/2}(\Omega)
\right),
$$
so $q=s+\gamma/2$ is the natural trace-matched choice. The endpoint choice $q=\gamma+s$ is the strongest initial Sobolev control furnished by the present product-Hilbert sampling argument.
This stronger control is obtained by including the initial trace as a separate $H^{\gamma+s}(\Omega)$ component in the product Hilbert space and approximating it simultaneously with the solution and the PDE residual.
No optimality beyond this range is claimed.
\end{rem}

\begin{rem}[Simultaneous control of the nonlocal residual]
\label{rem:nonlocal-residual-not-automatic}
For nonlocal orders, the local approximation estimate~\eqref{eq:intro-pde-network-approximation} alone does not imply the residual bound in Corollary~\ref{cor:pde-solution-class}.
The fractional Laplacian depends on values outside $\Omega$, even though the residual norm is measured on $(0,T)\times\Omega$.
For example, if $0<\gamma<2$ and $0\ne\psi\in C_c^\infty(\mathbb R^d\setminus\overline\Omega)$ is nonnegative, then $\psi|_\Omega=0$ but
$$
\Delta^{\gamma/2}\psi(x)
=c_{d,\gamma}\int_{\mathbb R^d\setminus\Omega}
\frac{\psi(y)}{|x-y|^{d+\gamma}}\,\mathrm dy>0,
\qquad x\in\Omega.
$$
Thus adding $\eta(t)\psi(x)$, with $0\ne\eta\in C_c^\infty((0,T))$, to a trial function preserves its restriction to $(0,T)\times\Omega$ and its initial trace, while changing its PDE residual.

In Subsection~\ref{subsec:pde-solution-class}, we instead approximate the solution, its image under the PDE operator, and its initial trace simultaneously.
Writing
$$
\mathcal A_{\bfb,c}w
:=\partial_tw-\Delta^{\gamma/2}w
-(\bfb\cdot\nabla w)\bfone_{\set{\gamma>1}}-cw,
$$
we apply Hilbert-space sampling to the vectors $(Z,\mathcal A_{\bfb,c}Z,Z(0,\cdot))$ associated with global cosine atoms.
Lemma~\ref{lem:cosine-ridge-graph-norm} bounds all three components by the same Fourier weight.
The temporal extension agrees with $v$ on $[0,T]\times\mathbb R^d$, so its Fourier representation retains the spatial information needed by the nonlocal operator and yields the target vector $(v,f,v_0)$.
Using the same sampled atoms and coefficients in all three components gives a single network $\Phi_n$ with simultaneous bounds on the solution error, the residual $\mathcal A_{\bfb,c}\Phi_n-f$, and the initial-trace error, as stated in~\eqref{eq:pde-solution-class}.
\end{rem}

\begin{rem}[Width exponent and dimension dependence]
The exponent $1/2$ in the approximation rate $n^{-1/2}$ is independent of the spatial dimension $d$.
The approximation constants may nevertheless depend on $d$ and the geometry of $\Omega$.
Moreover, a width-$n$ network in $\Sigma_n^\sigma$ has $n(d+3)$ trainable parameters.
\end{rem}

\subsection{Relation to previous work}

\subsubsection{\textbf{Quantitative neural approximation and Barron spaces.}}
Classical universal approximation theorems establish density of shallow networks under broad assumptions on the activation \cite{cybenko1989approximation,hornik1989multilayer,leshno1993multilayer}, but do not quantify the network width required to achieve a prescribed error.
Barron's Fourier-sampling argument \cite{MR1237720} introduced a quantitative alternative: a finite weighted Fourier moment yields an $n^{-1/2}$ approximation rate for shallow ridge networks.

Subsequent works developed representation-space formulations, direct and inverse theorems, and Sobolev approximation estimates for shallow networks \cite{MR4376565,e2020banach,e2022representation,MR4538591,
siegel2020approximation,MR4357282}. Related studies established lower bounds, improved rates under additional smoothness, and spectral or logarithmic refinements \cite{NEURIPS2022_8caa10fb,MR4931876,lu2025approximation,song2026sobolev,MR4734714,yarotsky2017error,zhao2026logarithmic}. 
See also \cite{LuZhang2026generalized} for generalized Barron spaces associated with general activation functions and \cite{AbdeljawadCarioniCordero2026metaplectic} for approximation using metaplectic dictionaries.
The present paper uses the Fourier-side, or spectral, Barron viewpoint: the main question is not merely how to approximate a given Barron function, but whether the PDE solution belongs to the required approximation class.

\subsubsection{\textbf{Weighted Fourier--Lebesgue spaces.}}
These spaces have a well-developed theory in harmonic and microlocal
analysis. Sharp convolution and multiplication estimates in weighted Lebesgue, Fourier--Lebesgue, and modulation spaces were established in \cite{toft2015sharp}, while pseudodifferential continuity and microlocal regularity in quasi-homogeneous Fourier--Lebesgue spaces were studied in \cite{garello2020microlocal}. 
The latter framework also reflects quasi-homogeneous anisotropy arising in PDEs.

Our spaces form a scale of restriction spaces based on weighted Fourier $L^1$ norms with the additive parabolic weight $1+|\tau|^\alpha+|\xi|^\beta$.
The contribution of the present work is the development of a finite-time parabolic theory in this scale:
a parameter-uniform extension of one-sided semigroup trajectories, product and vector-field estimates with no additional dimension factor when the coefficient norm is controlled componentwise in $\ell^2$, and a solution-operator mapping theorem for fractional parabolic equations.

\subsubsection{\textbf{Space--time neural approximation.}}
Abdeljawad and Grohs \cite{abdeljawad2022sobolev} prove deep-network approximation results in Bochner--Sobolev space--time norms using ReCU activations. 
Abdeljawad and Dittrich \cite{abdeljawad2023spacetime} provide the closest precedent in approximation theory.
They derive $n^{-1/2}$ shallow-network approximation bounds in anisotropic Bochner-Sobolev norms for targets in suitably weighted anisotropic Fourier--Lebesgue spaces. 
Their framework includes architectures with separate time and space blocks and assumes membership in a strengthened target class determined by the weight and activation. 
A broader time--frequency approach based primarily on weighted modulation spaces, with consequences for Fourier--Lebesgue and Barron classes, was developed in \cite{abdeljawad2026timefrequency}.

These works begin with a prescribed target class and study its neural approximation. 
The present paper additionally addresses the preceding PDE regularity problem: we prove that fractional parabolic data and lower-order coefficients generate a solution in the required anisotropic Fourier--Lebesgue class. 
Moreover, for smooth nonconstant periodic scalar activations, our approximation theorem uses the natural additive $\mathcal{B}^{\alpha,\beta}(T)$ norm without an additional Barron regularity shift. 
Combining the two steps yields an $n^{-1/2}$ shallow-network approximation rate for the PDE solution in mixed space--time Sobolev norms.

\subsubsection{\textbf{Extension operators and matching constructions.}}
Seeley's classical half-space construction \cite{seeley1964extension} and the recent Seeley-type universal extension of Lu and Yao \cite{lu2024seeley} use reflected copies satisfying algebraic moment conditions in order to match boundary jets and preserve regularity. 
Rychkov's universal extension operator \cite{rychkov1999restrictions} provides a different construction on Lipschitz domains, based on Littlewood--Paley decompositions rather than reflected copies.

Our construction is not intended as a general extension theorem. 
It is a modewise extension tailored to the one-sided semigroup trajectory $t\longmapsto \mathrm{e}^{-t\lambda}$.
We use finitely many reflected semigroup pieces and match their derivatives at $t=0$ only up to the order required by the temporal Fourier weight.
The feature specific to our argument is the resulting weighted Fourier-$L^1$ estimate with a constant uniform in the spectral parameter $\lambda$. 
This uniformity permits the frequency-wise substitution $ \lambda\longmapsto\lambda+|\xi|^\gamma$ for the fractional parabolic evolution.

\subsubsection{\textbf{Maximal regularity and uniform-in-time control.}}
Classical maximal $L^p$-regularity for $1<p<\infty$ is closely connected with operator-valued Fourier multipliers and $R$-boundedness \cite{weis2001multiplier}. 
At the level of continuous-in-time maximal regularity, Baillon-type results reveal strong rigidity; see, for example, the refinement in \cite{jacob2022refinement}.

The present framework is not a direct endpoint case of either theory.
Here the exponent one refers to integrability of the joint space--time Fourier transform, rather than to an $L^1((0,T);X)$ maximal-regularity norm. 
Within this distinct endpoint regime, Proposition~\ref{prop:failure-uniform-time} provides a concrete frequency-localized obstruction to the specific estimate $L^\infty((0,T);\mathcal{B}^0(\mathbb{R}))
\rightarrow
L^\infty((0,T);\mathcal{B}^2(\mathbb{R}))$ for the one-dimensional heat equation. The construction shows that a pointwise-in-time spatial norm can miss the accumulation of frequency packets generated at different times.

\subsubsection{\textbf{Barron regularity of PDE solutions.}}
For static equations, Barron representation estimates have been obtained for linear elliptic problems \cite{chen2021representation}, while spectral Barron regularity theories have been developed for Schr\"odinger, fractional, variational, and Hamilton--Jacobi--Bellman equations \cite{chen2026regularity,MR4551470,choulli2025functional,MR5033018,MR4993598,marwah2023neural}. 
Lu, Lu, and Wang \cite{LuLuWang2021DeepRitz} establish spectral Barron regularity from assumptions on the forcing and potential for elliptic problems on a hypercube and use it to derive a priori generalization bounds for the Deep Ritz method.
See also \cite{Wojtowytsch2026elliptic} for Dirichlet problems in ReLU Barron spaces, including obstructions to Barron regularity and applications to the Deep Ritz method.

In Barron-type representation spaces, E and Wojtowytsch \cite{e2022observations} use explicit whole-space representation formulas to study the screened Poisson equation, heat evolution, and a viscous Hamilton--Jacobi equation.

In concurrent work, Choulli, Lu, and Takase \cite{choulli2026heat} develop a sectorial-operator theory for heat equations on the spatial spectral Barron space $\mathcal{B}^0(\mathbb{R}^d)$. 
They treat linear and local semilinear heat equations, backward-heat stability, and time-fractional evolution, and extend parts of the theory to variable-coefficient elliptic operators under perturbative assumptions. 
Their time regularity is formulated through abstract evolution and interpolation spaces over a spatial Barron space. 
In contrast, our regularity is measured by a joint anisotropic space--time Fourier restriction norm.

The two approaches are therefore complementary. 
The present work treats a spatially fractional principal operator, finite-time anisotropic regularity, and lower-order drift and potential terms without smallness assumptions. 
It also proves the failure of the corresponding uniform-in-time spatial gain and derives shallow-network approximation bounds in mixed Sobolev norms. 
The constant in the PDE a priori estimate has no additional dependence on $d$, provided that the componentwise
$\ell^2$-norm of the drift and the prescribed norm of the potential are controlled uniformly in $d$.

In related work, Ji and Wang \cite{JiWang2026finitehorizon} establish short-time well-posedness for finite-horizon Hamilton--Jacobi--Bellman equations with state-dependent diffusion in augmented spectral Barron spaces.
Their joint space--time cosine-network approximation combines temporal Jackson approximation with spatial Barron approximation, using separate temporal and spatial regularity estimates.

\subsubsection{\textbf{High-dimensional neural PDE approximation.}}
A separate line of work studies neural networks as numerical or approximation tools for specified high-dimensional PDE classes.
E, Han, and Jentzen \cite{e2017deep} introduced a deep-BSDE method and reported numerical results for nonlinear parabolic equations in high dimensions. 
Hutzenthaler, Jentzen, Kruse, and Nguyen \cite{hutzenthaler2020proof} proved polynomial parameter-complexity bounds in $d$ and $\varepsilon^{-1}$ for a class of semilinear heat equations with gradient-independent nonlinearities. 
Hornung, Jentzen, and Salimova \cite{hornung2025space} established space--time deep-network approximation results for certain high-dimensional Kolmogorov equations.

These results address constructive algorithms or architecture-specific complexity for prescribed equation classes. 
The present work instead identifies intrinsic function-space regularity of the exact solution and then derives a shallow-network approximation rate from that regularity. 
The two viewpoints are complementary: one begins with a solver or approximation architecture, whereas the other begins with a PDE mapping property in a neural-network-adapted function space.

\subsubsection{\textbf{PINN error analysis, Barron-based approaches,
and fractional models.}}
The modern PINN framework popularized by Raissi, Perdikaris, and Karniadakis \cite{RaissiPerdikarisKarniadakis2019} trains neural networks by minimizing losses that combine PDE residuals with initial, boundary, and, when available, observational data. 
Existing analyses address consistency of empirical minimizers \cite{ShinDarbonKarniadakis2020}, residual-to-solution stability and empirical-to-population estimates \cite{DeRyckMishra2022,MishraMolinaro2023}, abstract a priori and a posteriori error bounds based on coercivity \cite{ZeinhoferMasriMardal2025}, and statistical consistency under regularization \cite{DoumecheBiauBoyer2025}.

Barron-type complexity has also been used explicitly in PINN analysis.
Luo and Yang \cite{LuoYang2024} study least-squares neural solvers for second-order linear PDEs, including optimization of overparameterized two-layer networks and generalization estimates under Barron-type regularization. 
Hu, Jagtap, Karniadakis, and Kawaguchi \cite{HuJagtapKarniadakisKawaguchi2022} introduce generalized Barron and tree-like spaces and derive a priori and a posteriori generalization bounds for PINNs and XPINNs. 
Xu, Li, and Huang \cite{XuLiHuang2025} obtain refined generalization bounds for PINNs applied to linear second-order elliptic equations and establish controlled Sobolev approximation results for Barron targets.

The closest comparison to our population estimate is the SSBE-PINN framework of Zhou, Chen, Luo, and Xiang \cite{ZhouChenLuoXiang2025}. 
They introduce Sobolev-stable boundary losses and operator-dependent Barron pairs and triplets for elliptic and parabolic equations. 
In the parabolic setting, the interior forcing, boundary data, and initial data are generated through a common ridge representation. 
Under these assumptions, their a priori generalization bounds contain an $O(m^{-1})$ population-approximation term together with an empirical-to-population sampling term.

A separate line of work concerns fractional and nonlocal equations. 
Pang, Lu, and Karniadakis \cite{PangLuKarniadakis2019} introduce fPINNs whose residual combines automatic differentiation for integer-order terms with numerical discretization of fractional operators.
More recently, Abdo, Chai, Hu, and Yang \cite{AbdoChaiHuYang2026} establish convergence for time-dependent spectral fractional diffusion equations on bounded domains by introducing a spectrally defined mollification compatible with the nonlocal operator and its boundary conditions.
Their setting involves bounded-domain spectral fractional operators, whereas our residual is formed using the whole-space fractional Fourier multiplier and is then measured on a bounded observation cylinder.

\medskip
\noindent\textit{Relation to the present result.}
Against this background, the distinctive contribution of our result lies in the passage from assumptions on the coefficients, initial condition, and source term to quantitative neural approximation bounds, via anisotropic spectral Barron regularity of the solution.
From separate anisotropic spectral Barron bounds on the initial datum, forcing, drift, and potential, we first derive the anisotropic spectral Barron regularity of the exact solution through a fractional parabolic maximal-regularity theorem. 
This regularity is not assumed a priori and is not encoded through an operator-dependent common ridge representation.

For each admissible pair $(v_0,f)$ and each fixed $q\in[0,\gamma+s]$, we then obtain a quantitative approximation bound for the corresponding solution using shallow cosine networks. 
More precisely, there exists a single width-$n$ network that simultaneously controls the mixed-Sobolev solution error, the population residual of the whole-space spatially fractional operator on the observation cylinder, and the initial-trace error in $H^q(\Omega)$. 
The corresponding squared combined error is bounded by $C(T,\Omega,d,s,\gamma,q,\|\bfb\|,\|c\|)R^2/n$ uniformly over the prescribed Fourier--Barron data ball.

In the nonlocal case, obtaining this simultaneous estimate requires global spatial information that the local mixed-Sobolev approximation norm alone does not provide; see Remark~\ref{rem:nonlocal-residual-not-automatic}.

To our knowledge, the present result is the first to derive this simultaneous graph-space estimate from a joint anisotropic spectral Barron maximal-regularity theorem for a spatially fractional parabolic equation with lower-order drift and potential terms.
The result is nevertheless existential and population-level. 
The approximating network may depend on the datum, and the estimate does not provide a residual-to-error certificate for an arbitrary trial or trained network. 
It also does not address empirical-to-population generalization or optimization dynamics.

\subsection{Proof strategy}
The analytic argument has three main steps. 
First, we construct a finite reflected extension of the scalar semigroup trajectory $t\longmapsto \mathrm{e}^{-t\lambda}$, for $t\geq0$.
The reflection coefficients solve a Vandermonde system chosen so that the derivatives match at $t=0$ up to a prescribed order. 
The resulting cancellations yield a temporal Fourier estimate that is uniform in $\lambda$. 
Substituting $\lambda+|\xi|^\gamma$ for the scalar parameter then gives a controlled global extension of the homogeneous fractional heat evolution.

Second, we solve the damped principal equation globally by the Fourier multiplier $ \frac{1}{\lambda+\mathrm{i}\tau+|\xi|^\gamma}$.
After restricting the global solution to the finite time interval, we remove its unwanted initial trace and insert the prescribed initial datum by means of the reflected homogeneous evolution. 
This gives the full gain of one derivative in time and $\gamma$ derivatives in space for the principal fractional parabolic problem.

Third, we incorporate the drift and potential terms using dimension-controlled product and vector-field estimates, anisotropic interpolation, a finite regularity bootstrap, and a sufficiently large damping parameter. 
These estimates give a uniform a priori bound along a continuity path, allowing the method of continuity to complete the proof without smallness assumptions on the lower-order coefficients.
The failure of the uniform-in-time spatial estimate is proved separately by placing forcing packets on disjoint time intervals and separated frequency shells; the source norm remains bounded at every time, whereas all packets accumulate in the solution at a later observation time.

For neural approximation, we first estimate individual ridge atoms in the relevant temporal and spatial Sobolev norms and obtain fractional orders by Hilbert-valued interpolation. 
Fourier inversion then represents the target as an expectation of ridge atoms, and a Hilbert-space sampling lemma gives the $n^{-1/2}$ approximation rate. 
For periodic activations, a nonzero Fourier-series coefficient provides the required ridge representation without an additional Barron regularity shift.

For the population PINN consequence, we repeat the Fourier-sampling argument in the product Hilbert space
$$
\mathcal{H}_{1+s/\gamma,\gamma+s}(\Omega)
\oplus
L^2\bigl((0,T)\times\Omega\bigr)
\oplus
H^q(\Omega),
\qquad
0\leq q\leq\gamma+s.
$$
The three coordinates correspond to the solution, its image under the PDE operator, and its initial trace; differences in these coordinates measure the three errors.

Starting from a global anisotropic Fourier representation of the exact solution, we associate with each sampled cosine atom its image in all three coordinates. 
In the nonlocal case, retaining this global representation allows us to control the residual on $\Omega$ despite its dependence on values outside $\Omega$.
The anisotropic Fourier weight controls the corresponding product-space norm uniformly over the sampling parameters. 
The Hilbert-space sampling lemma then produces a single width-$n$ cosine network for which the three errors are simultaneously bounded at the $n^{-1/2}$ scale. 
Equivalently, the squared combined graph error is of order $n^{-1}$.

\subsection{Organization}
Section~\ref{sec:preliminaries} develops the additive-weight Fourier--Barron calculus, including the finite-time restriction spaces, embeddings, interpolation inequalities, multiplication estimates, and the vector-field estimate.

Section~\ref{sec:semigroup-extension} constructs the finite reflected extension, proves the parameter-uniform temporal Fourier estimate of Theorem~\ref{lem:fourier_bound}, and applies it to homogeneous fractional heat trajectories to obtain Corollary~\ref{thm:homogeneous_solution}.

Section~\ref{sec:regularity} establishes the global multiplier estimate for the principal operator, solves the finite-time inhomogeneous problem with prescribed initial data, and incorporates the lower-order drift and potential terms to prove Theorem~\ref{thm: intro PDE estimate}.
It also constructs the counterexample to the corresponding uniform-in-time spatial estimate in Proposition~\ref{prop:failure-uniform-time}.

Section~\ref{sec: approximation} proves the shallow-network approximation
theorems in mixed space--time Sobolev norms. It also establishes
Corollary~\ref{cor:pde-solution-class} by a separate product-Hilbert-space
sampling argument.

\section{Anisotropic Fourier--Barron Spaces and Dimension-Controlled Calculus}
\label{sec:preliminaries}
The spaces used in this paper are additive-weight Fourier--Lebesgue spaces with integrability exponent one.
We use the term Fourier--Barron to emphasize their role in shallow-network approximation and retain the notation $\mathcal B$ used in the spectral Barron literature.

The Fourier--Barron spaces, their global norms, and their finite-time restriction norms were defined in \eqref{eq:intro-spatial-barron}--\eqref{eq:intro-anisotropic-restriction}.
We record here only the additional Fourier conventions and structural
estimates used below.\medskip

\paragraph{\textbf{Fourier-analytic conventions.}}
For the forward transform fixed in Subsection~\ref{sec:main}, we use the
inverse transform
$$
\cF_N^{-1}[v](x)
:=
(2\pi)^{-N}
\int_{\bR^N}\mathrm e^{\mathrm i x\cdot\xi}v(\xi)\,\dd\xi.
$$
With these conventions, we have
    $\cF_N[\partial_{x_j}v](\xi)=\mathrm{i}\xi_j\cF_N[v](\xi)$,
    for $j\in\{1,\ldots,N\}$.
Moreover, whenever the operations are well-defined,
$$
\cF_N[v\ast w]=\cF_N[v]\cF_N[w],
    \qquad
    \cF_N[vw]=(2\pi)^{-N}\cF_N[v]\ast\cF_N[w].
$$

\paragraph{\textbf{Tempered distributions and inversion.}}
The forward transform fixed in Subsection~\ref{sec:main} and the inverse
transform above extend to the space of tempered distributions $\cS'(\bR^N)$
by duality. More precisely, for \(v\in \cS'(\bR^N)\), we define
\(\cF_N[v]\in \cS'(\bR^N)\) and \(\cF_N^{-1}[v]\in \cS'(\bR^N)\) by
\[
    \langle \cF_N[v],\varphi\rangle
    :=
    \langle v,\cF_N[\varphi]\rangle,
    \qquad \langle \cF_N^{-1}[v],\varphi\rangle
    :=
    \langle v,\cF_N^{-1}[\varphi]\rangle,
    \qquad \varphi\in\cS(\bR^N).
\]
Moreover, if \(v\in\cS'(\bR^N)\) and \(\cF_N[v]\) has an
\(L^1\)-representative, then $v$ is recovered as
\(v=\cF_N^{-1}[\cF_N[v]]\) pointwise. Throughout the paper, we identify tempered distributions with their function representatives whenever they exist.

\begin{rem}
\label{26.02.22.15.15}
The space--time Barron norm controls the spatial Barron norm at each fixed time. In particular, let \(v\in \cB^{\alpha,\beta}(T)\) and
\(V\in \cB^{\alpha,\beta}(\bR^{d+1})\) be any extension of \(v\). For every \(t\in[0,T]\),
\begin{align*}
    \|v(t,\cdot)\|_{\cB^\beta(\bR^d)}
    &=
    \int_{\bR^d}
    (1+|\xi|^\beta)
    |\cF_d[V(t,\cdot)](\xi)|
    \dd\xi \\
    &\leq
    \frac{1}{2\pi}
    \int_{\bR^d}
    (1+|\xi|^\beta)
    \int_{\bR}
    |\cF_{d+1}[V](\tau,\xi)|
    \dd\tau\dd\xi \leq
    \frac{1}{2\pi}
    \|V\|_{\cB^{\alpha,\beta}(\bR^{d+1})}.
\end{align*}
Taking the infimum over all such extensions \(V\), we obtain $
    \|v(t,\cdot)\|_{\cB^\beta(\bR^d)}
    \leq
    \frac{\|v\|_{\cB^{\alpha,\beta}(T)}}{2\pi}
    $.
\end{rem}

\subsection{Banach spaces}
\label{26.07.22.11.24}
We next verify that our anisotropic Barron spaces are Banach spaces, using an isomorphism with a weighted $L^1$ space and a quotient-space argument.
\begin{prop}
\label{26.06.19.14.00}
Let \(\alpha,\beta\geq 0\) and \(T\in(0,\infty)\).
Then \(\cB^{\alpha,\beta}(\mathbb{R}^{d+1})\) and \(\cB^{\alpha,\beta}(T)\) are Banach spaces.
\end{prop}

\begin{proof}
By definition, the Fourier transform $\cF_{d+1}$ is an isometric isomorphism from $\cB^{\alpha,\beta}(\bR^{d+1})$ onto the weighted space $L^1(\bR^{d+1},(1+|\tau|^{\alpha}+|\xi|^{\beta})\,\dd\tau\dd\xi)$.
Indeed, if $g\in L^1(\bR^{d+1},(1+|\tau|^{\alpha}+|\xi|^{\beta})\,\dd\tau\dd\xi)$, then $g\in L^1(\bR^{d+1})$ and $\cF_{d+1}^{-1}[g]\in \cB^{\alpha,\beta}(\bR^{d+1})$. 
Hence $\cB^{\alpha,\beta}(\bR^{d+1})$ is a Banach space.

Let $R:\cB^{\alpha,\beta}(\bR^{d+1})\to\mathcal B^{\alp,\bt}(T)$ be the restriction operator defined by $RV:=V|_{[0,T]\times\bR^d}$. 
By the definition of $\cB^{\alpha,\beta}(T)$, the map $R$ is surjective.
By Fourier inversion,
$$
    \|RV\|_{L^\infty([0,T]\times\bR^d)}
    \leq
    \frac{1}{(2\pi)^{d+1}}
    \|\cF_{d+1}[V]\|_{L^1(\bR^{d+1})}
    \leq
    \frac{1}{(2\pi)^{d+1}}\|V\|_{\cB^{\alpha,\beta}(\bR^{d+1})}.
$$
Thus $R$ is bounded, and hence $N:=\ker R$ is a closed subspace of $\cB^{\alpha,\beta}(\bR^{d+1})$.
Therefore, by the definition of the quotient norm \cite[Ch.~III, \S4, (4.1)]{ConwayFA},
$$
\begin{aligned}
    \|RV\|_{\cB^{\alpha,\beta}(T)}
    =
    \inf\bigl\{\|U\|_{\cB^{\alpha,\beta}(\bR^{d+1})}:RU=RV\bigr\} =
    \inf_{W\in N}\|V+W\|_{\cB^{\alpha,\beta}(\bR^{d+1})}
    =
    \|V+N\|_{\cB^{\alpha,\beta}(\bR^{d+1})/N}.
\end{aligned}
$$
Consequently, the induced map $\widetilde R:\cB^{\alpha,\beta}(\bR^{d+1})/N\longrightarrow \cB^{\alpha,\beta}(T)$ defined by $\widetilde R(V+N):=RV$ is an isometric isomorphism. 
Since the quotient of a Banach space by a closed subspace is a Banach space \cite[Ch.~III, \S4, Thm.~4.2(b)]{ConwayFA}, $\cB^{\alpha,\beta}(T)$ is a Banach space.
\end{proof}

\subsection{Embedding relations}
\label{26.07.22.11.25}
We next establish embeddings from higher to lower regularity within the same scale, with constants independent of the regularity indices. We then prove embeddings into classical H\"older spaces. 

\begin{prop}\label{lem:self-embedding}
Let \(T\in(0,\infty)\), $0\leq \alpha_1\leq \alpha_2<\infty$, and $0\leq \beta_1\leq \beta_2<\infty$.
Then
$$
\cB^{\alp_2}(\bR^{d})
    \embed
    \cB^{\alp_1}(\bR^{d}),\quad \cB^{\alp_2,\bt_2}(\bR^{d+1})
    \embed
    \cB^{\alp_1,\bt_1}(\bR^{d+1}),\quad 
    \cB^{\alp_2,\bt_2}(T)
    \embed
    \cB^{\alp_1,\bt_1}(T).
$$
More precisely,
\begin{align}\label{2607081031}
\begin{alignedat}{2}
    \nrm{f}_{\cB^{\alp_1}(\bR^{d})}
    &\leq
    2\nrm{f}_{\cB^{\alp_2}(\bR^{d})}
    \qd&&
    \text{for every }f\in\cB^{\alp_2}(\bR^{d}),\\
    \nrm{V}_{\cB^{\alp_1,\bt_1}(\bR^{d+1})}
    &\leq
    3\nrm{V}_{\cB^{\alp_2,\bt_2}(\bR^{d+1})}
    \qd&&
    \text{for every }V\in\cB^{\alp_2,\bt_2}(\bR^{d+1}),\\
    \nrm{v}_{\cB^{\alp_1,\bt_1}(T)}
    &\leq
    3\nrm{v}_{\cB^{\alp_2,\bt_2}(T)}
    \qd&&
    \text{for every }v\in\cB^{\alp_2,\bt_2}(T).
\end{alignedat}
\end{align}
\end{prop}
\begin{proof}
   For \(0\leq A_1\leq A_2\) and \(a\geq0\), we have \(a^{A_1}\leq 1+a^{A_2}\). 
   Hence \(1+|\xi|^{\alp_1}\leq 2(1+|\xi|^{\alp_2})\), which gives the first inequality in \eqref{2607081031}. 
   Similarly, \(1+|\tau|^{\alp_1}+|\xi|^{\bt_1}\leq 3(1+|\tau|^{\alp_2}+|\xi|^{\bt_2})\), and therefore the second inequality in \eqref{2607081031} follows directly from the definition of the \(\cB^{\alp,\bt}(\bR^{d+1})\)-norm.

It remains to prove the assertion on \((0,T)\). 
Let \(v\in\cB^{\alp_2,\bt_2}(T)\), and let \(V\in\cB^{\alp_2,\bt_2}(\bR^{d+1})\) be any extension of \(v\). 
By the second inequality and the definition of \(\|v\|_{\cB^{\alp_1,\bt_1}(T)}\), we have $\|v\|_{\cB^{\alp_1,\bt_1}(T)}\leq \|V\|_{\cB^{\alp_1,\bt_1}(\bR^{d+1})}\leq 3\|V\|_{\cB^{\alp_2,\bt_2}(\bR^{d+1})}$.
Taking the infimum over all such \(V\), we obtain the third inequality in \eqref{2607081031}.
\end{proof}

The weighted Fourier integrability also implies classical regularity of the corresponding functions. In what follows, we make this relation precise in an anisotropic Hölder scale.  For \(E\subset\bR^{d+1}\) and \(a,b\in[0,1]\), define
\[
[h]_{a,b;E}
:=
\begin{cases}
\displaystyle
\sup_{\substack{(t,x),(t',x')\in E\\ (t,x)\neq(t',x')}}
\frac{|h(t,x)-h(t',x')|}
{|t-t'|^a+|x-x'|^b},
& a>0,\ b>0,\\[2ex]
\displaystyle
\sup_{\substack{(t,x),(t',x)\in E\\ t\neq t'}}
\frac{|h(t,x)-h(t',x)|}{|t-t'|^a},
& a>0,\ b=0,\\[2ex]
\displaystyle
\sup_{\substack{(t,x),(t,x')\in E\\ x\neq x'}}
\frac{|h(t,x)-h(t,x')|}{|x-x'|^b},
& a=0,\ b>0,\\[2ex]
\qd0,
& a=b=0.
\end{cases}
\]

Below, we use the convention that \(0/0=0\) and
\(a/0=+\infty\) for \(a>0\). For \(\alpha,\beta\geq0\), define the admissible set of derivatives by
\[
    \mathcal D_{\alpha,\beta}
    :=
    \left\{
    (i,\eta)\in\bN_0\times\bN_0^d:
    \frac{i}{\alpha}+\frac{|\eta|}{\beta}\leq1
    \right\},
\]
where $|\eta|$ is the order of the multi-index $\eta$.
For \((i,\eta)\in\mathcal D_{\alpha,\beta}\), let
\begin{alignat*}{2}
    \sigma_{i,\eta}
    \,:=
    \sup\left\{
    s\in[0,1]:
    \frac{i+s}{\alpha}+\frac{|\eta|}{\beta}\leq1
    \right\},\,\rho_{i,\eta}
    \,:=
    \sup\left\{
    r\in[0,1]:
    \frac{i}{\alpha}+\frac{|\eta|+r}{\beta}\leq1
    \right\}.
\end{alignat*}
Thus \(\sigma_{i,\eta}\) and \(\rho_{i,\eta}\) are the remaining time and
space regularities, respectively, after taking the derivative
\(\partial_t^i\partial_x^\eta\).

We define \(C^{\alpha,\beta}_{t,x}(E)\) to be the space of all functions
\(f\in C(E)\) such that, for every \((i,\eta)\in\mathcal D_{\alpha,\beta}\),
the derivative \(\partial_t^i\partial_x^\eta f\) exists classically in the
interior of \(E\) and extends continuously to \(E\), and
\[
\|f\|_{C^{\alpha,\beta}_{t,x}(E)}
:=
\sum_{(i,\eta)\in\mathcal D_{\alpha,\beta}}
\|\partial_t^i\partial_x^\eta f\|_{L^\infty(E)}
+
\sum_{(i,\eta)\in\mathcal D_{\alpha,\beta}}
[\partial_t^i\partial_x^\eta f]_{\sigma_{i,\eta},\rho_{i,\eta};E}
<\infty .
\]

\begin{prop}\label{prop:B-to-Ctx}
Let \(T\in(0,\infty)\) and \(\alpha,\beta\geq0\). Then
\[
    \cB^{\alpha,\beta}(\bR^{d+1})
    \embed
    C^{\alpha,\beta}_{t,x}(\bR^{d+1}),
    \qquad
    \cB^{\alpha,\beta}(T)
    \embed
    C^{\alpha,\beta}_{t,x}([0,T]\times\bR^d).
\]
More precisely, for every
\(f\in \cB^{\alpha,\beta}(\bR^{d+1})\) and
\(v\in \cB^{\alpha,\beta}(T)\), we have
\begin{align*}
&\|f\|_{C^{\alpha,\beta}_{t,x}(\bR^{d+1})}
\leq
C(d,\alpha,\beta)\|f\|_{\cB^{\alpha,\beta}(\bR^{d+1})},\\ &\|v\|_{C^{\alpha,\beta}_{t,x}([0,T]\times\bR^d)}
\leq
C(d,\alpha,\beta)\|v\|_{\cB^{\alpha,\beta}(T)}.
\end{align*}
\end{prop}

\begin{proof}
We first prove the estimate on \(\bR^{d+1}\). We use the convention that
\(0/0=0\) and \(a/0=+\infty\) for \(a>0\). We shall repeatedly use the
following elementary consequence of Young's inequality: if \(p,q\geq0\) and
    $\frac{p}{\alpha}+\frac{q}{\beta}\leq1$,
then $(1+|\tau|)^p(1+|\xi|)^q
    \lesssim_{\alpha,\beta}
    1+|\tau|^\alpha+|\xi|^\beta$.
Indeed, when \(\alpha,\beta>0\), this follows from Young's inequality. If
\(\alpha=0\) or \(\beta=0\), the convention forces the corresponding exponent
to be zero.

Fix \((k,\eta)\in\mathcal D_{\alpha,\beta}\). Since $\frac{k}{\alpha}+\frac{|\eta|}{\beta}\leq1$,
the preceding bound gives
\[
    \int_{\bR^{d+1}}
    (1+|\tau|)^k(1+|\xi|)^{|\eta|}
    |\cF_{d+1}f(\tau,\xi)|\,\dd\xi\dd\tau
    \lesssim_{d,\alpha,\beta}
    \|f\|_{\cB^{\alpha,\beta}(\bR^{d+1})}.
\]
The same estimate holds with \(k,\eta\) replaced by any lower order \(\ell,\mu\), where \(0\leq \ell\leq k\) and \(\mu\leq\eta\). 
Hence, the Fourier inversion formula may be differentiated under the integral sign, and \(\partial_t^k\partial_x^\eta f\) exists classically and is continuous. It is given by $
    g_{k,\eta}
    :=
    \cF_{d+1}^{-1}
    \left[
        (\mathrm{i}\tau)^k(\mathrm{i}\xi)^\eta \cF_{d+1}f
    \right]$,
and satisfies $\|g_{k,\eta}\|_{L^\infty(\bR^{d+1})}
    \lesssim_{d,\alpha,\beta}
    \|f\|_{\cB^{\alpha,\beta}(\bR^{d+1})}$.

It remains to estimate the H\"older seminorms. 
By the definition of \(\sigma_{k,\eta}\) and \(\rho_{k,\eta}\), we have $\frac{k+\sigma_{k,\eta}}{\alpha}
    +
    \frac{|\eta|}{\beta}
    \leq1$ and $\frac{k}{\alpha}
    +
    \frac{|\eta|+\rho_{k,\eta}}{\beta}
    \leq1$.
Thus the same weight bound yields
\begin{align}\label{260710810}
\begin{aligned}
&\int_{\bR^{d+1}}
\left\{
    |\tau|^{k+\sigma_{k,\eta}}|\xi|^{|\eta|}
    +
    |\tau|^k|\xi|^{|\eta|+\rho_{k,\eta}}
\right\}
|\cF_{d+1}f(\tau,\xi)|\,\dd\xi\dd\tau  \lesssim_{d,\alpha,\beta}
    \|f\|_{\cB^{\alpha,\beta}(\bR^{d+1})}.
\end{aligned}
\end{align}
If \((\sigma_{k,\eta},\rho_{k,\eta})\neq(0,0)\), we compare points satisfying
\(t=t'\) when \(\sigma_{k,\eta}=0\) and \(x=x'\) when \(\rho_{k,\eta}=0\).
Fourier inversion and
\(|\mathrm{e}^{\mathrm{i}a}-1|\leq 2^{1-\gamma}|a|^\gamma\), \(0\leq\gamma\leq1\), give
\[
\begin{aligned}
&|g_{k,\eta}(t,x)-g_{k,\eta}(t',x')| \\
\lesssim_{d,\alpha,\beta}\,&\int_{\bR^{d+1}}
\left(
    \mathbf 1_{\{\sigma_{k,\eta}>0\}}
    |t-t'|^{\sigma_{k,\eta}}
    |\tau|^{\sigma_{k,\eta}}+
    \mathbf 1_{\{\rho_{k,\eta}>0\}}
    |x-x'|^{\rho_{k,\eta}}
    |\xi|^{\rho_{k,\eta}}
\right) |\tau|^{k}|\xi|^{|\eta|}
|\cF_{d+1}f(\tau,\xi)|\,\dd\xi\dd\tau .
\end{aligned}
\]
Dividing according to the relevant case in the definition of the anisotropic
seminorm and using \eqref{260710810}, we obtain
\[
[g_{k,\eta}]_{\sigma_{k,\eta},\rho_{k,\eta};\bR^{d+1}}
\lesssim_{d,\alpha,\beta}
    \|f\|_{\cB^{\alpha,\beta}(\bR^{d+1})}.
\]
If \((\sigma_{k,\eta},\rho_{k,\eta})=(0,0)\), this seminorm is zero by
definition. 
Hence, for every \((k,\eta)\in\mathcal D_{\alpha,\beta}\),
\[
    \|\partial_t^k\partial_x^\eta f\|_{L^\infty(\bR^{d+1})}
    +
    [\partial_t^k\partial_x^\eta f]_{\sigma_{k,\eta},\rho_{k,\eta};\bR^{d+1}}
    \lesssim_{d,\alpha,\beta}
    \|f\|_{\cB^{\alpha,\beta}(\bR^{d+1})}.
\]
Summing over the finite set \(\mathcal D_{\alpha,\beta}\), we obtain $\|f\|_{C^{\alpha,\beta}_{t,x}(\bR^{d+1})}
    \lesssim_{d,\alpha,\beta}
    \|f\|_{\cB^{\alpha,\beta}(\bR^{d+1})}$.

Finally, let \(v\in\cB^{\alpha,\beta}(T)\), and let
\(V\in\cB^{\alpha,\beta}(\bR^{d+1})\) be any extension of \(v\). Applying the
global estimate to \(V\) and restricting to \([0,T]\times\bR^d\), we get $
    \|v\|_{C^{\alpha,\beta}_{t,x}([0,T]\times\bR^d)}
    \lesssim_{d,\alpha,\beta}
    \|V\|_{\cB^{\alpha,\beta}(\bR^{d+1})}$.
Taking the infimum over all such extensions \(V\) gives the desired estimate
on \([0,T]\times\bR^d\).
\end{proof}

\begin{comm}

\begin{defn}\label{def:vector in B} Let $\alp\in(0,\ift)$. For a vector-valued function $\bff:\bR^d\to\bR^d$, we define
$$
\|\bff\|_{\cB^\alp(\mathbb R^d;\mathbb R^d)}
:=
\int_{\mathbb{R}^d}
(1+|\xi|^\alp)
\abs{\cF_{d}[\bff](\xi)}_{\mathbb{R}^d}
\,\mathrm{d}\xi.
$$
\end{defn}
\begin{rem}
According to Definition~\ref{def:vector in B}, any vector fields
$\bff,\bfg:\mathbb{R}^d\to\mathbb{R}^d$ satisfy
$$
\|\bff\cdot\bfg\|_{\cB^\alp(\mathbb{R}^d)}
\,\lesssim_{\alp}\,(2\pi)^{-d}
\|\bff\|_{\cB^\alp(\mathbb{R}^d;\mathbb{R}^d)}
\|\bfg\|_{\cB^\alp(\mathbb{R}^d;\mathbb{R}^d)}.
$$
Indeed, from $\cF_d[f_i\cdot g_i]
=(2\pi)^{-d}\cF_d[f_i]\ast\cF_d[g_i]$ for $i\in\set{1,2,\dots,d}$, we get
$$
\abs{\cF_d[\bff\cdot\bfg](\xi)}
\,\le\, (2\pi)^{-d}
\int_{\mathbb{R}^d}
|\cF_d[\bff](\eta)|_{\mathbb{R}^d}
|\cF_d[\bfg](\xi-\eta)|_{\mathbb{R}^d}
\,\mathrm{d}\eta,
$$ and hence
\begin{align*}
\begin{alignedat}{2}
\|\bff\cdot\bfg\|_{\cB^\alp(\mathbb{R}^d)}&=&&\int_{\mathbb{R}^d}(1+|\xi|^\alp)\left|\cF_d[\bff\cdot\bfg](\xi)\right|
\,\mathrm{d}\xi\\
&\leq&&(2\pi)^{-d}\int_{\mathbb{R}^d}\int_{\mathbb{R}^d}(1+|\xi|^\alp)|\cF_d[\bff](\eta)|_{\mathbb{R}^d}
|\cF_d[\bfg](\xi-\eta)|_{\mathbb{R}^d}\,\mathrm{d}\eta\,\mathrm{d}\xi\\
&\lesssim_\alp &&(2\pi)^{-d}\int_{\mathbb{R}^d}\int_{\mathbb{R}^d}(1+|\eta|^\alp)(1+|\xi-\eta|^\alp)|\cF_d[\bff](\eta)|_{\mathbb{R}^d}
|\cF_d[\bfg](\xi-\eta)|_{\mathbb{R}^d}\,\mathrm{d}\eta\,\mathrm{d}\xi\\
&=&&(2\pi)^{-d}\|\bff\|_{\cB^\alp(\mathbb{R}^d;\mathbb{R}^d)}\|\bfg\|_{\cB^\alp(\mathbb{R}^d;\mathbb{R}^d)}.
\end{alignedat}
\end{align*}
\end{rem}
\end{comm}

\subsection{Interpolation inequality}
We next record an interpolation inequality for anisotropic Fourier--Barron spaces, which controls intermediate temporal and spatial regularities by the corresponding endpoint norms.

\begin{prop}\label{2607201245}
Let $\alpha_0,\alpha_1,\beta_0,\beta_1\in[0,\infty)$ satisfy $\alpha_0\leq\alpha_1$ and $\beta_0\leq\beta_1$.
For any $v\in \cB^{\alpha_1,\beta_1}(\bR^{d+1})$ and $\theta\in[0,1]$,
\begin{align}\label{2607201248}
\|v\|_{\cB^{\alpha_\theta,\beta_\theta}(\bR^{d+1})}\leq\|v\|_{\cB^{\alpha_0,\beta_0}(\bR^{d+1})}^{1-\theta}\|v\|_{\cB^{\alpha_1,\beta_1}(\bR^{d+1})}^\theta,
\end{align}
where $\alpha_\theta:=(1-\theta)\alpha_0+\theta\alpha_1$ and $\beta_\theta:=(1-\theta)\beta_0+\theta\beta_1$.
\end{prop}
\begin{proof}
Let $\theta\in[0,1]$ and set $w_\tht:=1+|\tau|^{\alpha_\tht}+|\xi|^{\beta_\tht}$. By H\"older's inequality for finite sums, $w_\tht\leq w_0^{1-\theta}w_1^\theta$.
Therefore, applying H\"older's inequality on $\mathbb R^{d+1}$, 
\begin{align*}
\|v\|_{\cB^{\alpha_\theta,\beta_\theta}(\mathbb{R}^{d+1})}
&=\int_{\bR^{d+1}}w_\tht\abs{\mathcal{F}_{d+1}[v](\tau,\xi)}\dd \tau\dd\xi
\leq \int_{\bR^{d+1}}w_0^{1-\theta}w_1^\theta\abs{\mathcal{F}_{d+1}[v](\tau,\xi)}\dd \tau\dd\xi\\
&\leq
\left(\int_{\mathbb{R}^{d+1}}w_0|\mathcal{F}_{d+1}[v](\tau,\xi)|\dd\tau\dd\xi\right)^{1-\theta}
\left(\int_{\mathbb{R}^{d+1}}w_1|\mathcal{F}_{d+1}[v](\tau,\xi)|\dd\tau\dd\xi\right)^\theta\\
&=\|v\|_{\cB^{\alpha_0,\beta_0}(\bR^{d+1})}^{1-\theta}\|v\|_{\cB^{\alpha_1,\beta_1}(\bR^{d+1})}^\theta,
\end{align*} as desired.
\end{proof}

\subsection{Dimension-controlled product and vector-field estimates}
We next derive the multiplication estimates needed to handle the lower-order drift terms in the parabolic equation, where the resulting constant does not deteriorate with the spatial dimension.
\begin{prop}\label{260516347} Let $T\in(0,\ift)$ and $\alp,\bt\in [0,\infty)$. \medskip
\begin{enumerate}
\item[(i)] For any $f,g\in\cB^{\alpha,\beta}(T)$, we have
\begin{align}\label{260311545_11}
\|fg\|_{\mathcal{B}^{\alpha,\beta}(T)}\leq C(\alp,\bt)\f{1}{(2\pi)^{d+1}}\|f\|_{\mathcal{B}^{\alpha,\beta}(T)}\|g \|_{\mathcal{B}^{\alpha,\beta}(T)}\,.
\end{align}

\item[(ii)] We additionally assume that
$\alpha\geq 0$ if $\beta>0$, and $\alpha=0$ if $\beta=0$.
For any $v\in\cB^{\alpha+\alpha/\beta,\beta+1}(T)$ and $b_1,\ldots,b_d\in\cB^{\alp,\bt}(T)$, denoting $\bfb:=(b_1,\ldots,b_d)$,
we have
\begin{align}\label{260311545}
\begin{aligned}
\|\bfb\cdot \nabla v\|_{\mathcal{B}^{\alpha,\beta}(T)}\leq C(\alp,\bt) \f{\|\bfb\|}{(2\pi)^{d+1}}
\|v\|_{\mathcal{B}^{\alpha+\alpha/\beta,\beta+1}(T)}\,,
\end{aligned}
\end{align}
where
\[
\|\bfb\|:=\left(\sum_{k=1}^d\|b_k\|_{\mathcal B^{\alpha,\beta}(T)}^2\right)^{1/2}.
\]
\end{enumerate}
Here we adopt the convention $\alp/\bt:=0$ when $\alp=\bt=0$.

\end{prop}
\begin{proof}
For (i), let \(F,G\in\cB^{\alp,\bt}(\bR^{d+1})\) be arbitrary extensions of \(f\) and \(g\), respectively. Then \(FG\) is an extension of \(fg\), and
\[
\nrm{fg}_{\cB^{\alp,\bt}(T)}
\leq
\nrm{FG}_{\cB^{\alp,\bt}(\bR^{d+1})}=\int_{\mathbb{R}^{d+1}}(1+|\tau|^{\alpha}+|\xi|^{\beta})\abs{\cF_{d+1}[FG](\tau,\xi)}\dd\xi\dd\tau.
\]
Since
\[
\cF_{d+1}[FG]
=
\f{1}{(2\pi)^{d+1}}
\cF_{d+1}[F]*\cF_{d+1}[G],
\]
and $1+|\tau|^\alp+|\xi|^\bt
\lesssim_{\alp,\bt}
\paren{1+|\tau'|^\alp+|\xi'|^\bt}
\paren{1+|\tau-\tau'|^\alp+|\xi-\xi'|^\bt}$,
we obtain
\[
\nrm{FG}_{\cB^{\alp,\bt}(\bR^{d+1})}
\lesssim_{\alp,\bt}
\f{1}{(2\pi)^{d+1}}
\nrm{F}_{\cB^{\alp,\bt}(\bR^{d+1})}
\nrm{G}_{\cB^{\alp,\bt}(\bR^{d+1})}.
\]
Taking the infimum over all such extensions \(F\) and \(G\) gives (i).

For (ii), let \(\bfB=(B_1,\ldots,B_d)\) and \(V\) be arbitrary extensions of \(\bfb\) and \(v\), respectively. By (i),
\[
\begin{aligned}
\nrm{\bfb\cdot\nabla v}_{\cB^{\alp,\bt}(T)}
&\leq
\nrm{\bfB\cdot\nabla V}_{\cB^{\alp,\bt}(\bR^{d+1})}\leq
\sum_{k=1}^{d}
\nrm{B_k\partial_{x_k}V}_{\cB^{\alp,\bt}(\bR^{d+1})} \\
&\lesssim_{\alp,\bt}
\f{1}{(2\pi)^{d+1}}
\sum_{k=1}^{d}
\nrm{B_k}_{\cB^{\alp,\bt}(\bR^{d+1})}
\nrm{\partial_{x_k}V}_{\cB^{\alp,\bt}(\bR^{d+1})} \\
&\leq
\f{1}{(2\pi)^{d+1}}
\paren{\sum_{k=1}^{d}
\nrm{B_k}_{\cB^{\alp,\bt}(\bR^{d+1})}^{2}}^{1/2}
\paren{\sum_{k=1}^{d}
\nrm{\partial_{x_k}V}_{\cB^{\alp,\bt}(\bR^{d+1})}^{2}}^{1/2}.
\end{aligned}
\]
Moreover, by Minkowski's inequality and Young's inequality,
\begin{align*} \paren{\sum_{k=1}^{d}\nrm{\partial_{x_k}V}_{\cB^{\alp,\bt}(\bR^{d+1})}^2}^{\f{1}{2}}&\leq\int_{\bR^d}\int_\bR\left(\sum_{k=1}^{d}|\xi_k|^2\right)^{1/2}\paren{1+|\tau|^\alp+|\xi|^\bt}|\cF_{d+1}[V](\tau,\xi)|\,\dd\tau\dd\xi\\ &=\int_{\bR^d}\int_\bR\paren{|\xi|+|\tau|^\alp|\xi|+|\xi|^{\bt+1}}|\cF_{d+1}[V](\tau,\xi)|\,\dd\tau\dd\xi\\ &\lesssim_{\bt}\int_{\bR^d}\int_\bR\paren{1+|\tau|^{\alp+\alp/\bt}+|\xi|^{\bt+1}}| \cF_{d+1}[V](\tau,\xi)|\,\dd\tau\dd\xi\lesssim\,\nrm{V}_{\cB^{\alp+\alp/\bt,\bt+1}(\bR^{d+1})}
\end{align*}
If \(\alp=\bt=0\), the same bound follows from \(|\xi|\leq 1+|\xi|\). Therefore
\[
\nrm{\bfb\cdot\nabla v}_{\cB^{\alp,\bt}(T)}
\lesssim_{\alp,\bt}
\f{1}{(2\pi)^{d+1}}
\paren{\sum_{k=1}^{d}
\nrm{B_k}_{\cB^{\alp,\bt}(\bR^{d+1})}^{2}}^{1/2}
\nrm{V}_{\cB^{\alp+\alp/\bt,\bt+1}(\bR^{d+1})}.
\]
Taking the infimum over \(V\), and then over \(B_1,\ldots,B_d\), gives (ii).
\end{proof}

\section{Parameter-Uniform Extensions of One-Sided Semigroup Trajectories}
\label{sec:semigroup-extension}

This section isolates the extension argument needed to place a forward semigroup trajectory in a global temporal Fourier class. The construction first enforces finite-order matching at the initial time, then proves a weighted Fourier estimate whose constant is uniform in the positive semigroup parameter.

\subsection{Finite reflected extensions and matching conditions}
\label{subsec: heat with shift}

We start by establishing the reflection coefficients, which ensure that the temporal derivatives from both sides match at the initial time.

\begin{lem}
\label{lem:reflection_coeffs}
For any integer $N \ge 1$, there exists a unique vector $(c_1,\dots, c_N)^T\in\bR^N$ such that
$$
\sum_{k=1}^N c_k k^m = (-1)^m \quad \text{for all } m = 0, 1, \dots, N-1.
$$ 
\end{lem}
\begin{proof}
With $V_N=(k^{j-1})_{j,k=1}^N$ as in Theorem~\ref{lem:fourier_bound}, the system is
$$
V_N(c_1,\ldots,c_N)^T=(1,-1,\ldots,(-1)^{N-1})^T.
$$
The Vandermonde matrix $V_N$ is invertible because its nodes $1,\ldots,N$ are distinct, giving the stated existence and uniqueness.
\end{proof}

\subsection{Weighted temporal Fourier estimates}
\label{subsec:weighted-temporal-estimates}

Using these coefficients, we extend an exponentially decaying kernel to negative times while preserving sufficient regularity at the initial time $t=0$. We now prove Theorem~\ref{lem:fourier_bound}, stated in Subsection~\ref{sec:main}.

\begin{proof}[Proof of Theorem~\ref{lem:fourier_bound}]
Lemma~\ref{lem:reflection_coeffs} give
$h_N^{(m)}(0-)=h_N^{(m)}(0+)=(-1)^m$
for $m=0,\ldots,N-1$, and hence $h_N\in C^{N-1}(\mathbb R)$.
Direct integration gives
\begin{equation}\label{eq260423_4}
\begin{aligned}
\cF_1[h_N](\tau) &= \int_0^\infty \mathrm{e}^{-t(1 + \mathrm{i}\tau)} \, \mathrm{d}t + \sum_{k=1}^N c_k \int_{-\infty}^0 \mathrm{e}^{t(k - \mathrm{i}\tau)} \, \mathrm{d}t = \frac{1}{1 + \mathrm{i}\tau} + \sum_{k=1}^N \frac{c_k}{k - \mathrm{i}\tau}.
\end{aligned}
\end{equation}
For $|\tau|>N$, expanding each term as an absolutely convergent geometric series and using Lemma~\ref{lem:reflection_coeffs} gives
$$
\cF_1[h_N](\tau) = \sum_{m=N}^\infty \frac{1}{(\mathrm{i}\tau)^{m+1}} \left( (-1)^m - \sum_{k=1}^N c_k k^m \right).
$$
For $|\tau|>2N$, it follows that
$$
|\cF_1[h_N](\tau)| \q\lesssim_N\q \sum_{m=N}^\ift\f{N^m}{|\tau|^{m+1}}
\q\lesssim_N\q |\tau|^{-N-1}.
$$
The triangle inequality in \eqref{eq260423_4} also gives $|\cF_1[h_N](\tau)|\lesssim_N1$ for every $\tau$, proving the asserted decay of $\cF_1[h_N]$.
For $\lambda>0$, the Fourier scaling identity therefore yields
\begin{equation}\label{eq260423_3}
    |\cF_1[h_N(\lambda\cdot)](\tau)|
    =\lambda^{-1}|\cF_1[h_N](\tau/\lambda)|
    \lesssim_N\lambda^{-1}\left(1+\frac{|\tau|}{\lambda}\right)^{-N-1}.
\end{equation}

We now prove \eqref{260410549}. Fix $\chi\in C^\ift_c(\bR)$ and set $H_\lambda(t):=\chi(t)h_N(\lambda t)$.
Using
$\cF_1[H_\lambda]
=(2\pi)^{-1}\cF_1[\chi]\ast\cF_1[h_N(\lambda\cdot)]$
and \eqref{eq260423_3}, we obtain
\begin{equation}\label{260423334}
\begin{aligned}
\int_{\bR}
|\cF_1[H_\lambda](\tau)|
\dd \tau
\q&\lesssim_N\q
\nrm{\cF_1[\chi]}_{L^1(\bR)}
\int_{\bR}
\frac{1}{\lmb}
\left(
1+\frac{|\tau'|}{\lmb}
\right)^{-N-1}
\dd \tau'\lesssim_N\q
\nrm{\cF_1[\chi]}_{L^1(\bR)}.
\end{aligned}
\end{equation}
Moreover, using
$1+|\tau|^\alp
\lesssim_\alp
(1+|\tau-\tau'|^\alp)(1+|\tau'|^\alp)$
together with \eqref{eq260423_3}, we obtain
\begin{equation}\label{260423335}
\begin{aligned}
\int_{\bR}
(1+|\tau|^\alp)
|\cF_1[H_\lambda](\tau)|
\dd \tau&\lesssim_{N,\alp}
\nrm{\chi}_{\cB^\alp(\bR)}
\int_{\bR}
(1+|\tau'|^\alp)
\frac{1}{\lmb}
\left(
1+\frac{|\tau'|}{\lmb}
\right)^{-N-1}
\dd \tau'
\\
&\lesssim_{N,\alp}
\nrm{\chi}_{\cB^\alp(\bR)}
(1+\lmb^\alp).
\end{aligned}
\end{equation}
In the last inequality, we used $N>\alp$ and the change of variables $\tau'=\lmb s$.
 Combining \eqref{260423334} and \eqref{260423335}, we get \eqref{260410549}. 
\end{proof}

\subsection{Application to the homogeneous fractional heat evolution}
\label{subsec:homogeneous-evolution}

We now prove Corollary~\ref{thm:homogeneous_solution}, stated in Subsection~\ref{sec:main}, by applying the preceding Fourier estimate frequency by frequency to the homogeneous heat equation with damping.

\begin{proof}[Proof of Corollary~\ref{thm:homogeneous_solution}]
In the proof, we explicitly construct a global extension $V \in \cB^{ s /\gamma, s }(\mathbb{R}^{d+1})$ of $v$. 
We fix a cutoff function $\chi \in C_c^\infty(\mathbb{R})$ such that $\chi = 1$ on $[0, T]$. Let $N$ be the integer in the interval $( s /\gmm, s /\gmm+1]$, and let $h_N$ be the profile from Theorem~\ref{lem:fourier_bound}. Define $V(t,x)$ through its spatial Fourier transform by
$$
\mathcal{F}_d[V(t,\cdot)](\xi) = \chi(t) h_N\bigl(t(|\xi|^{\gamma}+\lmb)\bigr) \mathcal{F}_d[v_0](\xi).
$$
For each $t \in [0,T]$, we have $\chi(t)=1$ and $h_N\bigl(t(|\xi|^\gmm+\lmb)\bigr) = \mathrm{e}^{-t(|\xi|^{\gamma}+\lmb)}$, and hence $\mathcal{F}_d[V(t,\cdot)]$ coincides with the spatial Fourier transform of $v(t,\cdot)$, i.e., $V|_{[0,T]\times\mathbb{R}^d} = v$. Moreover, Fubini's theorem and Theorem~\ref{lem:fourier_bound} with $\alp=\bt=s/\gmm$ give
\begin{equation}\begin{aligned}\label{eq260424_1}
&\|V\|_{\cB^{s/\gmm, s }(\mathbb{R}^{d+1})}+\lambda^{s/\gmm}\|V\|_{\cB^{0,0}(\mathbb{R}^{d+1})}\\
\lesssim_{s/\gamma}& \int_{\mathbb{R}^d}\left[ \int_{\mathbb{R}} (1 + |\tau|^{s/\gmm} +(|\xi|^\gmm+\lambda)^{s/\gmm}) |\cF_1[\chi(\cdot)h_N\bigl((|\xi|^{\gamma}+\lambda)\cdot\bigr)](\tau)|\, \mathrm{d}\tau\right]|\mathcal{F}_d[v_0](\xi)|\, \mathrm{d}\xi \\
\lesssim_{s/\gmm}&\nrm{\chi}_{\cB^{s/\gmm}(\bR)}\int_{\mathbb{R}^d} \left(1 +(|\xi|^\gamma+\lambda)^{s/\gmm}\right) |\mathcal{F}_d[v_0](\xi)|\, \mathrm{d}\xi.
\end{aligned} \end{equation}
Applying 
$\,(|\xi|^\gamma+\lambda)^{s/\gmm}
\,\lesssim_{s/\gmm}\,
|\xi|^{s}+\lambda^{s/\gmm}\,$ to the above relation gives  
\begin{equation}\label{eq260519_1}
\|V\|_{\cB^{s/\gmm, s }(\mathbb{R}^{d+1})}+\lambda^{s/\gmm}\|V\|_{\cB^{0,0}(\mathbb{R}^{d+1})}\,\lesssim_{T,s/\gmm}\,\nrm{v_0}_{\cB^ s (\bR^d)}+\lmb^{s/\gmm}\nrm{v_0}_{\cB^0(\bR^d)}
\end{equation} where the dependence on $\nrm{\chi}_{\cB^{s/\gmm}(\bR)}$ in \eqref{eq260424_1} is absorbed into the dependence on $T$ and $s/\gmm$.
This gives
\begin{alignat*}{2}
\lambda^{(1-\theta) s/\gmm}\|V\|_{\cB^{\theta s/\gamma,\theta s}(\bR^{d+1})}\,&\leq && \left(\|V\|_{\cB^{s/\gmm, s }(\mathbb{R}^{d+1})}\right)^{\theta}\left(\lambda^{s/\gmm}\|V\|_{\cB^{0,0}(\mathbb{R}^{d+1})}\right)^{1-\theta}\\
&\lesssim_{T,s/\gamma,\tht}\,&&\nrm{v_0}_{\cB^ s (\bR^d)}+\lmb^{s/\gmm}\nrm{v_0}_{\cB^0(\bR^d)},
\end{alignat*} where we used the interpolation inequality in Proposition~\ref{2607201245} and Young's inequality. The estimates \eqref{260510326} and \eqref{2607201155} follow from \eqref{eq260519_1} and the preceding inequality, respectively. This completes the proof.
\end{proof} 

\section{Fourier--Barron Regularity for Fractional Parabolic PDEs}
\label{sec:regularity}

\subsection{Temporal cutoffs and the whole-space Fourier multiplier}

In the subsequent arguments, exponential factors $\mathrm{e}^{\lmb t},\mathrm{e}^{-\lmb t}$ in time will be used to introduce or remove the damping term $\lmb v$.
For $\alpha\geq0$, let $\cB^\alpha(T)$ denote the restriction of $\cB^\alpha(\mathbb R)$ to $[0,T]$, equipped with the norm
$$
\|\phi\|_{\cB^\alpha(T)}
:=
\inf_{\substack{\Phi\in\cB^\alpha(\mathbb R)\\ \Phi|_{[0,T]}=\phi}}
\|\Phi\|_{\cB^\alpha(\mathbb R)}.
$$
The following lemma shows that multiplication by such temporal functions is bounded in the relevant Barron spaces.

\begin{lem}
\label{lem:schwartz_multiplier}
Let $T\in (0,\ift)$ and $\alpha_1,\alpha_2,\beta \in [0,\infty)$. Then we have the following.\medskip
\begin{enumerate}
\item[(i)] Suppose that $\chi \in \cB^{\alpha_1}(\mathbb{R})$ and
$v \in \cB^{\alpha_2,\beta}(\mathbb{R}^{d+1})$. The pointwise product
$(t,x)\mapsto \chi(t)v(t,x)$ belongs to
$\cB^{\min\{\alpha_1,\alpha_2\},\beta}(\domain)$ and satisfies
$$
\|\chi v\|_{\cB^{\min\{\alpha_1,\alpha_2\},\beta}(\mathbb{R}^{d+1})} \,\leq\, C(\alpha_1,\alpha_2)\|\chi\|_{\cB^{\alpha_1}(\mathbb{R})}\|v\|_{\cB^{\alpha_2,\beta}(\mathbb{R}^{d+1})}.$$
\item[(ii)] Suppose that $\phi\in \cB^{\alp_1}(T)$ and
$w\in \cB^{{\alp_2},\bt}(T)$. The pointwise product
$(t,x)\mapsto\phi(t)w(t,x)$ belongs to
$\cB^{\min\{\alp_1,\alp_2\},\bt}(T)$ and satisfies
$$\nrm{\phi w}_{\cB^{\min\{\alp_1,\alp_2\},\bt}(T)}\,\leq\,
C(\alp_1,\alp_2)\|\phi\|_{\cB^{\alpha_1}(T)}\nrm{w}_{\cB^{\alp_2,\bt}(T)}.$$
\end{enumerate}

\end{lem}

\begin{proof}
We first prove (i). Using $2\pi\mathcal{F}_{d+1}[\chi v](\tau, \xi) = (\mathcal{F}_{1}[\chi] *_{\tau} \mathcal{F}_{d+1}[v])(\tau, \xi)$, we have
\begin{equation} \label{eq:multiplier_integral}
    \begin{aligned}
    \|\chi v\|_{\cB^{\min\{\alpha_1,\alpha_2\},\beta}(\mathbb{R}^{d+1})}\lesssim\int_{\mathbb{R}^d} \int_{\mathbb{R}} \int_{\mathbb{R}} \paren{1 + |\tau|^{\min\{\alpha_1,\alpha_2\}} + |\xi|^{\beta}} |\mathcal{F}_1[\chi](\tau - \tau')| |\mathcal{F}_{d+1}[v](\tau', \xi)| \, \mathrm{d}\tau' \mathrm{d}\tau \mathrm{d}\xi.
    \end{aligned}
\end{equation}
By the inequality $t^a\leq 1+t^b$, which holds for any $t>0$ and $0\leq a\leq b$, we have
$$\begin{alignedat}{2}
    |\tau|^{\min\{\alpha_1,\alpha_2\}} \lesssim_{\alpha_1,\alpha_2}  |\tau-\tau'|^{\min\{\alpha_1,\alpha_2\}} +\ |\tau'|^{\min\{\alpha_1,\alpha_2\}} \,&\leq&&\,2+|\tau - \tau'|^{\alpha_1}+ |\tau'|^{\alpha_2}
\end{alignedat}$$ and hence $1 + |\tau|^{\min\{\alpha_1,\alpha_2\}} + |\xi|^{\beta}\,\lesssim_{\alp_1,\alp_2}(1+|\tau-\tau'|^{\alp_1})(1+|\tau'|^{\alp_2}+|\xi|^{\bt})$.
Applying this inequality to \eqref{eq:multiplier_integral} and changing variables yields
$$\begin{aligned}
    \|\chi v\|_{\cB^{\min\{\alpha_1,\alpha_2\},\beta}(\mathbb{R}^{d+1})}\lesssim_{\alpha_1,\alpha_2}\|\chi\|_{\cB^{\alpha_1}(\mathbb{R})}\|v\|_{\cB^{\alpha_2,\beta}(\mathbb{R}^{d+1})},
\end{aligned}$$ where we used Fubini-Tonelli's theorem to interchange the order of integration. This proves (i).

To prove (ii), let $\Phi\in\cB^{\alp_1}(\bR)$ and $W\in\cB^{\alp_2,\bt}(\bR^{d+1})$ be extensions of the given functions $\phi\in\cB^{\alp_1}(T)$ and $w\in\cB^{\alp_2,\bt}(T)$, respectively, such that $\Phi=\phi$ on $[0,T]$, $W=w$ on $[0,T]\times\bR^d$, and $\nrm{\Phi}_{\cB^{\alp_1}(\bR)}\leq
2\nrm{\phi}_{\cB^{\alp_1}(T)}$, $\nrm{W}_{\cB^{\alp_2,\bt}(\bR^{d+1})}\leq
2\nrm{w}_{\cB^{\alp_2,\bt}(T)}$.
The product $\Phi W$ extends $\phi w$, and (i) gives
\begin{align*}
\begin{aligned}
\|\phi w\|_{\cB^{\min\{\alpha_1,\alpha_2\},\beta}(T)}\leq \|\Phi W\|_{\cB^{\min\{\alpha_1,\alpha_2\},\beta}(\mathbb{R}^{d+1})}\lesssim_{\alp_1,\alp_2} \|\Phi\|_{\cB^{\alpha_1}(\mathbb{R})}\|W\|_{\cB^{\alpha_2,\beta}(\mathbb{R}^{d+1})}\leq4\|\phi\|_{\cB^{\alpha_1}(T)}\|w\|_{\cB^{\alpha_2,\beta}(T)}.
\end{aligned}
\end{align*} 
This completes the proof. 
\end{proof}

We next consider the inhomogeneous equation on the full space--time domain, in which the solution gains enhanced regularity from the source term $f$.

\begin{thm}
\label{thm:global_shifted_mr} Let $\gmm,\lmb\in(0,\ift)$ and $\alpha,\beta\in [0,\infty)$. 
For any $f\in \cB^{\alpha,\beta}(\mathbb{R}^{d+1})$, the function $v\in \mathcal{S}'(\bR^{d+1})$ satisfying the relation
$$
\mathcal{F}_{d+1}[v](\tau,\xi)
=\frac{\mathcal{F}_{d+1}[f](\tau,\xi)}{\lambda+\mathrm{i}\tau+|\xi|^{\gamma}}
$$
belongs to $\cB^{1+\alpha,\gamma+\beta}(\mathbb{R}^{d+1})$ and solves 
$\,\partial_t v-\Delta^{\gamma/2}v+\lambda v=f\,$ in $\,\cS'(\mathbb{R}^{d+1})\,$.
Moreover, it holds that
\begin{equation}\label{eq260521_1}
\min\{\lambda , 1\}\|v\|_{\cB^{1+\alpha,\gamma+\beta}(\mathbb{R}^{d+1})}+\lambda \|v\|_{\cB^{\alpha,\beta}(\bR^{d+1})}\leq 3\|f\|_{\cB^{\alpha,\beta}(\bR^{d+1})}\,.
\end{equation}

\end{thm}

\begin{proof}
By construction of $v$, we have $(\partial_t-\Delta^{\gamma/2}+\lambda)v=f$ in $\mathcal{S}'(\mathbb{R}^{d+1})$ and $\lambda \|v\|_{\cB^{\alpha,\beta}(\mathbb{R}^{d+1})}\leq \|f\|_{\cB^{\alpha,\beta}(\mathbb{R}^{d+1})}$. For $\|v\|_{\cB^{1+\alpha,\gamma+\beta}(\mathbb{R}^{d+1})}$, we observe that any $(\tau,\xi)\in\mathbb{R}\times\mathbb{R}^d$ satisfies
$$
\left|\frac{1+|\tau|^{1+\alpha}+|\xi|^{\gamma+\beta}}{\lambda+\mathrm{i}\tau+|\xi|^{\gamma}}\right|
\,\leq\, \left(1+\frac{1}{\lambda}\right)(1+|\tau|^{\alpha}+|\xi|^{\beta}),
$$
which gives $\|v\|_{\cB^{1+\alpha,\gamma+\beta}(\mathbb{R}^{d+1})}
\leq 2\max\{\lambda^{-1}, 1\}\|f\|_{\cB^{\alpha,\beta}(\bR^{d+1})}$. This completes the proof.
\end{proof}

\subsection{Finite-time inhomogeneous problem}

Before restricting the preceding global solution to a finite time interval, we record the uniqueness statement required for the initial-value problem. 

\begin{lem}[Uniqueness in $\cB^{1,\gmm}(T)$]
\label{lem:uniqueness_homogeneous_heat}
Let $T,\gamma\in(0,\infty)$ and $\lmb\in[0,\ift)$. If $q\in \cB^{1,\gmm}(T)$ satisfies
$$
\partial_t q-\Delta^{\gmm/2}q+\lmb q=0\quad 
 \text{in }(0,T)\times\bR^d,\qquad
q(0,\cdot)=0\quad 
 \text{in }\bR^d,
$$
then $q=0$ on \([0,T]\times\bR^d\).
\end{lem}
\begin{proof}
Let \(Q\in \cB^{1,\gmm}(\bR^{d+1})\) be an extension of \(q\). By Fubini's
theorem, for a.e. \(\xi\in\bR^d\), the functions $\tau\mapsto \cF_{d+1}[Q](\tau,\xi)$, $\tau\mapsto \tau\,\cF_{d+1}[Q](\tau,\xi)$ belong to \(L^1(\bR)\). For such \(\xi\), define
\(\tld Q(t,\xi):=\cF_1^{-1}\left[\cF_{d+1}[Q](\cdot,\xi)\right](t)\). Then \(t\mapsto \tld Q(t,\xi)\) belongs to \(C^1(\bR)\), by dominated
convergence and the estimate \(|\mathrm{e}^{\mathrm{i}h\tau}-1|/|h|\leq |\tau|\). Moreover,
for every \(t\in\bR\), \(\tld Q(t,\cdot)\) agrees with
\(\cF_d[Q(t,\cdot)]\) in \(\cS'(\bR^d)\). In particular, for
\(t\in[0,T]\), $\tld Q(t,\cdot)=\cF_d[q(t,\cdot)]$ in $\cS'(\bR^d)$.

Taking the spatial Fourier transform of the equation for \(q\) on \((0,T)\times\bR^d\), we obtain, for a.e. \(\xi\in\bR^d\),
\[
    \partial_t \tld Q(t,\xi)
    +
    \paren{|\xi|^\gmm+\lmb}\tld Q(t,\xi)
    =
    0
    \quad\text{in }\cD'(0,T).
\]
Since \(t\mapsto \tld Q(t,\xi)\) is \(C^1\), this identity holds pointwise in
\(t\in(0,T)\). 
Hence $\tld Q(t,\xi)
    =
    \mathrm{e}^{-t\paren{|\xi|^\gmm+\lmb}}\tld Q(0,\xi)$ for $t\in[0,T]$.
By the initial condition \(q(0,\cdot)=0\), $\tld Q(0,\xi)
    =
    \cF_d[q(0,\cdot)](\xi)
    =
    0$ for almost all $\xi\in\mathbb{R}^d$.
Therefore \(\tld Q(t,\xi)=0\) for every \(t\in[0,T]\) and for a.e. \(\xi\in\bR^d\). 
Thus \(q(t,\cdot)=0\) in \(\cS'(\bR^d)\) for every \(t\in[0,T]\). 
Since \(q\in C([0,T]\times\bR^d)\), we conclude that \(q=0\) on \([0,T]\times\bR^d\) pointwise. 
This completes the proof.
\end{proof}

We now apply the temporal multiplier estimate (Lemma~\ref{lem:schwartz_multiplier}) to the finite-time restriction of the global solution in Theorem~\ref{thm:global_shifted_mr}. Combining this with the homogeneous correction in Corollary~\ref{thm:homogeneous_solution}, we obtain the maximal regularity for the finite-time solution with zero initial data.

\begin{thm}
\label{thm:finite_mr_unshifted}
Let \(T,\gmm,\lmb\in(0,\infty)\) and \(s\in[0,\infty)\).
For every \(f\in \cB^{s/\gmm,s}(T)\), there exists a unique
\(v\in \cB^{1+s/\gmm,\gmm+s}(T)\) satisfying
\[
\begin{cases}
    \partial_t v-\Delta^{\gmm/2} v+\lmb v=f
        &\quad\text{in }(0,T)\times\mathbb R^d,\\
    v(0,\cdot)=0
        &\quad\text{in }\mathbb R^d.
\end{cases}
\]
Moreover, for any $\theta\in[0,1]$,
\begin{align}\label{2607201144}
\lmb^{1-\theta}\|v\|_{\cB^{\theta+s/\gmm,\theta\gamma+s}(T)}\leq C(T,\gamma,s,\theta)(1+\lambda^{s/\gamma})\|f\|_{\cB^{s/\gmm,s}(T)}.
\end{align}
In particular, 
\begin{equation}\label{eq260429_4}
\|v\|_{\cB^{1+s/\gmm,\gmm+s}(T)}
+
\lmb\|v\|_{\cB^{s/\gmm,s}(T)}
\le
C(T,\gmm,s)\bigl(1+\lmb^{s/\gmm}\bigr)
\|f\|_{\cB^{s/\gmm,s}(T)}.
\end{equation}
\end{thm}

\begin{proof}
Set \(a:=s/\gmm\). 

\medskip

\noindent\textbf{Step 1. An inhomogeneous solution.}
Choose \(\chi_-,\chi_+\in C_c^\infty(\bR)\) such that
\[
    \chi_-(t)=\mathrm{e}^{-t},
    \qquad
    \chi_+(t)=\mathrm{e}^t
    \quad\text{on a neighborhood of }[0,T].
\]
By Lemma~\ref{lem:schwartz_multiplier}(ii), $\|\chi_- f\|_{\cB^{a,s}(T)}
\lesssim_{T,\gamma,s}
    \|f\|_{\cB^{a,s}(T)}$.
Choose an extension \(F\in\cB^{a,s}(\bR^{d+1})\) of \(\chi_- f\) such that
\[
\|F\|_{\cB^{a,s}(\bR^{d+1})}
\leq 2\|\chi_- f\|_{\cB^{a,s}(T)}.
\]
By Theorem~\ref{thm:global_shifted_mr}, there exists
\(U\in\cB^{1+a,\gmm+s}(\bR^{d+1})\) solving
\[
    \partial_t U-\Delta^{\gmm/2}U+(\lmb+1)U=F
    \quad\text{in }\bR^{d+1},
\]
and satisfying
\begin{equation}\label{eq:finite-mr-U}
    \|U\|_{\cB^{1+a,\gmm+s}(\bR^{d+1})}
    +
    (\lmb+1)\|U\|_{\cB^{a,s}(\bR^{d+1})}
    \leq
    3\|F\|_{\cB^{a,s}(\bR^{d+1})}
    \lesssim_{T,\gamma,s}
    \|f\|_{\cB^{a,s}(T)}.
\end{equation}
Therefore, for any $\theta\in[0,1]$, we apply Proposition~\ref{2607201245} and Young's inequality to get
\begin{align}\label{eq:finite-mr-U_11}
\begin{aligned}
    \lambda^{1-\theta}\|U\|_{\cB^{\theta+a,\theta\gmm+s}(\bR^{d+1})}\,&\leq 
    \|U\|_{\cB^{1+a,\gmm+s}(\bR^{d+1})}^{\theta}
    \paren{\lambda\|U\|_{\cB^{a,s}(\bR^{d+1})}}^{1-\theta}\lesssim_{T,\gamma,s,\theta}\|f\|_{\cB^{a,s}(T)}.
\end{aligned}
\end{align}
Define
\(u:=\chi_+ U\big|_{[0,T]\times\bR^d}.\)
Since \(\chi_+=\mathrm{e}^t\) and \(F=\mathrm{e}^{-t}f\) on \((0,T)\times\bR^d\), we have
\[
    \partial_t u-\Delta^{\gmm/2}u+\lmb u=f
    \quad\text{in }(0,T)\times\bR^d.
\]
Moreover, Lemma~\ref{lem:schwartz_multiplier}(ii) and
\eqref{eq:finite-mr-U_11} imply
\begin{equation}\label{eq:finite-mr-u}
    \lambda^{1-\theta}\|u\|_{\cB^{\theta+a,\theta\gmm+s}(T)}
    \lesssim_{T,\gamma,s,\theta}\|f\|_{\cB^{a,s}(T)}.
\end{equation}

\noindent\textbf{Step 2. Homogeneous correction.}
By Remark~\ref{26.02.22.15.15},
\begin{equation}\label{eq:finite-mr-trace}
    \|u(0,\cdot)\|_{\cB^{\gmm+s}(\bR^d)}
    \lesssim_{T,\gamma,s}
    \|u\|_{\cB^{1+a,\gmm+s}(T)},
    \qquad
    \|u(0,\cdot)\|_{\cB^0(\bR^d)}
    \lesssim_{T,\gamma,s}
    \|u\|_{\cB^{0,0}(T)}.
\end{equation}
Set $w(t,\cdot):=
    \mathrm{e}^{t(\Delta^{\gmm/2}-\lmb)}u(0,\cdot)$.
Then
\[
    \partial_t w-\Delta^{\gmm/2}w+\lmb w=0,
    \qquad
    w(0,\cdot)=u(0,\cdot).
\]
By Corollary~\ref{thm:homogeneous_solution}, for any $\delta\in[0,1]$,
$$
\lambda^{(1-\delta)(1+a)}\|w\|_{\cB^{\delta (1+a),\delta (\gamma+s)}(T)}\lesssim_{T,\gamma,s}
    \left(
    \|u(0,\cdot)\|_{\cB^{\gmm+s}(\bR^d)}
    +
    \lmb^{1+a}\|u(0,\cdot)\|_{\cB^0(\bR^d)}
    \right).
$$
For $\theta\in[0,1]$, we put $\delta=\frac{\theta+a}{1+a}$ to get 
\begin{equation}\label{eq:finite-mr-w-raw}
    \lmb^{1-\theta}\|w\|_{\cB^{\theta+a,\theta\gamma+s}(T)}
    \lesssim_{T,\gamma,s}
    \left(
    \|u(0,\cdot)\|_{\cB^{\gmm+s}(\bR^d)}
    +
    \lmb^{1+a}\|u(0,\cdot)\|_{\cB^0(\bR^d)}
    \right).
\end{equation}
Using \eqref{eq:finite-mr-trace} and
\eqref{eq:finite-mr-u}, we obtain
\begin{align}\label{260713348}
    \|u(0,\cdot)\|_{\cB^{\gamma+s}(\bR^d)}
    \lesssim_{T,\gamma,s}
    \|u\|_{\cB^{1+a,\gamma+s}(T)}\lesssim_{T,\gamma,s}
    \|f\|_{\cB^{a,s}(T)}
\end{align}
and
\begin{align}\label{260713349}
\begin{aligned}
    \lmb^{1+a}\|u(0,\cdot)\|_{\cB^0(\bR^d)}
    &\lesssim_{T,\gamma,s}
    \lmb^{1+a}\|u\|_{\cB^{0,0}(T)}                                      \\
    &\lesssim_{T,\gamma,s}
    \lmb^{1+a}\|u\|_{\cB^{a,s}(T)}\lesssim_{T,\gamma,s}
    \lmb^a\|f\|_{\cB^{a,s}(T)}.
\end{aligned}
\end{align}
Combining \eqref{eq:finite-mr-w-raw}--\eqref{260713349}, we get
\begin{equation}\label{eq:finite-mr-w}
    \lmb^{1-\theta}\|w\|_{\cB^{\theta+a,\theta\gamma+s}(T)}
    \lesssim_{T,\gamma,s}
    \bigl(1+\lmb^a\bigr)
    \|f\|_{\cB^{a,s}(T)}.
\end{equation}

\medskip

\noindent\textbf{Step 3. Existence.}
Define \(v:=u-w\). Then
\[
    \partial_t v-\Delta^{\gmm/2}v+\lmb v=f
    \quad\text{in }(0,T)\times\bR^d,
    \qquad
    v(0,\cdot)=0.
\]
The estimate \eqref{2607201144} follows from \eqref{eq:finite-mr-u} and
\eqref{eq:finite-mr-w}. The estimate \eqref{eq260429_4} follows directly from \eqref{2607201144} and the interpolation inequality in Proposition~\ref{2607201245}.
\medskip

\noindent\textbf{Step 4. Uniqueness.}
Let \(v_1,v_2\in\cB^{1+s/\gmm,\gmm+s}(T)\) be two solutions, and set
\(q:=v_1-v_2\). Then
\[
    \partial_t q-\Delta^{\gmm/2}q+\lmb q=0,
    \qquad
    q(0,\cdot)=0.
\]
By Proposition~\ref{lem:self-embedding}, \(q\in\cB^{1,\gmm}(T)\). Hence
Lemma~\ref{lem:uniqueness_homogeneous_heat} implies \(q=0\) on
\([0,T]\times\bR^d\). This proves uniqueness.
\end{proof}
\subsection{Initial data and the homogeneous correction}\label{subsec: inhomo. heat}

We now recover the undamped fractional heat equation without lower-order terms by conjugating with an exponential temporal factor. The required regularity estimate follows from the damped estimates in Corollary~\ref{thm:homogeneous_solution} and Theorem~\ref{thm:finite_mr_unshifted}.

\begin{thm}\label{thm: inhomo. heat}
Let \(T,\gmm\in(0,\infty)\) and \(s\in[0,\infty)\). 
For every
    $v_0\in\cB^{\gmm+s}(\bR^d)$ and 
    $f\in \cB^{s/\gmm,s}(T)$,
there exists a unique solution
\(v\in\cB^{1+s/\gmm,\gmm+s}(T)\) satisfying
\[
\begin{cases}
    \partial_t v-\Delta^{\gmm/2}v=f
        &\quad \text{in }(0,T)\times\bR^d,\\
    v(0,\cdot)=v_0
        &\quad \text{in }\bR^d.
\end{cases}
\]
Moreover, it holds that
\begin{equation}\label{260311510}
\|v\|_{\cB^{1+s/\gmm,\gmm+s}(T)}
\leq
C(T,s,\gamma)\left(
\|v_0\|_{\cB^{\gmm+s}(\mathbb R^d)}
+
\|f\|_{\cB^{s/\gmm,s}(T)}
\right).
\end{equation}
\end{thm}

\begin{proof}
Choose \(\chi_-,\chi_+\in C_c^\infty(\bR)\) such that
\[
    \chi_-(t)=\mathrm{e}^{-t},
    \qquad
    \chi_+(t)=\mathrm{e}^t
    \quad\text{on a neighborhood of }[0,T].
\]

\medskip

\noindent\textbf{Step 1. Inhomogeneous part with zero initial data.}
By Lemma~\ref{lem:schwartz_multiplier}(ii), $\|\chi_- f\|_{\cB^{s/\gmm,s}(T)}
    \lesssim_{T,\gamma,s}
    \|f\|_{\cB^{s/\gmm,s}(T)}$.
Applying Theorem~\ref{thm:finite_mr_unshifted} with \(\lmb=1\), we obtain
\(\tld u\in\cB^{1+s/\gmm,\gmm+s}(T)\) satisfying
\[
    \partial_t\tld u-\Delta^{\gmm/2}\tld u+\tld u=\chi_-f,
    \qquad
    \tld u(0,\cdot)=0,
\]
and
\begin{equation}\label{eq260521_3}
    \|\tld u\|_{\cB^{1+s/\gmm,\gmm+s}(T)}
    \lesssim_{T,\gamma,s}
    \|\chi_-f\|_{\cB^{s/\gmm,s}(T)}
    \lesssim_{T,\gamma,s}
    \|f\|_{\cB^{s/\gmm,s}(T)}.
\end{equation}
Set $u:=\chi_+\tld u$. Since \(\chi_+=\mathrm{e}^t\) and \(\chi_-=\mathrm{e}^{-t}\) on \((0,T)\), we have
\[
    \partial_tu-\Delta^{\gmm/2}u=f,
    \qquad
    u(0,\cdot)=0.
\]
Moreover, by Lemma~\ref{lem:schwartz_multiplier}(ii) and
\eqref{eq260521_3},
\begin{equation}\label{eq260521_5}
    \|u\|_{\cB^{1+s/\gmm,\gmm+s}(T)}
    \lesssim_{T,\gamma,s}
    \|f\|_{\cB^{s/\gmm,s}(T)}.
\end{equation}

\medskip

\noindent\textbf{Step 2. Homogeneous part.}
Define $\tld w(t,\cdot):=\mathrm{e}^{t(\Delta^{\gmm/2}-1)}v_0$. By Corollary~\ref{thm:homogeneous_solution} with \(\lmb=1\), together with
\eqref{260510326} and Proposition~\ref{lem:self-embedding}, we get
\begin{equation}\label{eq260521_4}
    \|\tld w\|_{\cB^{1+s/\gmm,\gmm+s}(T)}
    \lesssim_{T,\gamma,s}
    \left(
    \|v_0\|_{\cB^{\gmm+s}(\bR^d)}
    +
    \|v_0\|_{\cB^0(\bR^d)}
    \right)
    \lesssim_{T,\gamma,s}
    \|v_0\|_{\cB^{\gmm+s}(\bR^d)}.
\end{equation}
Set $w:=\chi_+\tld w$.
Since \(\chi_+=\mathrm{e}^t\) on \((0,T)\), we have
\[
    \partial_tw-\Delta^{\gmm/2}w=0,
    \qquad
    w(0,\cdot)=v_0.
\]
Again, Lemma~\ref{lem:schwartz_multiplier}(ii) and \eqref{eq260521_4}
give
\begin{equation}\label{eq260521_6}
    \|w\|_{\cB^{1+s/\gmm,\gmm+s}(T)}
    \lesssim_{T,\gamma,s}
    \|v_0\|_{\cB^{\gmm+s}(\bR^d)}.
\end{equation}

\medskip

\noindent\textbf{Step 3. Existence and estimate.}
Set $v:=u+w$. Then
\[
    \partial_t v-\Delta^{\gmm/2}v=f,
    \qquad
    v(0,\cdot)=v_0.
\]
The estimate \eqref{260311510} follows from \eqref{eq260521_5} and
\eqref{eq260521_6}.

\medskip

\noindent\textbf{Step 4. Uniqueness.}
Let \(v_1,v_2\in\cB^{1+s/\gmm,\gmm+s}(T)\) be two solutions, and set
\(q:=v_1-v_2\). Then
\[
    \partial_tq-\Delta^{\gmm/2}q=0,
    \qquad
    q(0,\cdot)=0.
\]
By Proposition~\ref{lem:self-embedding}, \(q\in\cB^{1,\gmm}(T)\). Hence
Lemma~\ref{lem:uniqueness_homogeneous_heat} with \(\lmb=0\) gives
\(q=0\) on \([0,T]\times\bR^d\). This completes the proof.
\end{proof}

\subsection{Lower-order drift and potential terms}
\label{subsec:lower-order-terms}

We use Theorem~\ref{thm: inhomo. heat} as the starting point for treating the drift and potential terms by the method of continuity. We first derive an a priori estimate that is uniform along the continuity path. The proof consists of a finite regularity bootstrap followed by a damped base estimate that absorbs the lower-order terms.

\begin{proof}[Proof of Theorem~\ref{thm: intro PDE estimate}]
We set up the method of continuity. Let
\[
    Y:=\cB^{\gmm+s}(\bR^d)\times \cB^{s/\gmm,s}(T),
    \qquad
    \|(v_0,f)\|_Y
    :=
    \|v_0\|_{\cB^{\gmm+s}(\bR^d)}
    +
    \|f\|_{\cB^{s/\gmm,s}(T)}.
\]
Define $Bu
    :=
    (\bfb\cdot\nabla u)\mathbf 1_{\{\gmm>1\}}+cu$.
For \(\tht\in[0,1]\), set 
$$
L_\tht: u\mapsto
    \left(
    u(0,\cdot),
    \partial_tu-\Delta^{\gmm/2}u-\tht Bu
    \right).
    $$
By Propositions~\ref{lem:self-embedding} and \ref{260516347}, $L_\tht:\cB^{1+s/\gmm,\gmm+s}(T)\to Y$ is a bounded linear operator. Moreover, \(L_0\) is bijective by
Theorem~\ref{thm: inhomo. heat}. Therefore, by the method of continuity, it
suffices to prove the following estimate uniformly in \(\tht\in[0,1]\): if
\(v\in\cB^{1+s/\gmm,\gmm+s}(T)\), \(v_0:=v(0,\cdot)\), and $
    f:=
    \partial_tv-\Delta^{\gmm/2}v-\tht Bv
    \in \cB^{s/\gmm,s}(T)$,
then
\begin{equation}
\label{26.05.31.15.56}
    \|v\|_{\cB^{1+s/\gmm,\gmm+s}(T)}
    \lesssim_{T,s,\gmm,\nrm{\bfb},\nrm{c}}
    \|v_0\|_{\cB^{\gmm+s}(\bR^d)}
    +
    \|f\|_{\cB^{s/\gmm,s}(T)}.
\end{equation}

We now prove \eqref{26.05.31.15.56}.

\medskip

\noindent\textbf{Step 1. Bootstrap argument.}

If \(s=0\), we skip Step 1 and proceed directly to the base estimate in
Step 2. Suppose now that \(s>0\). We first put $\gmm':= \gmm-\bfone_{\set{\gmm>1}}\in(0,\gmm]$ and take $n\in\bN\cup\set{0}$ such that
$$
s-(n+1)\gamma'\leq 0< s-n\gamma'\,,
$$
and for $0\leq i\leq n+1$, put $s_i:=s-i\gamma'$.
For each $0\leq i\leq n$, applying Theorem~\ref{thm: inhomo. heat} for $s_i$ in place of $s$ gives 
\begin{align}
\begin{aligned}
\nrm{v}_{\cB^{1+s_i/\gmm,\gmm+s_i}(T)}
&\lesssim_{T,s,\gmm}\,\nrm{v_0}_{\cB^{\gmm+s_i}(\bR^d)}+\nrm{f}_{\cB^{s_i/\gmm,s_i}(T)}+\nrm{\bfb\cdot\nabla v}_{\cB^{s_i/\gmm,s_i}(T)}\bfone_{\set{\gmm>1}}+\nrm{cv}_{\cB^{s_i/\gmm,s_i}(T)}\\
&\lesssim_{s,\gmm}\,\nrm{v_0}_{\cB^{\gmm+s}(\bR^d)}+\nrm{f}_{\cB^{s/\gmm,s}(T)}+\nrm{\bfb\cdot\nabla v}_{\cB^{s_i/\gmm,s_i}(T)}\bfone_{\set{\gmm>1}}+\nrm{cv}_{\cB^{s_i/\gmm,s_i}(T)}
\end{aligned} 
\end{align}
By Propositions~\ref{lem:self-embedding} and \ref{260516347}, we have
\begin{align}\label{eq260528_1}
\begin{alignedat}{2}
\nrm{\bfb\cdot\nabla v}_{\cB^{s_i/\gmm,s_i}(T)}\,&\lesssim_{s,\gmm}&&\,\paren{\sum_{k=1}^d\nrm{b_k}^2_{\cB^{s_i/\gmm,s_i}(T)}}^{1/2}\nrm{v}_{\cB^{(1+s_i)/\gmm,1+s_i}(T)}\\
&\lesssim&&\,\underbrace{\paren{\sum_{k=1}^d\nrm{b_k}^2_{\cB^{s/\gmm,s}(T)}}^{1/2}}_{=:\nrm{\bfb}}\nrm{v}_{\cB^{(1+s_i)/\gmm,1+s_i}(T)}
\end{alignedat}
\end{align} 
and
\begin{align}\label{eq260528_2}
\begin{aligned}
\nrm{cv}_{\cB^{s_i/\gmm,s_i}(T)}\,\lesssim_{s,\gmm}\,\nrm{c}_{\cB^{s_i/\gmm,s_i}(T)}\nrm{v}_{\cB^{s_i/\gmm,s_i}(T)}
\lesssim\,\underbrace{\nrm{c}_{\cB^{s/\gmm,s}(T)}}_{=:\nrm{c}}\nrm{v}_{\cB^{s_i/\gmm,s_i}(T)}.
\end{aligned}
\end{align} 
In sum, we obtain
$$\begin{aligned}
\nrm{v}_{\cB^{1+s_i/\gmm,\gmm+s_i}(T)}
\lesssim_{T,s,\gmm,\nrm{\bfb},\nrm{c}}\,\nrm{v_0}_{\cB^{\gmm+s}(\bR^d)}+\nrm{f}_{\cB^{s/\gmm,s}(T)}+
\nrm{v}_{\cB^{(1+s_i)/\gmm,1+s_i}(T)}\bfone_{\set{\gmm>1}}
+\nrm{v}_{\cB^{s_i/\gmm,s_i}(T)}.
\end{aligned}
$$
By the embedding estimates, the last two terms on the right-hand side satisfy
\begin{equation}\label{eq260528_3}
\begin{aligned}
\nrm{v}_{\cB^{(1+s_i)/\gmm,1+s_i}(T)}\bfone_{\set{\gmm>1}}+\nrm{v}_{\cB^{s_i/\gmm,s_i}(T)}\,&\lesssim\, \begin{cases}
\nrm{v}_{\cB^{(1+s_i)/\gmm,1+s_i}(T)}&(\gmm>1)\\
\nrm{v}_{\cB^{s_i/\gmm,s_i}(T)}&(\gmm\leq1)
\end{cases}\\
&=\,\nrm{v}_{\cB^{1+s_{i+1}/\gmm,\gmm+s_{i+1}}(T)}.
\end{aligned}\end{equation} 
Setting $V_i:=\nrm{v}_{\cB^{1+s_i/\gmm,\gmm+s_i}(T)}$, we can rewrite the preceding estimate as 
$$V_i\q\lesssim_{T,s,\gmm,\nrm{\bfb},\nrm{c}}\q\nrm{v_0}_{\cB^{\gmm+s}(\bR^d)}+\nrm{f}_{\cB^{s/\gmm,s}(T)}+V_{i+1},\qd\text{for each }\q i\in\set{0,\ldots,n}.$$ 
Iterating this inequality from \(i=0\) to \(i=n\), we obtain
\begin{align}\label{eq260528_7}
\begin{aligned}
\nrm{v}_{\cB^{1+s/\gmm,\gmm+s}(T)}
\q
\lesssim_{T,s,\gmm,\nrm{\bfb},\nrm{c}}\qd
\nrm{v_0}_{\cB^{\gmm+s}(\bR^d)}+\nrm{f}_{\cB^{s/\gmm,s}(T)}+\nrm{v}_{\cB^{1+s_{n+1}/\gmm,\gmm+s_{n+1}}(T)}.
\end{aligned}
\end{align}
Now it remains to estimate $\nrm{v}_{\cB^{1+s_{n+1}/\gmm,\gmm+s_{n+1}}(T)}$.
Observe that because of $s_{n+1}\leq0$, we have
\begin{align}\label{260529735}
\nrm{v}_{\cB^{1+s_{n+1}/\gmm,\gmm+s_{n+1}}(T)}\lesssim \nrm{v}_{\cB^{1,\gmm}(T)}\,.
\end{align}
Therefore it remains to estimate $\nrm{v}_{\cB^{1,\gmm}(T)}$.

\medskip

\noindent\textbf{Step 2. Base estimate.} 

Let $\lmb>0$ and put 
$z(t):=\mathrm{e}^{-\lambda t}(v(t)-v_0)$ to get
$$
\begin{cases}
\p_t z-\Delta^{\gamma/2}z+\lambda z= \tld f_1+\tld f_2
& \qd\text{in }(0,T)\times \mathbb{R}^d,
\\
z(0,\cdot)=0
& \qd\text{in }  \mathbb{R}^d,
\end{cases}
$$
where $\tld f_1:=\tht\paren{(\bfb\cdot\nabla z)\bfone_{\set{\gmm>1}}+cz}$ and 
$$
\tld f_2:=
\mathrm{e}^{-\lmb t}\paren{f+\Delta^{\gmm/2}v_0+\tht\paren{(\bfb\cdot\nabla v_0)\bfone_{\set{\gmm>1}}+cv_0}}.$$
We now apply Theorem~\ref{thm:finite_mr_unshifted} with $s=0$ (and $\theta=1/\gamma$ when $\gamma>1$) to $z$ to obtain
\begin{align}\label{eq260528_4}
\begin{aligned}
\|z\|_{\cB^{1,\gmm}(T)}+\lambda \|z\|_{\cB^{0,0}(T)}+\lambda^{1-1/\gamma}\|z\|_{\cB^{1/\gamma,1}(T)}\bfone_{\{\gamma>1\}} \lesssim_{T,s,\gmm}\,\|\tld f_1\|_{\cB^{0,0}(T)}+\|\tld f_2\|_{\cB^{0,0}(T)}\,.
\end{aligned}
\end{align}
In the same way we used Proposition~\ref{260516347} in \eqref{eq260528_1} and \eqref{eq260528_2}, we have 
\begin{align}\label{eq260529_2}
\begin{alignedat}{2}
\|\tld f_1\|_{\cB^{0,0}(T)}\,\lesssim\, \|\bfb\cdot \nabla z\|_{\cB^{0,0}(T)}\bfone_{\{\gamma>1\}}+\|cz\|_{\cB^{0,0}(T)}\lesssim_{s,\gmm,\nrm{\bfb},\nrm{c}}\,\nrm{z}_{\cB^{1/\gamma,1}(T)}\bfone_{\{\gamma>1\}}
+\nrm{z}_{\cB^{0,0}(T)}.
\end{alignedat} 
\end{align} 
Hence, we can choose $\lmb$ sufficiently large so that the contribution of $\|\tld f_1\|_{\cB^{0,0}(T)}$ on the right-hand side of \eqref{eq260528_4} is absorbed into the left-hand side. This yields
\begin{equation}\label{eq260528_5}
\|z\|_{\cB^{1,\gamma}(T)}\,\lesssim_{T,s,\gmm,\nrm{\bfb},\nrm{c}}\,\|\tld f_2\|_{\cB^{0,0}(T)}.
\end{equation}

With \eqref{eq260528_5} in hand, we will estimate $\nrm{v}_{\cB^{1,\gamma}(T)}$. Noting that $v(t)=\mathrm{e}^{\lmb t}z(t)+v_0$ and applying Lemma~\ref{lem:schwartz_multiplier}, we have 
\begin{align}\label{eq260528_6}
\begin{aligned}
\nrm{v}_{\cB^{1,\gamma}(T)}\,\lesssim_{T,s,\gmm}\,\nrm{z}_{\cB^{1,\gamma}(T)}+\nrm{v_0}_{\cB^{\gmm}(\bR^d)}\lesssim_{T,s,\gmm,\nrm{\bfb},\nrm{c}}\,
\|\tld f_2\|_{\cB^{0,0}(T)}+\nrm{v_0}_{\cB^{\gmm}(\bR^d)}.
\end{aligned}
\end{align}
Applying Lemma~\ref{lem:schwartz_multiplier} again, we obtain
\begin{align}
\begin{aligned}
\|\tld f_2\|_{\cB^{0,0}(T)}
&\lesssim_T\,\nrm{f}_{\cB^{0,0}(T)}+
\nrm{\Delta^{\gmm/2}v_0}_{\cB^{0}(\bR^d)}+
\nrm{\bfb\cdot\nabla v_0}_{\cB^{0,0}(T)}\bfone_{\set{\gmm>1}}+\nrm{cv_0}_{\cB^{0,0}(T)}\\
&\lesssim_{T,s,\gmm,\nrm{\bfb},\nrm{c}}\,
\nrm{f}_{\cB^{0,0}(T)}+
\nrm{\Delta^{\gmm/2}v_0}_{\cB^{0}(\bR^d)}+\nrm{v_0}_{\cB^{1}(\bR^d)}\bfone_{\set{\gmm>1}}+\nrm{v_0}_{\cB^0(\bR^d)}\\
&\lesssim\,\nrm{f}_{\cB^{s/\gmm,s}(T)}+\nrm{v_0}_{\cB^{\gmm}(\bR^d)},
\end{aligned}
\end{align} where 
$$
\begin{aligned}
\nrm{\Delta^{\gmm/2}v_0}_{\cB^{0}(\bR^d)}&=2\int_{\bR^d}|\xi|^\gmm|\cF[v_0](\xi)|d\xi\leq\,
2\int_{\bR^d}\paren{1+|\xi|^{\gmm}}|\cF[v_0](\xi)|d\xi=2\nrm{v_0}_{\cB^{\gmm}(\bR^d)}.
\end{aligned}
$$ 
Combining the preceding estimates with \eqref{eq260528_6}, we obtain
$\nrm{v}_{\cB^{1,\gamma}(T)}\,\lesssim_{T,s,\gmm,\nrm{\bfb},\nrm{c}}\,
\nrm{v_0}_{\cB^{\gmm}(\bR^d)}+\nrm{f}_{\cB^{0,0}(T)}$.
Substituting this bound into \eqref{eq260528_7} and using \eqref{260529735}, we conclude that
$$\nrm{v}_{\cB^{1+s/\gmm,\gmm+s}(T)}
\q\lesssim_{T,s,\gmm,\nrm{\bfb},\nrm{c}}\q
\nrm{v_0}_{\cB^{\gmm+s}(\bR^d)}+\nrm{f}_{\cB^{s/\gmm,s}(T)}.$$ This completes the proof of the desired a priori estimate.
\end{proof}

\subsection{Failure of uniform-in-time spatial Fourier--Barron regularity}
\label{subsec:uniform-time-counterexample}
Replacing the anisotropic Fourier--Barron norm $\cB^{\alp,\bt}(T)$ by a uniform-in-time spatial norm does not yield the directly analogous two-derivative gain, even for the one-dimensional heat equation.

The obstruction is produced by a source term whose increasingly high-frequency components are activated on pairwise disjoint and increasingly short time intervals. Consequently, its uniform-in-time spatial Barron norm remains bounded, while the contributions of all frequency components accumulate in the solution at a fixed observation time.
We first prove the case $s=0$ and $\gamma=2$ stated in
Proposition~\ref{prop:failure-uniform-time}, and then explain the extension
to arbitrary $s\geq0$ and $\gamma>0$.

\begin{proof}[Proof of Proposition~\ref{prop:failure-uniform-time}]
Fix $T>0$ and choose a time $t_*\in(0,T)$.
Let $M\in\bN$ be arbitrary.
Choose $R>0$ sufficiently large so that $t_*-\frac{2}{R^2}>0$, and define $r_m:=4^{m-1}R$ for $m=1,\dots,M$.
For each $m$, choose a nonnegative function $\psi_m\in C_c^\infty(\mathbb R)$ such that
$$
    \operatorname{supp}\psi_m\subset [r_m,2r_m],
    \qquad
    \int_{\mathbb R}\psi_m(\xi)\,\dd\xi=1.
$$
The supports of $\psi_m$ are pairwise disjoint. In particular,
$$
    \int_{\mathbb R}|\psi_m(\xi)|\,\mathrm{d}\xi=1,
    \qquad
    \int_{\mathbb R}\xi^2\psi_m(\xi)\,\mathrm{d}\xi
    \ge r_m^2 .
$$
Next, define the time intervals $I_m:=
    \left(
        t_*-\frac{2}{r_m^2},
        t_*-\frac{1}{r_m^2}
    \right)$.
Since $r_{m+1}=4r_m$, these intervals are pairwise disjoint due to $t_*-\frac{1}{r_m^2}
    <
    t_*-\frac{2}{r_{m+1}^2}$.
Choose functions $\eta_m\in C_c^\infty(I_m)$ such that
$$
    0\le \eta_m\le 1,
    \qquad
    \int_{I_m}\eta_m(s)\,\mathrm{d}s
    \ge
    \frac{c_0}{r_m^2},
$$
where $c_0>0$ is independent of $m$.

Define $f$ by prescribing its spatial Fourier transform:
\begin{align}\label{260713408}
    \cF_{d}[f(s,\cdot)](\xi)
    :=
    \sum_{m=1}^M \eta_m(s)\psi_m(\xi).
\end{align}
Since the time supports of $\eta_m$ are pairwise disjoint, at each fixed time $s$, at most one term in the above sum is nonzero. 
Therefore
$$
    \|f(s,\cdot)\|_{\mathcal B^0(\mathbb R)}
    \lesssim_T
    \int_{\mathbb R}|\cF_{d}[f(s,\cdot)](\xi)|\,\mathrm{d}\xi
    \lesssim_T 1
$$
for every $s\in[0,T]$. Hence $    \|f\|_{L^\infty((0,T);\mathcal B^0(\mathbb R))}
    \lesssim_T 1$.

For $u$ in \eqref{260713405}, by \eqref{260713408}, we obtain
$$
   \cF_d[u(t_*,\cdot)](\xi)
    =
    \sum_{m=1}^M
    \psi_m(\xi)
    \int_{I_m}
    \mathrm{e}^{-(t_*-s)|\xi|^2}\eta_m(s)\,\dd s .
$$
Fix $m\in\{1,\dots,M\}$.
If $\xi\in\operatorname{supp}\psi_m\subset [r_m,2r_m]$ and $s\in I_m$, then $\frac{1}{r_m^2}
    \le
    t_*-s
    \le
    \frac{2}{r_m^2}$.
Hence $(t_*-s)|\xi|^2
    \le
    \frac{2}{r_m^2}(2r_m)^2
    =
    8$.
Therefore, $\mathrm{e}^{-(t_*-s)|\xi|^2}
    \ge \mathrm{e}^{-8}$.
Consequently,
$$
    \int_{I_m}
    \mathrm{e}^{-(t_*-s)|\xi|^2}\eta_m(s)\,\mathrm{d}s
    \ge
    \mathrm{e}^{-8}\int_{I_m}\eta_m(s)\,\mathrm{d}s
    \ge
    \frac{c_0\mathrm{e}^{-8}}{r_m^2}.
$$

Since the functions $\psi_m$ are nonnegative and have disjoint supports, we get
\begin{align*}
    \|u(t_*,\cdot)\|_{\mathcal B^2(\mathbb R)}
    &\ge
    \int_{\mathbb R}\xi^2|\cF_d[u(t_*,\cdot)](\xi)|\,\mathrm{d}\xi
    \\
    &=
    \sum_{m=1}^M
    \int_{\operatorname{supp}\psi_m}
    \xi^2\psi_m(\xi)
    \left[
        \int_{I_m}
        \mathrm{e}^{-(t_*-s)|\xi|^2}\eta_m(s)\,\mathrm{d}s
    \right]
    d\xi
    \\
    &\ge
    \sum_{m=1}^M
    \frac{c_0\mathrm{e}^{-8}}{r_m^2}
    \int_{\operatorname{supp}\psi_m}
    \xi^2\psi_m(\xi)\,\mathrm{d}\xi
    \ge
    \sum_{m=1}^M
    \frac{c_0\mathrm{e}^{-8}}{r_m^2}
    r_m^2
    \int_{\mathbb R}\psi_m(\xi)\,\mathrm{d}\xi
    =\,c_0\mathrm{e}^{-8}M .
\end{align*}
Since $u$ is smooth in time with values in
$\mathcal B^2(\mathbb R)$, the same lower bound implies $\|u\|_{L^\infty((0,T);\mathcal B^2(\mathbb R))} \ge cM$ with a constant $c>0$ independent of $M$.

Since $M\in\mathbb N$ was arbitrary, no constant $C(T)$ can make \eqref{eq:false-endpoint-barron} valid. This proves the proposition.
\end{proof}

For a general fractional order $\gamma>0$, repeat the construction with
$r_m=A^{m-1}R$, where $A>\max\{2,2^{1/\gamma}\}$ and $R$ is sufficiently
large that $t_*-2R^{-\gamma}>0$, and set
$$
I_m:=\bigl(t_*-2r_m^{-\gamma},\,t_*-r_m^{-\gamma}\bigr).
$$
The choice of $A$ keeps both the spatial-frequency supports and the time
intervals pairwise disjoint. Choose $\psi_m$ as above and
$\eta_m\in C_c^\infty(I_m)$ with $0\leq\eta_m\leq1$ and
$\int_{I_m}\eta_m(r)\,\mathrm dr\geq c_0r_m^{-\gamma}$, using a fixed
rescaled cutoff so that $c_0>0$ is independent of $m$ and $M$.
Define $f$ by \eqref{260713408} and let
$u(t)=\int_0^t\mathrm e^{(t-r)\Delta^{\gamma/2}}f(r)\,\mathrm dr$.
For $\xi\in[r_m,2r_m]$ and $r\in I_m$, we have
$(t_*-r)|\xi|^\gamma\leq2^{\gamma+1}$. The same estimates therefore give
$$
\|f\|_{L^\infty((0,T);\mathcal B^0(\mathbb R))}\leq2,
\qquad
\|u(t_*,\cdot)\|_{\mathcal B^\gamma(\mathbb R)}
\geq c_0\mathrm e^{-2^{\gamma+1}}M.
$$
Thus the uniform-in-time gain from $\mathcal B^0$ to $\mathcal B^\gamma$
fails for every $\gamma>0$.

To obtain the corresponding counterexample at any regularity level
$s\geq0$, set $J^{-s}:=(1-\Delta)^{-s/2}$.
For every $\beta\geq0$, this operator is an isomorphism from
$\mathcal B^\beta(\mathbb R)$ onto $\mathcal B^{s+\beta}(\mathbb R)$:
the weight
$(1+|\xi|^{s+\beta})(1+|\xi|^2)^{-s/2}$ is comparable to
$1+|\xi|^\beta$, with constants depending only on $s$ and $\beta$.
Moreover, $J^{-s}$ preserves $\mathcal S(\mathbb R)$ and commutes with the
fractional heat semigroup. Applying it to both $f$ and $u$ therefore
transfers the preceding counterexample to the pair of spaces
$\mathcal B^s$ and $\mathcal B^{s+\gamma}$.
Consequently, for every fixed $T>0$, $s\geq0$, and $\gamma>0$, there is no
constant $C(T,s,\gamma)$ such that
$$
\|u\|_{L^\infty((0,T);\mathcal B^{s+\gamma}(\mathbb R))}
\leq C(T,s,\gamma)
\|f\|_{L^\infty((0,T);\mathcal B^s(\mathbb R))}
$$
holds for all $f\in C_c^\infty((0,T);\mathcal S(\mathbb R))$ and their
zero-initial-data mild solutions of $\partial_tu=\Delta^{\gamma/2}u+f$.

\section{Neural Approximation of the Parabolic Solution Class}\label{sec: approximation}
This section derives upper bounds on the error of approximating functions in anisotropic Fourier--Barron spaces by width-$n$ two-layer neural networks.
We then combine the periodic-activation result with Theorem~\ref{thm: intro PDE estimate} to obtain corresponding approximation error bounds for PDE solutions.
The width-error exponent is $1/2$; the constants and the total parameter count are discussed separately below.
Since a target function has different temporal and spatial regularities, the approximation error is measured in mixed Sobolev norms. We first record the Hilbert-valued interpolation and sampling tools used in both activation regimes.\medskip

\paragraph{\textbf{Sobolev spaces and restriction norms.}}
Let $\mathcal H$ be a Hilbert space, $N\geq1$, and $s\geq0$.
The scalar Fourier transform fixed in Subsection~\ref{sec:main} extends by continuity to $L^2(\mathbb R^N;\mathcal H)$.
Writing $\langle\zeta\rangle:=(1+|\zeta|^2)^{1/2}$, we define
$$
H^s(\mathbb R^N;\mathcal H)
:=
\left\{
V\in L^2(\mathbb R^N;\mathcal H):
\langle\zeta\rangle^s\mathcal F_N[V](\zeta)
\in L^2(\mathbb R^N;\mathcal H)
\right\},
$$
with the norm
$$
\|V\|_{H^s(\mathbb R^N;\mathcal H)}^2
:=
\int_{\mathbb R^N}
\langle\zeta\rangle^{2s}
\|\mathcal F_N[V](\zeta)\|_{\mathcal H}^2\,\mathrm d\zeta.
$$
If $\mathcal H$ is separable and $(e_j)_{j\geq1}$ is an orthonormal basis, this transform is given componentwise by
$$
\mathcal{F}_N[V](\zeta)
=
\sum_{j=1}^{\infty}
\mathcal F_N
\bigl[\langle V,e_j\rangle_{\mathcal H}\bigr](\zeta)e_j
$$
in $L^2(\mathbb{R}^N;\mathcal{H})$.
For a domain $D\subset\mathbb R^N$, define
$$
H^s(D;\mathcal H)
:=
\left\{V|_D:V\in H^s(\mathbb R^N;\mathcal H)\right\},
\qquad
\|v\|_{H^s(D;\mathcal H)}
:=
\inf_{\substack{V\in H^s(\mathbb R^N;\mathcal H)\\V|_D=v}}
\|V\|_{H^s(\mathbb R^N;\mathcal H)}.
$$
Here $V|_D=v$ is understood in $L^2(D;\mathcal H)$.
These global and restriction spaces are Hilbert spaces, and they are separable whenever $\mathcal H$ is separable.
We write $H^s(D)$ in the scalar-valued case.

For $d\geq1$ and $\alpha,\beta\geq0$, taking $N=1$, $s=\alpha$, and $\mathcal H=H^\beta(\mathbb R^d)$ gives the mixed norm
$$
\|u\|_{H^\alpha(\mathbb{R};H^\beta(\mathbb{R}^d))}^2
=
\int_{\mathbb R}\int_{\mathbb{R}^d}
\langle \tau \rangle^{2\alpha}
\langle \xi \rangle^{2\beta}
|\mathcal F_{d+1}[u](\tau,\xi)|^2
\,\mathrm{d}\xi\,\mathrm{d}\tau.
$$
For an interval $I\subset\mathbb R$ and a domain $\Omega\subset\mathbb R^d$, we define $H^\alpha(I;H^\beta(\Omega))$ by the same restriction construction with $N=1$, $D=I$, and $\mathcal H=H^\beta(\Omega)$.
The nested full-space norm above has a product Fourier weight, whereas $\mathcal H_{\alpha,\beta}(\Omega)$ in \eqref{eq:intro-mixed-sobolev} carries the square-sum intersection norm.
The product graph space used for the population estimate is introduced in Subsection~\ref{subsec:pde-solution-class}.

The Fourier-based definition allows arbitrary nonnegative real orders $s$.
In the estimates below, however, integer-order Sobolev norms are controlled through weak derivatives. 
We therefore use the following interpolation inequality to pass to fractional orders.

\begin{lem}[Hilbert-valued Sobolev interpolation inequality]
\label{lem:hilbert-valued-interpolation-explicit}
Let $\mathcal H$ be a Hilbert space and let $N\geq1$. 
Given $s\geq0$, set $m:=\lfloor s\rfloor$ and $\theta:=s-m\in[0,1)$.
Then the interpolation inequality
$$
    \|V\|_{H^s(\mathbb R^N;\mathcal H)}
    \leq
    \|V\|_{H^m(\mathbb R^N;\mathcal H)}^{1-\theta}
    \|V\|_{H^{m+1}(\mathbb R^N;\mathcal H)}^\theta
$$
holds for every $V\in H^{m+1}(\mathbb R^N;\mathcal H)$.
\end{lem}

\begin{proof}
If $\theta=0$, the assertion is immediate. 
Suppose $\theta\in(0,1)$.
Since $s=m+\theta$, we have
$$
\begin{aligned}
    \|V\|_{H^s(\mathbb R^N;\mathcal H)}^2
    &=
    \int_{\mathbb R^N}
    \langle\zeta\rangle^{2(m+\theta)}
    \|\mathcal F_N V(\zeta)\|_{\mathcal H}^2
    \,\mathrm d\zeta                                                    \\
    &=
    \int_{\mathbb R^N}
    \left[
        \langle\zeta\rangle^{2m}
        \|\mathcal F_N V(\zeta)\|_{\mathcal H}^2
    \right]^{1-\theta}
    \left[
        \langle\zeta\rangle^{2(m+1)}
        \|\mathcal F_N V(\zeta)\|_{\mathcal H}^2
    \right]^\theta
    \,\mathrm d\zeta .
\end{aligned}
$$
By H\"older's inequality,
$$
    \|V\|_{H^s(\mathbb R^N;\mathcal H)}^2
    \leq
    \|V\|_{H^m(\mathbb R^N;\mathcal H)}^{2(1-\theta)}
    \|V\|_{H^{m+1}(\mathbb R^N;\mathcal H)}^{2\theta}.
$$
Taking square roots gives the result.
\end{proof} 

Since the approximation error is measured on a bounded domain, we also need the corresponding interpolation inequality on bounded Lipschitz domains. 
This follows from the existence of a bounded extension operator.

\begin{corollary}[Interpolation on a bounded Lipschitz domain]
\label{cor:lipschitz-domain-sobolev-interpolation}
Let $\Omega\subset\mathbb R^d$ be a bounded Lipschitz domain, and let
$\mathcal H$ be a Hilbert space. Let $m\in\mathbb N_0$, $\theta\in(0,1)$,
and set $s:=m+\theta$. Then
\[
    \|v\|_{H^{s}(\Omega;\mathcal H)}
    \leq
    C(\Omega,m,\theta)
    \|v\|_{H^m(\Omega;\mathcal H)}^{1-\theta}
    \|v\|_{H^{m+1}(\Omega;\mathcal H)}^\theta
\]
for all $v\in H^{m+1}(\Omega;\mathcal H)$.
\end{corollary}

\begin{proof}
By Stein's extension theorem for bounded Lipschitz domains
\cite[Theorem VI.5]{stein1970singular}, applied in the same way to
Hilbert-space-valued functions, there exists a
linear extension operator
\[
    E_\Omega:H^j(\Omega;\mathcal H)\to H^j(\mathbb R^d;\mathcal H),
    \qquad j=0,1,\ldots,m+1,
\]
such that $(E_\Omega f)|_\Omega=f$ and
\[
    \|E_\Omega f\|_{H^j(\mathbb R^d;\mathcal H)}
    \lesssim_{\Omega,m,\theta}\|f\|_{H^j(\Omega;\mathcal H)},
    \qquad j=0,1,\ldots,m+1.
\]

Fix $v\in H^{m+1}(\Omega;\mathcal H)$. Applying
Lemma~\ref{lem:hilbert-valued-interpolation-explicit} to $E_\Omega v$ on
$\mathbb R^d$, we obtain
\[
    \|E_\Omega v\|_{H^s(\mathbb R^d;\mathcal H)}
    \leq
    \|E_\Omega v\|_{H^m(\mathbb R^d;\mathcal H)}^{1-\theta}
    \|E_\Omega v\|_{H^{m+1}(\mathbb R^d;\mathcal H)}^\theta .
\]
Using the boundedness of $E_\Omega$ on $H^m$ and $H^{m+1}$, we get
\[
    \|E_\Omega v\|_{H^s(\mathbb R^d;\mathcal H)}
    \lesssim_{\Omega,m,\theta}
    \|v\|_{H^m(\Omega;\mathcal H)}^{1-\theta}
    \|v\|_{H^{m+1}(\Omega;\mathcal H)}^\theta .
\]
Since $(E_\Omega v)|_\Omega=v$, we have $
    \|v\|_{H^s(\Omega;\mathcal H)}
    \leq
    \|E_\Omega v\|_{H^s(\mathbb R^d;\mathcal H)}$.
This proves the desired estimate.
\end{proof}
In the main theorem of this section, we construct a neural network that approximates a target function in an anisotropic Barron space. 
The construction uses the Hilbert-space sampling lemma to approximate the expectation of a Hilbert-space-valued random variable by a finite sample average.
We recall the sampling lemma for the sake of completeness.
\begin{lem}[Hilbert-space sampling lemma]
\label{lem:hilbert-sampling}
Let $(E,\mathcal A,\lambda)$ be a probability space and let
$\mathcal{H}$ be a Hilbert space. 
Suppose that $Z:E\to \mathcal{H}$ is measurable and satisfies
$$
    \|Z(\omega)\|_{\mathcal H}\leq M
    \qquad\text{for }\lambda\text{-a.e. }\omega\in E.
$$
Set $\bar Z:=\mathbb E_\lambda[Z]\in\mathcal H$.
Then, for every $n\in\mathbb N$, there exist
$\omega_1,\dots,\omega_n\in E$ such that
\begin{align}\label{260720622}
    \left\|
        \bar Z-\frac1n\sum_{i=1}^n Z(\omega_i)
    \right\|_{\mathcal H}
    \leq
    \frac{M}{\sqrt n}.
\end{align}
\end{lem}

\begin{proof}
Let $Z_1,\dots,Z_n$ be independent random variables with values in
$\mathcal{H}$, each distributed according to the law of $Z$. 
Then $\mathbb E[Z_i]=\bar Z$ for each $i$, and hence $\mathbb E\left[
        \bar Z-\frac1n\sum_{i=1}^n Z_i
    \right]$ equals zero.
We compute the second moment. 
Since the variables $Z_i-\bar Z$ are independent and mean zero in $\mathcal H$, the cross terms vanish:
$$
\begin{aligned}
    \mathbb E
    \left\|
        \frac1n\sum_{i=1}^n (Z_i-\bar Z)
    \right\|_{\mathcal H}^2
    &=
    \frac1{n^2}
    \sum_{i=1}^n
    \mathbb E\|Z_i-\bar Z\|_{\mathcal H}^2 =
    \frac1n
    \mathbb E\|Z-\bar Z\|_{\mathcal H}^2 .
\end{aligned}
$$
Moreover, $\mathbb E\|Z-\bar Z\|_{\mathcal H}^2\leq M^2$.
Therefore,
$$
    \mathbb E
    \left\|
        \bar Z-\frac1n\sum_{i=1}^n Z_i
    \right\|_{\mathcal H}^2
    \leq
    \frac{M^2}{n}.
$$
Set
$$
    X(\omega_1,\dots,\omega_n)
    :=
    \left\|
        \bar Z-\frac1n\sum_{i=1}^n Z(\omega_i)
    \right\|_{\mathcal H}^2 .
$$
The preceding computation gives $\mathbb{E}_{\lambda^{\otimes n}}[X]
    \leq
    \frac{M^2}{n}$.
Hence there exists at least one $(\omega_1,\dots,\omega_n)\in E^n$ such that $X(\omega_1,\dots,\omega_n)
    \leq
    \frac{M^2}{n}$.
Indeed, otherwise we would have
$$
    X(\omega_1,\dots,\omega_n)>\frac{M^2}{n}
    \quad
    \text{for } \lambda^{\otimes n}\text{-a.e. }(\omega_1,\dots,\omega_n),
$$
which would imply $\mathbb E_{\lambda^{\otimes n}}[X]>\frac{M^2}{n}$, a contradiction. 
Therefore \eqref{260720622} holds for these $\omega_1,\,\ldots,\,\omega_n$, which proves the lemma.
\end{proof}

In the proofs of the approximation theorems, we use Fourier inversion to represent the target function as the expectation of a random ridge function. A finite sample average then gives a neural network approximation.
We treat nonperiodic and periodic activation functions separately, beginning with the nonperiodic case.

\subsection{Nonperiodic activation}
\label{26.07.22.11.26}
We first consider general activation functions. 
As outlined above, the approximation argument is based on a Fourier representation of the target function. 
To control the regularity and Fourier integrability in the construction, we assume the polynomial-decay condition in the spirit of \cite{siegel2020approximation}.

\begin{assumption}[$r$]
\label{26.05.17.19.46}
    Let $r\in\mathbb N_0$. We assume that the activation function $\sigma:\mathbb R\to\mathbb R$ belongs to $W_{\mathrm{loc}}^{r,1}(\mathbb{R})$.
    Moreover, there exist $\ell\in\bN$, $c_j,w_j,b_j\in\bR$, $p\in(1,\infty)$ and a constant $C(\bar\sigma,r,p)>0$ such that $\bar{\sigma}(z) = \sum_{j=1}^{\ell}c_j\sigma(w_jz+b_j)$ satisfies $\bar\sigma\not\equiv0$ and 
    \begin{equation}
    \label{26.05.17.19.58}
       |\bar{\sigma}^{(k)}\left(z\right)| \le C(\bar\sigma,r,p)\left(1+\left\vert z\right\vert\right)^{-p} \quad a.e.,
    \end{equation}
    for every integer $0\leq k\leq r$.
\end{assumption}

Since the approximation error is measured in a mixed Sobolev norm, we consider both temporal and spatial Sobolev derivatives of the neural network approximation.
Consequently, the activation function must possess sufficient regularity to accommodate the additional differentiations required in each direction. 
In the present anisotropic setting, the temporal and spatial regularity requirements need not coincide; they are determined by the corresponding Sobolev orders.
In the isotropic case, where the temporal and spatial regularities coincide, the result reduces to the classical Barron-space approximation estimates of \cite{siegel2020approximation}.

\begin{thm}\label{thm:barron-mixed-sobolev-approx}
	Let $T>0$, $\gamma>0$, $\alpha,\beta\geq0$, and let $\Omega\subset\mathbb{R}^d$ be a bounded Lipschitz domain.
    Suppose that $\beta=\alpha\gamma$ and that $\sigma$ satisfies Assumption~\ref{26.05.17.19.46} ($\lceil\max\{\alpha,\beta\}\rceil$).
    Then for any real-valued $u\in \mathcal{B}^{\alpha+\max\{1,1/\gamma\},\beta+\max\{1,\gamma\}}(T)$ and $n\in\bN$, there exists a neural network of the form
    $$
        u_n(t,x)=\sum_{i=1}^{\ell n}c_{i}^n\sigma\left(W_{t,i}^nt+W_{x,i}^n\cdot x+b_i^n\right),
        \qquad (t,x)\in[0,T]\times \Omega,
    $$
    with $c_i^n,W_{t,i}^n,b_i^n\in\bR$ and $W_{x,i}^n\in\bR^d$, such that
    \begin{equation*}\begin{aligned}
        \|u-u_n\|_{H^\alpha((0,T);L^2(\Omega))}
        +
        \|u-u_n\|_{L^2((0,T);H^\beta(\Omega))}\leq
        C(d,T,\Omega,\alpha,\beta,\gamma,\sigma)n^{-1/2}
        \|u\|_{\mathcal{B}^{\alpha+\max\left\{1,1/\gamma\right\},\beta+\max\{1,\gamma\}}(T)} .
    \end{aligned}\end{equation*}
\end{thm}

Theorem~\ref{thm:barron-mixed-sobolev-approx} above will be proved after establishing the following lemma, which identifies the additional Barron regularity required to control the time and space Sobolev approximation errors separately.
Its proof also develops the corresponding componentwise sampling arguments. Theorem~\ref{thm:barron-mixed-sobolev-approx} is then proved by combining these arguments through a single sampling measure that controls both errors simultaneously.

\begin{lem}
\label{lem:barron-mixed-sobolev-approx}
Let $T,\gamma,\alpha,\beta,\Omega$, and $\sigma$ satisfy the parameter, domain, and activation assumptions of Theorem~\ref{thm:barron-mixed-sobolev-approx}.

\begin{enumerate}
    \item For any real-valued $u\in \mathcal{B}^{\alpha+\max\{1,1/\gamma\},\,\beta+1}(T)$, there exists a sequence of two-layer neural networks of the form
$$
    \overline u_n(t,x)
    =
    \sum_{i=1}^{\ell n}
    \overline c_i^n
    \sigma
    \left(
       \overline  W_{t,i}^nt+\overline W_{x,i}^n\cdot x+\overline b_i^n
    \right),
    \qquad (t,x)\in[0,T]\times\Omega,
$$
where $\overline c_i^n,\overline W_{t,i}^n,\overline b_i^n\in\mathbb R$ and $\overline W_{x,i}^n\in\mathbb R^d$, such that
    \begin{equation}
    \label{eq:approx-time-sobolev}
        \|u-\overline u_n\|_{H^\alpha((0,T);L^2(\Omega))}
        \leq
        C(d,T,\Omega,\alpha,\beta,\gamma,\sigma)n^{-1/2}
        \|u\|_{\mathcal{B}^{\alpha+\max\{1,1/\gamma\},\,\beta+1}(T)} .
    \end{equation}

    \item For any real-valued $u\in \mathcal B^{\alpha+1,\,\beta+\max\{1,\gamma\}}(T)$, there exists a sequence of two-layer neural networks of the form 
$$
        \widetilde u_n(t,x)
        =
        \sum_{i=1}^{\ell n}
        \widetilde{c}_i^n
        \sigma
        \left(
            \widetilde{W}_{t,i}^nt+\widetilde{W}_{x,i}^n\cdot x+\widetilde{b}_i^n
        \right),
        \qquad (t,x)\in[0,T]\times\Omega,
$$
    where $\widetilde{c}_i^n,\widetilde{W}_{t,i}^n,\widetilde{b}_i^n\in\mathbb R$ and $\widetilde{W}_{x,i}^n\in\mathbb R^d$, such that
    \begin{equation}
    \label{eq:approx-space-sobolev}
        \|u-\widetilde u_n\|_{L^2((0,T);H^\beta(\Omega))}
        \leq
        C(d,T,\Omega,\alpha,\beta,\gamma,\sigma)n^{-1/2}
        \|u\|_{\mathcal{B}^{\alpha+1,\,\beta+\max\{1,\gamma\}}(T)} .
    \end{equation}
\end{enumerate}
\end{lem}

\begin{proof}
It suffices to construct an approximating network with $n$ hidden units and
the activation from Assumption~\ref{26.05.17.19.46}, namely
\[
\bar{\sigma}(z)=\sum_{j=1}^{\ell}c_j\sigma(w_jz+b_j),
\]
instead of using $\ell n$ hidden units with activation function $\sigma$.
Thus, throughout the proof, we construct an $n$-term
network with activation function $\bar{\sigma}$.

Since $\bar{\sigma}\in L^1(\mathbb{R})$ and $\bar{\sigma}\not\equiv0$, there exists
$a\in\mathbb R\setminus\{0\}$ such that
$$
    \mathcal{F}_1[\bar{\sigma}](a)
    :=
    \int_{\mathbb{R}}\bar{\sigma}(z)\mathrm{e}^{-\mathrm{i}az}\,\mathrm{d}z
    \neq0.
$$
For every $y\in\mathbb{R}$, the change of variables $z=y+b$ gives
$$
    \int_{\mathbb R}\bar{\sigma}(y+b)\mathrm{e}^{-\mathrm{i}ab}\,\mathrm{d}b
    =
    \mathrm{e}^{\mathrm{i}ay}\mathcal{F}_1[\bar{\sigma}](a).
$$
Hence, for every $t,\tau\in\mathbb R$ and $x,\xi\in\mathbb R^d$,
\begin{equation}
\label{eq:plane-wave-ridge-rep}
    \mathrm{e}^{\mathrm{i}(t\tau+x\cdot\xi)}
    =
    \frac{1}{\mathcal{F}_1[\bar{\sigma}](a)}
    \int_{\mathbb{R}}
    \bar{\sigma}\left(
        \frac{t\tau+x\cdot\xi}{a}+b
    \right)\mathrm{e}^{-\mathrm{i}ab}\,\mathrm{d}b .
\end{equation}

    \noindent\textbf{Proof of (1): Temporal Sobolev estimate.}
By the definition of the restriction norm, we can choose an extension to $\mathbb R^{1+d}$, still denoted by $u$, such that
$$
    \|u\|_{\mathcal{B}^{\alpha+\max\left\{1,1/\gamma\right\},\beta+1}(\mathbb R^{1+d})}
    \leq
    2\|u\|_{\mathcal{B}^{\alpha+\max\left\{1,1/\gamma\right\},\beta+1}(T)}.
$$
By Fourier inversion and \eqref{eq:plane-wave-ridge-rep},
\begin{align}
\label{eq:u-ridge-rep}
    u(t,x)
    &=
    C_{\bar{\sigma}}
    \int_{\mathbb R}\int_{\mathbb R^d}\int_{\mathbb R}
    \mathcal{F}_{d+1}[u](\tau,\xi)
    \bar{\sigma}\left(
        \frac{t\tau+x\cdot\xi}{a}+b
    \right)\mathrm{e}^{-\mathrm{i}ab}
    \,\mathrm{d}b\mathrm{d}\xi\mathrm{d}\tau ,
\end{align}
where $C_{\bar{\sigma}}:=(2\pi)^{-(d+1)}\mathcal{F}_{1}[\bar{\sigma}](a)^{-1}$.

Let $R_\Omega:=\sup_{x\in\Omega}|x|<\infty$ and define $A(\tau,\xi):=\frac{T|\tau|+R_\Omega|\xi|}{|a|}$.
For $t\in[0,T]$, $x\in\Omega$, we have
\begin{equation}
\label{eq:h_inequality}
    \left|
        \frac{t\tau+x\cdot\xi}{a}+b
    \right|
    \geq
    \bigl(|b|-A(\tau,\xi)\bigr)_+.
\end{equation}
Set $h(\tau,\xi,b)
    :=
    \left(
        1+\bigl(|b|-A(\tau,\xi)\bigr)_+
    \right)^{-p}$.
We define
$$
    I_\alpha
    :=
    \int_{\mathbb R}\int_{\mathbb R^d}\int_{\mathbb R}
    (1+|\tau|^\alpha)
    h(\tau,\xi,b)
    |\mathcal{F}_{d+1}[u](\tau,\xi)|
    \,\mathrm{d}b\mathrm{d}\xi\mathrm{d}\tau .
$$
If $I_\alpha=0$, then $u=0$ a.e. and the claim is trivial. 
Hence, we assume $I_\alpha>0$.

Since $p>1$,
$$
    \int_{\mathbb R}h(\tau,\xi,b)\,\mathrm{d}b
    \lesssim_{p,a}
    \left(
        1+T|\tau|+R_\Omega|\xi|
    \right).
$$
Therefore
\begin{align*}
    I_\alpha
    &\lesssim_{d,T,\Omega,\alpha,\beta,\gamma,\sigma}
    \int_{\mathbb R}\int_{\mathbb R^d}
    (1+|\tau|^\alpha)
    (1+|\tau|+|\xi|)
    |\mathcal{F}_{d+1}[u](\tau,\xi)|
    \,\mathrm{d}\xi\mathrm{d}\tau .
\end{align*}
If $\alpha>0$, then, since $\beta=\alpha\gamma$, Young's inequality gives
$|\tau|^\alpha|\xi|
    \leq
    |\tau|^{\alpha+1/\gamma}
    +
    |\xi|^{\beta+1}$.
For $\alpha=0$, hence $\beta=0$, the same inequality is immediate.
Also, 
$$
|\tau|+|\tau|^\alpha+|\tau|^{\alpha+1}
    \leq
    3\left(1+|\tau|^{\alpha+\max\{1,1/\gamma\}}\right)
    $$
    and $|\xi|
    \leq
    1+|\xi|^{\beta+1}$.
Hence
$$
    (1+|\tau|^\alpha)(1+|\tau|+|\xi|)
    \lesssim_{d,T,\Omega,\alpha,\beta,\gamma,\sigma}
    \left(
        1+|\tau|^{\alpha+\max\left\{1,1/\gamma\right\}}+|\xi|^{\beta+1}
    \right).
$$
Consequently,
\begin{equation}
\label{eq:I-alpha-bound}
    I_\alpha
    \lesssim_{d,T,\Omega,\alpha,\beta,\gamma,\sigma}
    \|u\|_{\cB^{\alpha+\max\left\{1,1/\gamma\right\},\beta+1}(\mathbb R^{1+d})}
    \lesssim_{d,T,\Omega,\alpha,\beta,\gamma,\sigma}
    \|u\|_{\cB^{\alpha+\max\left\{1,1/\gamma\right\},\beta+1}(T)} .
\end{equation}

Now define a probability measure $\lambda_\alpha$ on
$\mathbb R\times\mathbb R^d\times\mathbb R$ by
$$
    \lambda_\alpha(\mathrm{d}\tau,\mathrm{d}\xi,\mathrm{d}b)
    :=
    \frac1{I_\alpha}
    (1+|\tau|^\alpha)
    h(\tau,\xi,b)
    |\mathcal{F}_{d+1}[u](\tau,\xi)|
    \,\mathrm{d}b\mathrm{d}\xi\mathrm{d}\tau .
$$
Writing $C_{\bar{\sigma}}\mathcal{F}_{d+1}[u](\tau,\xi)\mathrm{e}^{-\mathrm{i}ab}
    =
    |C_{\bar{\sigma}}|
    |\mathcal{F}_{d+1}[u](\tau,\xi)|
    \mathrm{e}^{\mathrm{i}\Theta(\tau,\xi,b)}$ and 
\begin{align}\label{260726338}
\Theta(\tau,\xi,b):=\mathrm{Arg}\cF_{d+1}\left[u\right]\left(\tau,\xi\right) + \mathrm{Arg} C_{\bar{\sigma}}-ab,
\end{align}
we rewrite \eqref{eq:u-ridge-rep} as $u(t,x)
    =
    \mathbb E_{\lambda_\alpha}
    \left[
        Z_{\tau,\xi,b}(t,x)
    \right]$,
where
$$
    Z_{\tau,\xi,b}(t,x)
    :=
    J_\alpha(\tau,\xi,b)
    \mathrm{e}^{\mathrm{i}\Theta(\tau,\xi,b)}
    \bar{\sigma}\left(
        \frac{t\tau+x\cdot\xi}{a} + b
    \right)
$$
and $J_\alpha(\tau,\xi,b)
    :=
    |C_{\bar{\sigma}}|I_\alpha
    (1+|\tau|^\alpha)^{-1}
    h(\tau,\xi,b)^{-1}$.
Here and below, we set $\operatorname{Arg}0:=0$.

We now estimate $Z_{\tau,\xi,b}$ in $H^\alpha((0,T);L^2(\Omega))$.
By Assumption~\ref{26.05.17.19.46}, $\bar\sigma^{(k)}\in L^\infty(\mathbb R)$ for every $0\leq k\leq r$. 
Since $\bar\sigma\in W_{\mathrm{loc}}^{r,1}(\mathbb R)$, it follows that $\bar\sigma\in W^{r,\infty}(\mathbb R)$.
Using \eqref{eq:h_inequality}, the definition of $h$, and
\eqref{26.05.17.19.58}, we obtain, for
$\lambda_\alpha$-a.e. $(\tau,\xi,b)$ and every integer
$0\leq k\leq\lceil\alpha\rceil$,
\begin{equation}
\label{eq:sigma-derivative-h-bound}
    \left|
    \bar{\sigma}^{(k)}
    \left(
        \frac{t\tau+x\cdot\xi}{a}+b
    \right)
    \right|
    \lesssim_{\bar\sigma,r,p}
    h(\tau,\xi,b)
\end{equation}
for a.e. $(t,x)\in(0,T)\times\Omega$.
Therefore, the Sobolev chain rule and
\eqref{eq:sigma-derivative-h-bound} yield, for
$\lambda_\alpha$-a.e. $(\tau,\xi,b)$ and every
$0\leq k\leq\lceil\alpha\rceil$,
\begin{align}
\label{eq:Z-derivative-bound}
    \left|
    \partial_t^k Z_{\tau,\xi,b}(t,x)
    \right|
    &\lesssim_{d,T,\Omega,\alpha,\beta,\gamma,\sigma}
    I_\alpha
    \frac{|\tau|^k}{1+|\tau|^\alpha}
\end{align}
for a.e. $(t,x)\in(0,T)\times\Omega$.
Using the integer-order norm equivalence
\[
H^k((0,T);L^2(\Omega))\simeq W^{k,2}((0,T);L^2(\Omega)),
\]
we integrate \eqref{eq:Z-derivative-bound} over $(0,T)\times\Omega$ and sum
over $k=0,\ldots,\lceil\alpha\rceil$ to obtain
\begin{equation}
\label{eq:Z-Hell-bound}
    \|Z_{\tau,\xi,b}\|_{H^{\lceil\alpha\rceil}((0,T);L^2(\Omega))}
    \lesssim_{d,T,\Omega,\alpha,\beta,\gamma,\sigma}
    I_\alpha
    \frac{1+|\tau|^{\lceil\alpha\rceil}}{1+|\tau|^\alpha}
\end{equation}
for $\lambda_\alpha$-a.e. $(\tau,\xi,b)$.

If $\alpha\in\bN\cup \{0\}$, then $\alpha=\lceil\alpha\rceil$ and \eqref{eq:Z-Hell-bound} directly gives $\|Z_{\tau,\xi,b}\|_{H^\alpha((0,T);L^2(\Omega))}
    \lesssim_{d,T,\Omega,\alpha,\beta,\gamma,\sigma}
    I_\alpha$.
We now consider the case $\alpha\notin\mathbb N_0$ and set $m:=\lfloor\alpha\rfloor$ and $\theta:=\alpha-m\in(0,1)$.
Then $\lfloor\alpha\rfloor+1 = \lceil{\alpha}\rceil$ and $\alpha=(1-\theta)\lfloor\alpha\rfloor+\theta\lceil{\alpha}\rceil$.
By the same argument,
\begin{equation}
\label{eq:Z-Hellminus-bound}
    \|Z_{\tau,\xi,b}\|_{H^{\lfloor\alpha\rfloor}((0,T);L^2(\Omega))}
    \lesssim_{d,T,\Omega,\alpha,\beta,\gamma,\sigma}
    I_\alpha
    \frac{1+|\tau|^{\lfloor\alpha\rfloor}}{1+|\tau|^\alpha}
\end{equation}
for $\lambda_\alpha$-a.e. $(\tau,\xi,b)$.
By Corollary~\ref{cor:lipschitz-domain-sobolev-interpolation}, applied with
$\mathcal H=L^2(\Omega)$,
$$
    \|v\|_{H^\alpha((0,T);L^2(\Omega))}
    \lesssim_{T,\alpha}
    \|v\|_{H^{\lfloor\alpha\rfloor}((0,T);L^2(\Omega))}^{1-\theta}
    \|v\|_{H^{\lceil{\alpha}\rceil}((0,T);L^2(\Omega))}^{\theta}.
$$
Therefore, using \eqref{eq:Z-Hell-bound} and
\eqref{eq:Z-Hellminus-bound},
\begin{align}\label{260726428}
\begin{aligned}
    \|Z_{\tau,\xi,b}\|_{H^\alpha((0,T);L^2(\Omega))}
    &\lesssim_{d,T,\Omega,\alpha,\beta,\gamma,\sigma}
    I_\alpha
    \frac{
        (1+|\tau|^{\lfloor\alpha\rfloor})^{1-\theta}
        (1+|\tau|^{\lceil{\alpha}\rceil})^\theta
    }
    {1+|\tau|^\alpha}\\
    &\lesssim_{d,T,\Omega,\alpha,\beta,\gamma,\sigma}
    I_\alpha 
\end{aligned}
\end{align}
Hence
\begin{equation}
\label{eq:Z-Halpha-uniform-bound}
    \esssup_{\tau,\xi,b}
    \|Z_{\tau,\xi,b}\|_{H^\alpha((0,T);L^2(\Omega))}
    \lesssim_{d,T,\Omega,\alpha,\beta,\gamma,\sigma}
    I_\alpha,
\end{equation}
where the essential supremum is taken with respect to $\lambda_\alpha$.

Applying Lemma~\ref{lem:hilbert-sampling} with $\mathcal H:=H^\alpha((0,T);L^2(\Omega))$ and $\lambda=\lambda_{\alpha}$, and using $u=\mathbb E_{\lambda_\alpha}[Z_{\tau,\xi,b}]$ and \eqref{eq:Z-Halpha-uniform-bound}, we obtain sampling points $(\tau_i,\xi_i,b_i)_{i=1}^n
    \subset
    \mathbb R\times\mathbb R^d\times\mathbb R$
such that
$$
    \left\|
        u-\frac1n\sum_{i=1}^n
        Z_{\tau_i,\xi_i,b_i}
    \right\|_{H^\alpha((0,T);L^2(\Omega))}
    \lesssim_{d,T,\Omega,\alpha,\beta,\gamma,\sigma}
    n^{-1/2}I_\alpha.
$$

Finally, since $u$ is real-valued, we take the real part and define
$$
    \overline u_n(t,x)
    :=\frac{1}{n}\sum_{i=1}^n\Re[Z_{\tau_i,\xi_i,b_i}(t,x)].
$$
Then $\overline u_n$ is a real-valued two-layer neural network of the form
$$
    \overline u_n(t,x)
    =
    \sum_{i=1}^n
    \overline c_i^n
    \bar{\sigma}
    \left(
       \overline  W_{t,i}^nt+\overline W_{x,i}^n\cdot x+\overline b_i^n
    \right),
$$
with
$$
    \overline c_i^n
    :=
    \frac1n
    J_\alpha(\tau_i,\xi_i,b_i)
    \cos\Theta(\tau_i,\xi_i,b_i),
    \qquad
    \overline W_{t,i}^n:=\frac{\tau_i}{a},
    \qquad
    \overline W_{x,i}^n:=\frac{\xi_i}{a},
    \qquad
    \overline b_i^n:=b_i.
$$
By the reduction at the beginning of the proof, $\overline u_n$ can be rewritten as a two-layer neural network with activation function $\sigma$ and $\ell n$ hidden units.
Moreover, using \eqref{eq:I-alpha-bound}, we get
$$
\|u-\overline u_n\|_{H^\alpha((0,T);L^2(\Omega))} \lesssim_{d,T,\Omega,\alpha,\beta,\gamma,\sigma} n^{-1/2} \|u\|_{\mathcal{B}^{\alpha+\max\{1,1/\gamma\},\,\beta+1}(T)}. 
$$ 
This proves the first assertion.

\medskip
\noindent\textbf{Proof of (2): Spatial Sobolev estimate.}
Let $u\in\mathcal{B}^{\alpha+1,\beta+\max\{1,\gamma\}}(T)$.
As in the proof of the first assertion, we choose an extension of $u$ to $\mathbb{R}^{1+d}$, still denoted by $u$, such that
$$
    \|u\|_{\mathcal{B}^{\alpha+1,\beta+\max\{1,\gamma\}}(\mathbb R^{1+d})}
    \leq
    2\|u\|_{\mathcal{B}^{\alpha+1,\beta+\max\{1,\gamma\}}(T)}.
$$
Define
$$
    I_\beta
    :=
    \int_{\mathbb R}\int_{\mathbb R^d}\int_{\mathbb{R}}
    (1+|\xi|^\beta)
    h(\tau,\xi,b)
    |\mathcal{F}_{d+1}[u](\tau,\xi)|
    \,\mathrm{d}b\mathrm{d}\xi\mathrm{d}\tau .
$$
If $I_\beta=0$, then $u=0$ a.e. and there is nothing to prove. 
Hence, we assume $I_\beta>0$.
Since $p>1$, as in the first part,
$$
    I_\beta
    \lesssim_{d,T,\Omega,\alpha,\beta,\gamma,\sigma}
    \int_{\mathbb R}\int_{\mathbb R^d}
    (1+|\xi|^\beta)
    (1+|\tau|+|\xi|)
    |\mathcal{F}_{d+1}[u](\tau,\xi)|
    \,\mathrm{d}\xi\mathrm{d}\tau .
$$
We claim that
$$
    (1+|\xi|^\beta)(1+|\tau|+|\xi|)
    \lesssim_{\alpha,\beta,\gamma}
    \left(
        1+|\tau|^{\alpha+1}
        +|\xi|^{\beta+\max\{1,\gamma\}}
    \right).
$$
Indeed, the only mixed term is $|\tau||\xi|^\beta$. 
Since $\beta=\alpha\gamma$, Young's inequality gives, for $\alpha>0$,
$$
    |\tau||\xi|^\beta
    \leq
    \frac1{\alpha+1}|\tau|^{\alpha+1}
    +
    \frac{\alpha}{\alpha+1}
    |\xi|^{\beta(\alpha+1)/\alpha}
    =
    \frac1{\alpha+1}|\tau|^{\alpha+1}
    +
    \frac{\alpha}{\alpha+1}
    |\xi|^{\beta+\gamma}.
$$
The remaining terms are bounded by $1+|\tau|^{\alpha+1}+|\xi|^{\beta+1}$.
Thus the claim follows from $\max\{\beta+1,\beta+\gamma\}
    =
    \beta+\max\{1,\gamma\}$.
The case $\alpha=0$, hence $\beta=0$, is immediate.
Consequently,
\begin{equation}
\label{eq:I-beta-bound}
    I_\beta
    \lesssim_{d,T,\Omega,\alpha,\beta,\gamma,\sigma}
    \|u\|_{\cB^{\alpha+1,\,\beta+\max\{1,\gamma\}}(\mathbb R^{1+d})}
    \lesssim_{d,T,\Omega,\alpha,\beta,\gamma,\sigma}
    \|u\|_{\cB^{\alpha+1,\,\beta+\max\{1,\gamma\}}(T)} .
\end{equation}

Now define a probability measure $\lambda_\beta$ on
$\mathbb R\times\mathbb R^d\times\mathbb R$ by
$$
    \lambda_\beta(\mathrm{d}\tau,\mathrm{d}\xi,\mathrm{d}b)
    :=
    \frac1{I_\beta}
    (1+|\xi|^\beta)
    h(\tau,\xi,b)
    |\mathcal{F}_{d+1}[u](\tau,\xi)|
    \,\mathrm{d}b\mathrm{d}\xi\mathrm{d}\tau .
$$
As in the proof of the first assertion, we may write $u(t,x)=
    \mathbb E_{\lambda_\beta}
    \left[
        Z_{\tau,\xi,b}(t,x)
    \right]$,
where
$$
    Z_{\tau,\xi,b}(t,x)
    :=
    J_\beta(\tau,\xi,b)
    \mathrm{e}^{\mathrm{i}\Theta(\tau,\xi,b)}
    \bar{\sigma}
    \left(
        \frac{t\tau+x\cdot\xi}{a}+b
    \right),
$$
$J_\beta(\tau,\xi,b)
    :=
    |C_{\bar{\sigma}}|I_\beta
    (1+|\xi|^\beta)^{-1}
    h(\tau,\xi,b)^{-1}$, and $\Theta(\tau,\xi,b)$ is defined by \eqref{260726338}
with the present extension $u$.
    
We repeat the argument used in the proof of the first assertion, replacing the weight $1+|\tau|^\alpha$ by $1+|\xi|^\beta$ and the time derivatives by the spatial derivatives. 
Since $\bar{\sigma}\in W^{r,\infty}(\mathbb R)$, the Sobolev chain rule gives, for $\lambda_\beta$-a.e. $(\tau,\xi,b)$ and every multi-index $\nu$ with $|\nu|\leq\lceil\beta\rceil$,
\begin{align}\label{eq:Z-derivative-bound_xx}
    |D_x^\nu Z_{\tau,\xi,b}(t,x)|
    \lesssim_{d,T,\Omega,\alpha,\beta,\gamma,\sigma}
    I_\beta
    \frac{|\xi|^{|\nu|}}{1+|\xi|^\beta}
\end{align}
for a.e. $(t,x)\in(0,T)\times\Omega$.
Using the equality $H^\ell(\Omega)=W^{\ell,2}(\Omega)$ with equivalent norms for nonnegative integers $\ell$, together with
$$
\sum_{|\nu|\leq\ell}|\xi^\nu|^2
\leq
\sum_{j=0}^{\ell}|\xi|^{2j}
\lesssim_{\ell}
    (1+|\xi|^{2\ell}),
$$
we obtain, for every integer $0\leq \ell\leq\lceil\beta\rceil$ and for a.e. $t\in(0,T)$,
$$
\|Z_{\tau,\xi,b}(t,\cdot)\|_{H^\ell(\Omega)}
\lesssim_{d,T,\Omega,\alpha,\beta,\gamma,\sigma} I_\beta
\frac{1+|\xi|^\ell}{1+|\xi|^\beta}.
$$
Interpolating between the orders $\lfloor\beta\rfloor$ and
$\lceil\beta\rceil$, and then integrating in time, yields 
\begin{equation}
\label{eq:Z-L2Hbeta-uniform}
    \operatorname*{ess\,sup}_{\tau,\xi,b}
    \|Z_{\tau,\xi,b}\|_{L^2((0,T);H^\beta(\Omega))}
    \lesssim_{d,T,\Omega,\alpha,\beta,\gamma,\sigma}
    I_\beta,
\end{equation}
where the essential supremum is taken with respect to $\lambda_\beta$.

    Applying Lemma~\ref{lem:hilbert-sampling} in $\mathcal H=L^2((0,T);H^\beta(\Omega))$, with $\lambda=\lambda_\beta$, and using \eqref{eq:Z-L2Hbeta-uniform}, we obtain points $(\tau_i,\xi_i,b_i)_{i=1}^n
    \subset \mathbb R\times\mathbb R^d\times\mathbb R$
such that
$$
    \left\|
        u-\frac1n\sum_{i=1}^n
        Z_{\tau_i,\xi_i,b_i}
    \right\|_{L^2((0,T);H^\beta(\Omega))}
    \lesssim_{d,T,\Omega,\alpha,\beta,\gamma,\sigma}
    n^{-1/2} I_\beta
    \lesssim_{d,T,\Omega,\alpha,\beta,\gamma,\sigma}
    n^{-1/2}
    \|u\|_{\cB^{\alpha+1,\,\beta+\max\{1,\gamma\}}(T)} .
$$
The last inequality follows from \eqref{eq:I-beta-bound}.

Finally, since $u$ is real-valued and the real-part map is a contraction on $L^2((0,T);H^\beta(\Omega))$, we set
$$
    \widetilde u_n(t,x)
    :=
    \Re\left(
    \frac1n\sum_{i=1}^n Z_{\tau_i,\xi_i,b_i}(t,x)
    \right)
    =
    \sum_{i=1}^n
    \widetilde{c}_i^n\bar{\sigma}(\widetilde{W}_{t,i}^nt+\widetilde{W}_{x,i}^n\cdot x+\widetilde{b}_i^n),
$$
where
$$
    \widetilde c_i^n
    :=
    \frac1n
    J_\beta(\tau_i,\xi_i,b_i)
    \cos\Theta(\tau_i,\xi_i,b_i),
    \qquad
    \widetilde W_{t,i}^n:=\frac{\tau_i}{a},
    \qquad
    \widetilde W_{x,i}^n:=\frac{\xi_i}{a},
    \qquad
    \widetilde b_i^n:=b_i.
$$
By the reduction at the beginning of the proof, $\widetilde u_n$ can be rewritten as a two-layer neural network with activation function $\sigma$ and $\ell n$ hidden units. Moreover,
$$
    \|u-\widetilde u_n\|_{L^2((0,T);H^\beta(\Omega))}
    \lesssim_{d,T,\Omega,\alpha,\beta,\gamma,\sigma} n^{-1/2}
    \|u\|_{\mathcal{B}^{\alpha+1,\,\beta+\max\{1,\gamma\}}(T)}.
$$
This proves the second assertion.
\end{proof}

\medskip

\begin{proof}[Proof of Theorem~\ref{thm:barron-mixed-sobolev-approx}]
By the definition of the restriction norm, we choose an extension of $u$
to $\mathbb R^{1+d}$, still denoted by $u$, such that
$$
    \|u\|_{\mathcal B^{          \alpha+\max\{1,1/\gamma\},\beta+\max\{1,\gamma\}}(\mathbb R^{1+d})}
    \leq
    2\|u\|_{\mathcal B^{\alpha+\max\{1,1/\gamma\},
            \beta+\max\{1,\gamma\}}(T)}.
$$
Let $a$, $C_{\bar\sigma}$, $A$, and $h$ be as in the proof of Lemma~\ref{lem:barron-mixed-sobolev-approx}, and define $\Theta$ by \eqref{260726338} using the present
extension $u$.
Set
$$
\begin{aligned}
    I
    &:=
    \int_{\mathbb R}
    \int_{\mathbb R^d}
    \int_{\mathbb R}
    \left(
        2+|\tau|^\alpha+|\xi|^\beta
    \right)
    h(\tau,\xi,b)
    |\mathcal F_{d+1}[u](\tau,\xi)|
    \dd b\dd \xi\dd \tau.
\end{aligned}
$$
If $I=0$, then $u=0$ a.e., and the claim is trivial.
Hence, assume that $I>0$.

The estimates leading to \eqref{eq:I-alpha-bound} and
\eqref{eq:I-beta-bound} give
\begin{equation}
\label{eq:I-mixed-bound}
    I
    \lesssim_{d,T,\Omega,\alpha,\beta,\gamma,\sigma}
    \|u\|_{
        \mathcal B^{
            \alpha+\max\{1,1/\gamma\},
            \,
            \beta+\max\{1,\gamma\}
        }(T)
    }.
\end{equation}
Define a probability measure $\lambda$ on
$\mathbb R\times\mathbb R^d\times\mathbb R$ by
$$
    \lambda(\mathrm d\tau,\mathrm d\xi,\mathrm db)
    :=
    \frac1I
    \left(
        2+|\tau|^\alpha+|\xi|^\beta
    \right)
    h(\tau,\xi,b)
    |\mathcal F_{d+1}[u](\tau,\xi)|
    \dd b\dd \xi\dd \tau.
$$
Set
$$
    Z_{\tau,\xi,b}(t,x)
    :=
    J(\tau,\xi,b)
    \mathrm e^{\mathrm i\Theta(\tau,\xi,b)}
    \bar{\sigma}
    \left(
        \frac{t\tau+x\cdot\xi}{a}+b
    \right),
$$
where $J(\tau,\xi,b)
    :=
    |C_{\bar{\sigma}}|I
    \left(
        2+|\tau|^\alpha+|\xi|^\beta
    \right)^{-1}
    h(\tau,\xi,b)^{-1}$, and $\Theta(\tau,\xi,b)$ is defined by \eqref{260726338} with the present
extension $u$. 
The ridge representation then yields $
    u(t,x)
    =
    \mathbb E_\lambda
    \left[
        Z_{\tau,\xi,b}(t,x)
    \right]$.

Consider the Hilbert space
 \begin{align}\label{260726508}
\begin{gathered}
    \mathcal H
    :=
    H^\alpha((0,T);L^2(\Omega))
    \cap
    L^2((0,T);H^\beta(\Omega)),\\
    \|v\|_{\mathcal H}^2
    :=
    \|v\|_{H^\alpha((0,T);L^2(\Omega))}^2
    +
    \|v\|_{L^2((0,T);H^\beta(\Omega))}^2.
\end{gathered}
 \end{align}
Since $2+|\tau|^\alpha+|\xi|^\beta
    \ge
    1+|\tau|^\alpha$ and $2+|\tau|^\alpha+|\xi|^\beta
    \ge
    1+|\xi|^\beta$, 
    the derivative and interpolation estimates used in the proof of
Lemma~\ref{lem:barron-mixed-sobolev-approx} (see \eqref{eq:Z-derivative-bound}, \eqref{260726428}, and \eqref{eq:Z-derivative-bound_xx}) yield 
\begin{equation}
\label{eq:Z-mixed-uniform-bound}
    \operatorname*{ess\,sup}_{\tau,\xi,b}
    \|Z_{\tau,\xi,b}\|_{\mathcal H}
    \lesssim_{d,T,\Omega,\alpha,\beta,\gamma,\sigma} I,
\end{equation}
where the essential supremum is taken with respect to $\lambda$.

Therefore, by Lemma~\ref{lem:hilbert-sampling}, there exist $(\tau_i,\xi_i,b_i)_{i=1}^n
    \subset
    \mathbb R\times\mathbb R^d\times\mathbb R$ such that
$$
    \left\|
        u-\frac1n\sum_{i=1}^n Z_{\tau_i,\xi_i,b_i}
    \right\|_{\mathcal H}
    \lesssim_{d,T,\Omega,\alpha,\beta,\gamma,\sigma}
    n^{-1/2} I .
$$

Since $u$ is real-valued, we take the real part of the sampled approximation.
The real-part map is a contraction in both Sobolev norms, and hence the same
estimate holds for the real-valued two-layer neural network
$$
    u_n(t,x)
    :=\sum_{i=1}^n\Re\left[\frac 1 n Z_{\tau_i,\xi_i,b_i}\right]=
    \sum_{i=1}^n
    c_i^n\bar \sigma(W_{t,i}^nt+W_{x,i}^n\cdot x+b_i^n),
$$
where
$$
     c_i^n
    :=
    \frac1n
    J(\tau_i,\xi_i,b_i)
    \cos\Theta(\tau_i,\xi_i,b_i),
    \qquad
     W_{t,i}^n:=\frac{\tau_i}{a},
    \qquad
     W_{x,i}^n:=\frac{\xi_i}{a},
    \qquad
     b_i^n:=b_i.
$$
By the definition of $\bar{\sigma}$, the same function $u_n$ can be
rewritten as a two-layer neural network with activation function $\sigma$
and $\ell n$ hidden units. Moreover, by the definition of $\|\cdot\|_{\cH}$ and \eqref{eq:I-mixed-bound},
\begin{align*}
    &\|u- u_n\|_{H^\alpha((0,T);L^2(\Omega))}+\|u- u_n\|_{L^2((0,T);H^\beta(\Omega))}
    \\
    \lesssim_{d,T,\Omega,\alpha,\beta,\gamma,\sigma}\,&n^{-1/2}
    \|u\|_{\mathcal{B}^{\alpha+\max\{1,1/\gamma\},\,\beta+\max\{1,\gamma\}}(T)}.
\end{align*}
This proves the theorem.
\end{proof}

\subsection{Periodic activation}
\label{26.07.22.11.27}
We now prove Theorem~\ref{thm:intro-periodic-approximation}.
The argument follows the Fourier-series approach of \cite{siegel2020approximation}.
In the periodic case, a nonzero Fourier-series coefficient provides a direct ridge representation.
Consequently, no polynomial-decay condition on the activation and no additional Barron regularity of the target are required.

\begin{proof}[Proof of Theorem~\ref{thm:intro-periodic-approximation}]
By the definition of the restricted Barron norm, choose an extension of $u$ to $\mathbb R^{1+d}$, still denoted by $u$, such that $
    \|u\|_{\mathcal{B}^{\alpha,\beta}(\mathbb R^{1+d})}
    \leq
    2\|u\|_{\mathcal{B}^{\alpha,\beta}(T)}$.
Let $P>0$ be a period of $\sigma$.
Since $\sigma$ is nonconstant, there exists
$k_0\in\big(\frac{2\pi}{P}\mathbb Z\big)\setminus\{0\}$ such that
$$
    a_{k_0}
    :=
    \frac1P
    \int_0^P
    \sigma(b)\mathrm e^{-\mathrm ik_0b}
    \dd b
    \neq0.
$$
For every $y\in\mathbb R$, periodicity gives
$$
\mathrm e^{\mathrm ik_0 y}
    =
    \frac1{Pa_{k_0}}
    \int_0^P
    \sigma
    \left(y+b
    \right)
    \mathrm e^{-\mathrm i{k_0}b}
    \dd b.
$$
Taking $y=(t\tau+x\cdot\xi)/k_0$ and applying Fourier inversion, we obtain
\begin{align}
\label{eq:periodic-ridge-rep-u}
        u(t,x)
    &=
    C_\sigma
    \int_{\mathbb R}
    \int_{\mathbb R^d}
    \int_0^P
    \mathcal F_{d+1}[u](\tau,\xi)
    \sigma
    \left(
        \frac{t\tau+x\cdot\xi}{{k_0}}+b
    \right)
    \mathrm e^{-\mathrm i{k_0}b}
    \dd b\dd\xi\dd\tau,
\end{align}
where $
C_\sigma
    :=
    (2\pi)^{-(d+1)}(Pa_{k_0})^{-1}$.

Set $w(\tau,\xi):=1+|\tau|^\alpha+|\xi|^\beta$ and $I:=
    \int_{\mathbb R}\int_{\mathbb R^d}
    w(\tau,\xi)
    |\mathcal F_{d+1}[u](\tau,\xi)|
    \,\mathrm d\xi\,\mathrm d\tau$.
Then
\begin{equation}
\label{eq:periodic-I-bound}
    I
    \leq
    2
    \|u\|_{\mathcal{B}^{\alpha,\beta}(T)} .
\end{equation}
If $I=0$, then $u=0$ a.e., and the assertion is trivial. 
Hence, assume that $I>0$.

Define a probability measure $\lambda$ on
$\mathbb R\times\mathbb R^d\times[0,P]$ by
$$
    \lambda(\mathrm d\tau,\mathrm d\xi,\mathrm db)
    :=
    \frac1{P I}
    w(\tau,\xi)
    |\mathcal F_{d+1}[u](\tau,\xi)|
    \,\mathrm db\,\mathrm d\xi\,\mathrm d\tau .
$$
Define
\begin{align*}
    \begin{gathered}
    J(\tau,\xi):=P|C_\sigma| w(\tau,\xi)^{-1}I,\qquad \Theta(\tau,\xi,b)
    :=
    \operatorname{Arg}\mathcal F_{d+1}[u](\tau,\xi)
    +
    \operatorname{Arg}C_\sigma
    -
    {k_0}b,\\
     Z_{\tau,\xi,b}(t,x)
    :=
    J(\tau,\xi)
    \mathrm e^{\mathrm i\Theta(\tau,\xi,b)}
    \sigma
    \left(
        \frac{t\tau+x\cdot\xi}{{k_0}}+b
    \right).
    \end{gathered}
\end{align*}
Then by \eqref{eq:periodic-ridge-rep-u}, $u=\mathbb E_\lambda[Z_{\tau,\xi,b}]$.

Consider the Hilbert space $\mathcal H$ in \eqref{260726508}.
Since
$\sigma\in W^{\lceil\max\{\alpha,\beta\}\rceil,\infty}(\mathbb R)$, the Sobolev chain rule gives, for $\lambda$-a.e.
$(\tau,\xi,b)$ and for all $0\le j\le\lceil\alpha\rceil$ and $|\nu|\le\lceil\beta\rceil$,
$$
|\partial_t^j Z_{\tau,\xi,b}(t,x)|
\lesssim_{\alpha,\beta,\sigma}I\frac{|\tau|^j}{w(\tau,\xi)},
\qquad
|D_x^\nu Z_{\tau,\xi,b}(t,x)|
\lesssim_{\alpha,\beta,\sigma}I\frac{|\xi|^{|\nu|}}{w(\tau,\xi)}
$$
for a.e. $(t,x)\in(0,T)\times\Omega$.

Since $w(\tau,\xi)\ge 1+|\tau|^\alpha$ and $w(\tau,\xi)\ge 1+|\xi|^\beta$, the derivative and interpolation arguments used in the proof of
Lemma~\ref{lem:barron-mixed-sobolev-approx} (see \eqref{eq:Z-derivative-bound}, \eqref{260726428}, and \eqref{eq:Z-derivative-bound_xx}) yield
\begin{equation}
\label{eq:periodic-atom-bound}
    \operatorname*{ess\,sup}_{\tau,\xi,b}
    \|Z_{\tau,\xi,b}\|_{\mathcal H}
    \lesssim_{T,\Omega,\alpha,\beta,\sigma} I,
\end{equation}
where the essential supremum is taken with respect to $\lambda$.

Applying Lemma~\ref{lem:hilbert-sampling} in $\mathcal H$, we obtain
points $(\tau_i,\xi_i,b_i)_{i=1}^n
    \subset
    \mathbb R\times\mathbb R^d\times[0,P]$
such that
$$
    \left\|
        u-\frac1n\sum_{i=1}^n
        Z_{\tau_i,\xi_i,b_i}
    \right\|_{\mathcal H}
    \lesssim_{T,\Omega,\alpha,\beta,\sigma} n^{-1/2}I.
$$
Since $u$ is real-valued, we take the real part of the sampled approximation.
The real-part map is a contraction in both Sobolev norms, and hence the same
estimate holds for the real-valued two-layer neural network
$$
    u_n(t,x)
    :=\sum_{i=1}^n\Re\left[\frac 1 n Z_{\tau_i,\xi_i,b_i}\right]=
    \sum_{i=1}^n
    c_i^n \sigma(W_{t,i}^nt+W_{x,i}^n\cdot x+b_i^n),
$$
where
$$
     c_i^n
    :=
    \frac1n
    J(\tau_i,\xi_i)
    \cos\Theta(\tau_i,\xi_i,b_i),
    \qquad
     W_{t,i}^n:=\frac{\tau_i}{k_0},
    \qquad
     W_{x,i}^n:=\frac{\xi_i}{k_0},
    \qquad
     b_i^n:=b_i.
$$
By the definition of $\|\cdot\|_{\cH}$ and \eqref{eq:periodic-I-bound},
\begin{align*}
    &\|u- u_n\|_{H^\alpha((0,T);L^2(\Omega))}+\|u- u_n\|_{L^2((0,T);H^\beta(\Omega))}
    \lesssim_{T,\Omega,\alpha,\beta,\sigma} n^{-1/2}
    \|u\|_{\mathcal{B}^{\alpha,\,\beta}(T)}.
\end{align*}
This proves the theorem.
\end{proof}

\subsection{The population graph-space estimate}
\label{subsec:pde-solution-class}

Throughout this subsection, we use the notation introduced in \eqref{eq:pinn-data-ball}--\eqref{eq:population-pinn-loss-q} and set $\alpha:=1+s/\gamma$ and $\beta:=\gamma+s$.
We write $Q_T:=(0,T)\times\Omega$ and denote the parameter space by
$\Theta:=\mathbb R\times\mathbb R^d\times[0,2\pi]$.
For brevity, write
$$
\mathcal A_{\bfb,c}w
:=
\partial_tw-\Delta^{\gamma/2}w
-
(\bfb\cdot\nabla w)1_{\{\gamma>1\}}
-cw.
$$

Viewed as a function on $\mathbb{R}^{1+d}$, a nonzero finite cosine network has an atomic Fourier transform. 
Since the spectral Barron space used in Theorem~\ref{thm: intro PDE estimate} is defined through a weighted $L^1$ Fourier density, such a network need not belong to that space. 
Consequently, the population estimate cannot be obtained by directly applying the Barron maximal-regularity estimate to the network error. 
Instead, we repeat the Fourier-sampling argument in a product Hilbert space that records the solution, its image under the PDE operator, and its initial trace simultaneously.

For $q\in[0,\gamma+s]$, define
\begin{equation}
\label{eq:pinn-graph-space-q}
\mathbb G_{s,q}(\Omega)
:=
\mathcal H_{1+\frac{s}{\gamma},\gamma+s}(\Omega)
\times
L^2\bigl((0,T)\times\Omega\bigr)
\times
H^q(\Omega),
\end{equation}
equipped with the Hilbert norm
\begin{equation}
\label{eq:pinn-graph-norm-q}
\begin{aligned}
\|(z,r,z_0)\|_{\mathbb G_{s,q}(\Omega)}^2
:={}&
\|z\|_{\mathcal H_{1+\frac{s}{\gamma},\gamma+s}(\Omega)}^2
+
\|r\|_{L^2((0,T)\times\Omega)}^2+
\|z_0\|_{H^q(\Omega)}^2.
\end{aligned}
\end{equation}

\begin{lem}[Graph-norm estimate for cosine ridge atoms]
\label{lem:cosine-ridge-graph-norm}
Suppose that the assumptions in Corollary~\ref{cor:pde-solution-class} hold, and let
$$
Z(t,x)
=
a\cos(\omega t+\xi\cdot x+\theta),
\qquad
a,\omega,\theta\in\mathbb R,
\quad
\xi\in\mathbb R^d.
$$
Then
\begin{equation}
\label{eq:cosine-graph-bound}
\begin{aligned}
&\left\|
\left(
Z,
\mathcal A_{\bfb,c}Z,
Z(0,\cdot)
\right)
\right\|_{\mathbb G_{s,q}(\Omega)}\\
&\quad \leq
C(T,\Omega,d,s,\gamma,q,\|\bfb\|,\|c\|)
|a|
\left(
1+|\omega|^{1+\frac{s}{\gamma}}
+|\xi|^{\gamma+s}
\right).
\end{aligned}
\end{equation}

Moreover, the parameter-to-graph map
\begin{align}
\label{eq:cosine-ridge-parameter-map}
\mathrm{P}:
(a,\omega,\xi,\theta)\in\mathbb R\times\mathbb R\times\mathbb R^d\times\mathbb R\longmapsto
\left(
Z,
\mathcal A_{\bfb,c}Z,
Z(0,\cdot)
\right)\in\mathbb{G}_{s,q}(\Omega),
\end{align}
is continuous.
\end{lem}

\begin{proof}
Set
$$
W(\omega,\xi)
:=
1+|\omega|^{1+\frac{s}{\gamma}}+|\xi|^{\gamma+s},\quad\text{and}\quad
\phi(t,x):=\omega t+\xi\cdot x+\theta.
$$

By the Fourier-multiplier convention in \eqref{eq: intro target PDE},
$$
\partial_tZ
=
-a\omega\sin\phi,
\qquad
\nabla_xZ
=
-a\xi\sin\phi,
\qquad
\Delta_x^{\gamma/2}Z
=
-a|\xi|^\gamma\cos\phi.
$$
Consequently,
\begin{equation}
\label{eq:cosine-ridge-residual}
\mathcal A_{\bfb,c}Z
=
a
\left[
\left(
\bfb\cdot\xi\,1_{\{\gamma>1\}}
-\omega
\right)\sin\phi
+
\left(
|\xi|^\gamma-c
\right)\cos\phi
\right].
\end{equation}

Since $1+\frac{s}{\gamma}\geq1$ and $\gamma+s\geq\gamma$, we have $
|\omega|
\leq
1+|\omega|^{1+\frac{s}{\gamma}}
\leq
W(\omega,\xi)$ and $|\xi|^\gamma
\leq
1+|\xi|^{\gamma+s}
\leq
W(\omega,\xi)$.
Moreover, if the drift term is present, then $\gamma>1$ and hence $\beta\geq\gamma>1$. 
Therefore, $
|\xi|
\leq
1+|\xi|^{\gamma+s}
\leq
W(\omega,\xi)$.
It follows from \eqref{eq:cosine-ridge-residual} and Proposition \ref{prop:B-to-Ctx} that, for almost every $(t,x)\in(0,T)\times\Omega$, $
\left|
\mathcal A_{\bfb,c}Z(t,x)
\right|
\leq
\left(
2+\|\bfb\|+\|c\|
\right)
|a|W(\omega,\xi)$.
Thus,
\begin{equation}
\label{eq:cosine-residual-bound}
\|\mathcal A_{\bfb,c}Z\|_{L^2((0,T)\times\Omega)}
\leq
(T|\Omega|)^{1/2}
\left(
2+\|\bfb\|+\|c\|
\right)
|a|W(\omega,\xi).
\end{equation}

We next estimate the mixed Sobolev norm. Choose
$\chi\in C_c^\infty(\mathbb R)$ and
$\eta\in C_c^\infty(\mathbb R^d)$ such that
$$
\chi=1
\quad\text{on a neighborhood of }[0,T],
\qquad
\eta=1
\quad\text{on a neighborhood of }\overline\Omega.
$$
Then
$$
\widetilde Z(t,x)
:=
a\chi(t)\eta(x)
\cos(\omega t+\xi\cdot x+\theta)
$$
is an extension of $Z$ from $(0,T)\times\Omega$ to
$\mathbb R\times\mathbb R^d$.

For every $r\geq0$, the standard modulation estimates give
$$
\|\chi(\cdot)\mathrm{e}^{\pm \mathrm{i}\omega(\cdot)}\|_{H^r(\mathbb R)}
\lesssim_{\chi,r}
    (1+|\omega|^r),\quad \text{and}\quad 
\|\eta(\cdot)\mathrm{e}^{\pm \mathrm{i}\xi\cdot(\cdot)}\|_{H^r(\mathbb R^d)}
\lesssim_{\eta,r}
    (1+|\xi|^r).
$$
Writing the cosine as a sum of two complex exponentials therefore yields
$$
\|\widetilde Z\|_{H^\alpha(\mathbb R;L^2(\mathbb R^d))}
\lesssim_{T,\Omega,d,s,\gamma,q,\|\bfb\|,\|c\|}
    |a|(1+|\omega|^{\alpha}),\quad \text{and}\quad
\|\widetilde Z\|_{L^2(\mathbb R;H^\beta(\mathbb R^d))}
\lesssim_{T,\Omega,d,s,\gamma,q,\|\bfb\|,\|c\|}
    |a|(1+|\xi|^\beta).
$$
By the restriction-space definition of $\mathcal H_{\alpha,\beta}(\Omega)$,
\begin{align}
\|Z\|_{\mathcal H_{\alpha,\beta}(\Omega)}
&\lesssim_{T,\Omega,d,s,\gamma,q,\|\bfb\|,\|c\|}
    |a|
\left(
\langle\omega\rangle^\alpha
+
\langle\xi\rangle^\beta
\right)\lesssim_{T,\Omega,d,s,\gamma,q,\|\bfb\|,\|c\|}
    |a|W(\omega,\xi).
\label{eq:cosine-mixed-bound}
\end{align}

It remains to estimate the initial trace. The function
$$
\widetilde Z_0(x)
:=
a\eta(x)\cos(\xi\cdot x+\theta)
$$
is an extension of $Z(0,\cdot)$ from $\Omega$ to $\mathbb R^d$.
Since
$$
\eta(x)\cos(\xi\cdot x+\theta)
=
\frac12
\left(
\mathrm{e}^{\mathrm{i}\theta}\eta(x)\mathrm{e}^{\mathrm{i}\xi\cdot x}
+
\mathrm{e}^{-\mathrm{i}\theta}\eta(x)\mathrm{e}^{-\mathrm{i}\xi\cdot x}
\right),
$$
the same modulation estimate gives
$$
\|\widetilde Z_0\|_{H^q(\mathbb R^d)}
\lesssim_{\eta,q}
    |a|(1+|\xi|^q).
$$
Since $0\leq q\leq\beta$ and $1+|\xi|\geq1$, $(1+|\xi|^q)
\leq
2(1+|\xi|^\beta)
\lesssim_{\beta}
    \left(
1+|\xi|^\beta
\right)$.
Hence,
\begin{equation}
\label{eq:cosine-initial-bound}
\|Z(0,\cdot)\|_{H^q(\Omega)}
\lesssim_{\Omega,q}
    |a|W(\omega,\xi).
\end{equation}
Combining
\eqref{eq:cosine-residual-bound},
\eqref{eq:cosine-mixed-bound}, and
\eqref{eq:cosine-initial-bound}, we obtain
$$
\begin{aligned}
\left\|
\left(
Z,
\mathcal A_{\bfb,c}Z,
Z(0,\cdot)
\right)
\right\|_{\mathbb G_{s,q}(\Omega)}^2\lesssim_{T,\Omega,d,s,\gamma,q,\|\bfb\|,\|c\|}
|a|^2W(\omega,\xi)^2.
\end{aligned}
$$
Taking square roots proves \eqref{eq:cosine-graph-bound}.

It remains to verify the continuity assertion. 
Let $(a_k,\omega_k,\xi_k,\theta_k)\rightarrow(a,\omega,\xi,\theta)$
and set $Z_k(t,x):=a_k\cos(\omega_kt+\xi_k\cdot x+\theta_k)$.
The standard continuity of the modulation maps in the relevant Sobolev spaces gives
$$
(Z_k,Z_k(0))\longrightarrow (Z,Z(0))
\qquad
\text{in }
\mathcal H_{1+\frac{s}{\gamma},\gamma+s}(\Omega)\times H^q(\Omega).
$$
Moreover, the explicit identity
$$
\mathcal A_{\bfb,c}Z_k
=
a_k
\left[
\left(
\bfb\cdot\xi_k\,1_{\{\gamma>1\}}-\omega_k
\right)
\sin(\omega_kt+\xi_k\cdot x+\theta_k)
+
\left(
|\xi_k|^\gamma-c
\right)
\cos(\omega_kt+\xi_k\cdot x+\theta_k)
\right]
$$
implies $
\mathcal A_{\bfb,c}Z_k
\longrightarrow
\mathcal A_{\bfb,c}Z$ in $L^2((0,T)\times\Omega)$.
Therefore, $
\left(
Z_k,
\mathcal A_{\bfb,c}Z_k,
Z_k(0,\cdot)
\right)
\longrightarrow
\left(
Z,
\mathcal A_{\bfb,c}Z,
Z(0,\cdot)
\right)$
in $\mathbb G_{s,q}(\Omega)$.
\end{proof}

We now prove the population estimate stated in
Corollary~\ref{cor:pde-solution-class}.

\begin{proof}[Proof of Corollary~\ref{cor:pde-solution-class}]
We divide the argument into three steps.

\noindent\textbf{Step 1. Fourier representation and graph atoms.}

Fix $q\in[0,\gamma+s]$, $R>0$, $n\in\mathbb{N}$, $(v_0,f)\in\mathcal{D}_R$ and let $v:=\mathrm{S}_{\bfb,c}(v_0,f)$.
By Theorem~\ref{thm: intro PDE estimate},
\begin{equation}
\label{eq:pinn-regularity-input}
\|v\|_{\mathcal B^{1+\frac{s}{\gamma},\gamma+s}(T)}
\lesssim_{T,s,\gamma,\|\bfb\|,\|c\|}
    \left(
\|v_0\|_{\mathcal B^{\gamma+s}(\mathbb R^d)}
+
\|f\|_{\mathcal B^{s/\gamma,s}(T)}
\right).
\end{equation}
Since the data and coefficients are real-valued, uniqueness implies that $v$ is real-valued.

By the definition of the finite-time restriction norm, there exists $U\in\mathcal{B}^{1+\frac{s}{\gamma},\gamma+s}(\mathbb{R}^{d+1})$ such that $U|_{[0,T]\times\mathbb R^d}=v$ and $\|U\|_{\mathcal{B}^{1+\frac{s}{\gamma},\gamma+s}(\mathbb{R}^{d+1})}
\leq
2\|v\|_{\mathcal B^{1+\frac{s}{\gamma},\gamma+s}(T)}$.
Set $V:=\Re[U]$.
Then $V$ is a real-valued extension of $v$. 
Moreover,
$$
\mathcal{F}_{d+1}[V](\tau,\xi)
=
\frac{
\mathcal{F}_{d+1}[U](\tau,\xi)
+
\overline{\mathcal{F}_{d+1}U(-\tau,-\xi)}
}{2}.
$$
One can directly check that $\|V\|_{\mathcal{B}^{1+\frac{s}{\gamma},\gamma+s}(\mathbb{R}^{d+1})}
\leq
\|U\|_{\mathcal{B}^{1+\frac{s}{\gamma},\gamma+s}(\mathbb{R}^{d+1})}$, and hence
\begin{equation}
\label{eq:pinn-I-bound}
\|V\|_{\mathcal{B}^{1+\frac{s}{\gamma},\gamma+s}(\mathbb{R}^{d+1})}
\leq
2\|v\|_{\mathcal B^{\alpha,\beta}(T)}.
\end{equation}

If $\|V\|_{\mathcal{B}^{1+\frac{s}{\gamma},\gamma+s}(\mathbb{R}^{d+1})}=0$, then $\mathcal{F}_{d+1}[V]=0$ almost everywhere and hence $V=0$.
Consequently, $v=v_0=f=0$.
The zero network then proves the desired estimate. 
We henceforth assume that $\|V\|_{\mathcal{B}^{1+\frac{s}{\gamma},\gamma+s}(\mathbb{R}^{d+1})}>0$.

Let $\psi:=\operatorname{Arg}_{[0,2\pi)}\mathcal{F}_{d+1}[V]$ on $\{\mathcal{F}_{d+1}[V]\neq0\}$, and set $\psi=0$ on $\{\mathcal{F}_{d+1}[V]=0\}$. 
Then $\psi$ is measurable and $\mathcal{F}_{d+1}[V]=|\mathcal{F}_{d+1}[V]|\mathrm{e}^{\mathrm{i}\psi}$.
Define the probability measure
$$
\rho_0(\mathrm{d}\tau,\mathrm{d}\xi)
:=
\frac{
(1+|\tau|^{1+\frac{s}{\gamma}}+|\xi|^{\gamma+s})|\mathcal{F}_{d+1}[V](\tau,\xi)|
}{\|V\|_{\mathcal{B}^{1+\frac{s}{\gamma},\gamma+s}(\mathbb{R}^{d+1})}}
\,\mathrm d\tau\,\mathrm d\xi
$$
on $\mathbb{R}\times\mathbb{R}^d$.
Let $\Psi(\tau,\xi):=(\tau,\xi,\psi(\tau,\xi))$.
We denote by $\mu:=\Psi_\#\rho_0$ the pushforward probability measure on $\mathbb{R}\times\mathbb{R}^d\times[0,2\pi]$.
Define
\begin{equation}
\label{eq:pinn-cosine-atom}
W_{\tau,\xi,\theta}(t,x)
:=
(2\pi)^{-(d+1)}
\frac{\|V\|_{\mathcal{B}^{1+\frac{s}{\gamma},\gamma+s}(\mathbb{R}^{d+1})}\cos\bigl(
\tau t+\xi\cdot x+\theta
\bigr)}{1+|\tau|^{1+\frac{s}{\gamma}}+|\xi|^{\gamma+s}}.
\end{equation}
Since $\mathcal{F}_{d+1}[V]\in L^1(\mathbb R^{d+1})$, the Fourier inversion integral converges absolutely.
By the definition of the pushforward measure $\mu=\Psi_\#\rho_0$, the polar decomposition $\mathcal{F}_{d+1}[V]=|\mathcal{F}_{d+1}[V]|\mathrm{e}^{\mathrm{i}\psi}$, and Fourier inversion, we have, for every $(t,x)\in\mathbb{R}\times\mathbb{R}^d$,
\begin{equation}
\label{eq:pinn-cosine-representation}
\begin{aligned}
\int_{\mathbb{R}\times\mathbb{R}^d\times[0,2\pi]} W_{\tau,\xi,\theta}(t,x)\,\mu(\mathrm{d}\tau,\mathrm{d}\xi,\mathrm{d}\theta)
&=
\int_{\mathbb R\times\mathbb R^d}
W_{(\tau,\xi,\psi(\tau,\xi))}(t,x)
\,\rho_0(\mathrm d\tau,\mathrm d\xi)\\
&=
\frac{1}{(2\pi)^{d+1}}
\int_{\mathbb{R}\times\mathbb{R}^d}
|\mathcal{F}_{d+1}[V](\tau,\xi)|
\cos\bigl(
\tau t+\xi\cdot x+\psi(\tau,\xi)
\bigr)
\,\mathrm{d}\tau\,\mathrm{d}\xi\\
&=
\Re
\left[
\frac{1}{(2\pi)^{d+1}}
\int_{\mathbb{R}\times\mathbb{R}^d}
\mathcal{F}_{d+1}[V](\tau,\xi)
\mathrm{e}^{\mathrm{i}(\tau t+\xi\cdot x)}
\,\mathrm{d}\tau\,\mathrm{d}\xi
\right]=V(t,x).
\end{aligned}
\end{equation}

For each atom, define its graph vector
$$
\mathbf{W}_{\tau,\xi,\theta}
:=
\left(
W_{\tau,\xi,\theta},
\mathcal A_{\bfb,c}W_{\tau,\xi,\theta},
W_{\tau,\xi,\theta}(0,\cdot)
\right)
\in
\mathbb{G}_{s,q}(\Omega).
$$
Lemma~\ref{lem:cosine-ridge-graph-norm}, applied with $a=(2\pi)^{-(d+1)}\|V\|_{\mathcal{B}^{1+\frac{s}{\gamma},\gamma+s}(\mathbb{R}^{d+1})}/(1+|\tau|^{1+\frac{s}{\gamma}}+|\xi|^{\gamma+s})$ and $\omega=\tau$, yields
\begin{equation}
\label{eq:pinn-graph-atom-uniform-bound}
M:=\sup_{\tau,\xi,\theta}\|\mathbf W_{\tau,\xi,\theta}\|_{\mathbb{G}_{s,q}(\Omega)}
\lesssim_{T,\Omega,d,s,\gamma,q,\|\bfb\|,\|c\|}
\|V\|_{\mathcal{B}^{1+\frac{s}{\gamma},\gamma+s}(\mathbb{R}^{d+1})}.
\end{equation}
Here the supremum is over $\mathbb R\times\mathbb R^d\times[0,2\pi]$.\medskip

\noindent\textbf{Step 2. Approximation by finite convex combinations.}

We claim that
\begin{equation}
\label{eq:pinn-closed-convex-hull-claim}
(v,f,v_0)
\in
\overline{\operatorname{co}}^{\,\mathbb{G}_{s,q}(\Omega)}
\left(
\{\mathbf 0\}
\cup
\{\mathbf W_{\tau,\xi,\theta}:(\tau,\xi,\theta)\in\mathbb{R}\times\mathbb{R}^d\times[0,2\pi]\}
\right).
\end{equation}
Here, for a subset $A\subset\mathbb G_{s,q}(\Omega)$, we use the
notation
\begin{align}
\operatorname{co}(A)
:=
\left\{
\sum_{k=1}^N\lambda_k\mathbf{z}_k
:
N\in\mathbb{N},
\mathbf{z}_k\in A,
\lambda_k\geq0,\sum_{k=1}^N\lambda_k=1
\right\},
\label{eq:convex-hull-definition}
\end{align}
and
\begin{align}
\overline{\operatorname{co}}^{\,\mathbb G_{s,q}(\Omega)}(A)
:=
\left\{
\mathbf z\in\mathbb G_{s,q}(\Omega)
:
\inf_{\mathbf y\in\operatorname{co}(A)}
\|\mathbf z-\mathbf y\|_{\mathbb G_{s,q}(\Omega)}
=0
\right\}.
\label{eq:closed-convex-hull-definition}
\end{align}

To prove the claim, let $E_N
:=
[-N,N]\times\overline{B_N(0)}
\subset
\mathbb{R}\times\mathbb{R}^d$.
Set $\Theta_N:=E_N\times[0,2\pi]$.
Define the truncated Fourier synthesis
\begin{equation*}
V_N(t,x)
:=
(2\pi)^{-(d+1)}
\Re\left[
\int_{E_N}
\mathcal{F}_{d+1}[V](\tau,\xi)
\mathrm{e}^{\mathrm{i}(\tau t+\xi\cdot x)}
\,\mathrm{d}\tau\,\mathrm{d}\xi\right].
\end{equation*}
Equivalently, by the definition of $\mu$,
\begin{equation}
\label{eq:pinn-truncated-cosine-synthesis}
V_N(t,x)
=
\int_{E_N\times[0,2\pi]}
W_{\tau,\xi,\theta}(t,x)
\,\mu(\mathrm{d}\tau,\mathrm{d}\xi,\mathrm{d}\theta).
\end{equation}
All integrals in
\eqref{eq:pinn-cosine-representation} and
\eqref{eq:pinn-truncated-cosine-synthesis}
are ordinary scalar integrals.
Set $R_N:=V-V_N$.
The same cutoff and modulation estimates used in the proof of Lemma~\ref{lem:cosine-ridge-graph-norm}, together with the scalar Minkowski inequality, give
\begin{equation}
\label{eq:pinn-truncation-solution-tail}
\|R_N\|_{\mathcal{H}_{1+\frac{s}{\gamma},\gamma+s}(\Omega)}
\lesssim_{T,\Omega,d,s,\gamma,q,\|\bfb\|,\|c\|}
    \int_{\mathbb{R}\times\mathbb{R}^d\setminus E_N} (1+|\tau|^{1+\frac{s}{\gamma}}+|\xi|^{\gamma+s})|\mathcal{F}_{d+1}[V](\tau,\xi)|
\,\mathrm{d}\tau\,\mathrm{d}\xi.
\end{equation}
Furthermore, $1+\frac{s}{\gamma}\geq1$, $\gamma+s\geq\gamma$, and, if $\gamma>1$, then $\gamma+s>1$. 
Hence $1+|\tau|+|\xi|^\gamma
+
1_{\{\gamma>1\}}|\xi|
\lesssim_{T,\Omega,d,s,\gamma,q,\|\bfb\|,\|c\|}(1+|\tau|^{1+\frac{s}{\gamma}}+|\xi|^{\gamma+s})$.
Using the explicit Fourier multipliers for $\partial_t$, $\Delta^{\gamma/2}$, and $\nabla$, we obtain
\begin{equation}
\label{eq:pinn-truncation-residual-tail}
\|\mathcal{A}_{\bfb,c}R_N\|_{L^2((0,T)\times\Omega)}
\lesssim_{T,\Omega,d,s,\gamma,q,\|\bfb\|,\|c\|}
    \int_{\mathbb{R}\times\mathbb{R}^d\setminus E_N}
(1+|\tau|^{1+\frac{s}{\gamma}}+|\xi|^{\gamma+s})|\mathcal{F}_{d+1}[V](\tau,\xi)|
\,\mathrm d\tau\,\mathrm d\xi.
\end{equation}
Similarly, since $q\leq\gamma+s$, the cutoff-modulation estimate in
$H^q(\Omega)$ gives
\begin{equation}
\label{eq:pinn-truncation-initial-tail}
\|R_N(0,\cdot)\|_{H^q(\Omega)}
\lesssim_{T,\Omega,d,s,\gamma,q,\|\bfb\|,\|c\|}
    \int_{\mathbb{R}\times\mathbb{R}^d\setminus E_N}
(1+|\tau|^{1+\frac{s}{\gamma}}+|\xi|^{\gamma+s})|\mathcal{F}_{d+1}[V](\tau,\xi)|
\,\mathrm d\tau\,\mathrm d\xi.
\end{equation}

The extension satisfies $V(t,\cdot)=v(t,\cdot)$ on the whole spatial domain $\mathbb R^d$ for every $t\in[0,T]$.
This global equality is needed in the nonlocal case and implies $\Delta^{\gamma/2}V=\Delta^{\gamma/2}v$ on $(0,T)\times\mathbb{R}^d$.
Together with $\partial_tV=\partial_tv$ and $\nabla V=\nabla v$ (when $\gamma>1$), this yields $\mathcal{A}_{\bfb,c}V=\mathcal{A}_{\bfb,c}v=f$ in $(0,T)\times\Omega$.
Moreover, $V(0,\cdot)=v(0,\cdot)=v_0$ on $\mathbb{R}^d$.
Combining
\eqref{eq:pinn-truncation-solution-tail}--
\eqref{eq:pinn-truncation-initial-tail}, we obtain
\begin{equation}
\label{eq:pinn-truncated-graph-convergence}
\begin{aligned}
&\left\|
\left(
V_N,
\mathcal A_{\bfb,c}V_N,
V_N(0,\cdot)
\right)
-
(v,f,v_0)
\right\|_{\mathbb G_{s,q}(\Omega)}\\
\lesssim_{T,\Omega,d,s,\gamma,q,\|\bfb\|,\|c\|}&
\int_{\mathbb{R}\times\mathbb{R}^d\setminus E_N} (1+|\tau|^{1+\frac{s}{\gamma}}+|\xi|^{\gamma+s})|\mathcal{F}_{d+1}[V](\tau,\xi)|
\,\mathrm{d}\tau\,\mathrm{d}\xi
\rightarrow0
\end{aligned}
\end{equation}
as $N\to\infty$.

It remains to show that, for each fixed $N$, the graph vector
$$
\mathbf V_N
:=
\left(
V_N,
\mathcal A_{\bfb,c}V_N,
V_N(0,\cdot)
\right)\in
\overline{\operatorname{co}}^{\,\mathbb G_{s,q}(\Omega)}
\left(
\{\mathbf 0\}
\cup
\{\mathbf W_\zeta:\zeta\in\Theta\}
\right).
$$
By the continuity assertion in
Lemma~\ref{lem:cosine-ridge-graph-norm}, the map
$$
(a,\omega,\xi,\theta)
\longmapsto
\left(
a\cos(\omega t+\xi\cdot x+\theta),
\mathcal A_{\bfb,c}
\bigl[
a\cos(\omega t+\xi\cdot x+\theta)
\bigr],
a\cos(\xi\cdot x+\theta)
\right)
$$
is continuous from
$\mathbb R\times\mathbb R\times\mathbb R^d\times\mathbb R$
into $\mathbb G_{s,q}(\Omega)$. Since
$$
(\tau,\xi,\theta)
\longmapsto
\left(
\frac{(2\pi)^{-(d+1)}\|V\|_{\mathcal{B}^{1+\frac{s}{\gamma},\gamma+s}(\mathbb{R}^{d+1})}}{(1+|\tau|^{1+\frac{s}{\gamma}}+|\xi|^{\gamma+s})},
\tau,\xi,\theta
\right)
$$
is continuous, it follows that $
\mathbb{R}\times\mathbb{R}^d\times[0,2\pi]\ni(\tau,\xi,\theta)
\longmapsto
\mathbf W_{\tau,\xi,\theta}
\in\mathbb G_{s,q}(\Omega)$
is continuous.

Since $E_N\times[0,2\pi]$ is compact, this map is uniformly continuous on $E_N\times[0,2\pi]$.
For each $\ell\in\mathbb{N}$, choose a finite Borel partition $E_N\times[0,2\pi]=A_{N,\ell,1}
\cup\cdots\cup
A_{N,\ell,L_{N,\ell}}$
such that
\begin{equation}
\label{eq:pinn-partition-oscillation}
\sup_{\zeta,\zeta'\in A_{N,\ell,k}}
\|\mathbf W_\zeta-\mathbf W_{\zeta'}\|_
{\mathbb G_{s,q}(\Omega)}
\leq
\frac1\ell
\end{equation}
for every $k$. For each set of positive $\mu$-measure, choose
$\zeta_{N,\ell,k}\in A_{N,\ell,k}$, and define
\begin{equation}
\label{eq:pinn-finite-convex-combination}
\mathbf Z_{N,\ell}
:=
\sum_{k=1}^{L_{N,\ell}}
\mu(A_{N,\ell,k})
\mathbf W_{\zeta_{N,\ell,k}}.
\end{equation}
Terms corresponding to sets of zero measure are omitted.

Since $A_{N,\ell,1},\ldots,A_{N,\ell,L_{N,\ell}}$ form a partition of $E_N\times[0,2\pi]$, we have
$$
\sum_{k=1}^{L_{N,\ell}}
\mu(A_{N,\ell,k})
=
\mu(\Theta_N)
\leq1,
$$
because $\mu$ is a probability measure on $\mathbb{R}\times\mathbb{R}^d\times[0,2\pi]$.
Hence, all coefficients on the right-hand side of
\begin{align}
\mathbf Z_{N,\ell}
&=
\bigl(1-\mu(E_N\times[0,2\pi])\bigr)\mathbf0
+
\sum_{k=1}^{L_{N,\ell}}
\mu(A_{N,\ell,k})
\mathbf W_{\zeta_{N,\ell,k}}.
\label{eq:pinn-convex-combination-with-zero}
\end{align}
are nonnegative, and their sum is
$$
1-\mu(E_N\times[0,2\pi])
+
\sum_{k=1}^{L_{N,\ell}}
\mu(A_{N,\ell,k})
=
1.
$$
Thus, $\mathbf{Z}_{N,\ell}
\in\operatorname{co}\left(\{\mathbf0\}\cup\{\mathbf{W}_\zeta:\zeta\in\Theta_N\}\right)$.

We next show that $\mathbf{Z}_{N,\ell}\rightarrow \mathbf{V}_N$ in $\mathbb{G}_{s,q}(\Omega)$ as $\ell\to\infty$.
Let $\mathbf{h}\in\mathbb G_{s,q}(\Omega)$.
The scalar Fourier representation \eqref{eq:pinn-truncated-cosine-synthesis}, differentiation under the bounded-frequency scalar integral, and evaluation at $t=0$ imply the weak identity
\begin{equation}
\label{eq:pinn-truncated-weak-identity}
\left\langle
\mathbf V_N,\mathbf h
\right\rangle_{\mathbb G_{s,q}(\Omega)}
=
\int_{E_N\times[0,2\pi]}
\left\langle
\mathbf{W}_{\tau,\xi,\theta},\mathbf h
\right\rangle_{\mathbb G_{s,q}(\Omega)}
\,\mu(\mathrm{d}\tau,\mathrm{d}\xi,\mathrm{d}\theta).
\end{equation}
Here the integral on the right-hand side is an ordinary scalar Lebesgue integral.
By the definition of $\mathbf{Z}_{N,\ell}$ and the fact that $A_{N,\ell,1},\cdots,A_{N,\ell,L_{N,\ell}}$ form a partition of $E_N\times[0,2\pi]$, we have
\begin{align*}
\left\langle
\mathbf Z_{N,\ell}-\mathbf V_N,
\mathbf h
\right\rangle_{\mathbb G_{s,q}(\Omega)}=
\sum_{k=1}^{L_{N,\ell}}
\int_{A_{N,\ell,k}}
\left\langle
\mathbf W_{\zeta_{N,\ell,k}}
-
\mathbf W_\zeta,
\mathbf h
\right\rangle_{\mathbb G_{s,q}(\Omega)}
\,\mu(\mathrm d\zeta).
\end{align*}
Therefore, by the Cauchy--Schwarz inequality in $\mathbb{G}_{s,q}(\Omega)$ and \eqref{eq:pinn-partition-oscillation},
\begin{align*}
\left|
\left\langle
\mathbf Z_{N,\ell}-\mathbf V_N,
\mathbf h
\right\rangle_{\mathbb G_{s,q}(\Omega)}
\right|
&\leq
\sum_{k=1}^{L_{N,\ell}}
\int_{A_{N,\ell,k}}
\left\|
\mathbf W_{\zeta_{N,\ell,k}}
-
\mathbf W_\zeta
\right\|_{\mathbb G_{s,q}(\Omega)}
\|\mathbf h\|_{\mathbb G_{s,q}(\Omega)}
\,\mu(\mathrm d\zeta)
\\
&\leq
\frac1\ell
\sum_{k=1}^{L_{N,\ell}}
\mu(A_{N,\ell,k})
\|\mathbf h\|_{\mathbb G_{s,q}(\Omega)}=
\frac{\mu(E_N\times[0,2\pi])}{\ell}
\|\mathbf h\|_{\mathbb G_{s,q}(\Omega)}\leq
\frac1\ell
\|\mathbf h\|_{\mathbb G_{s,q}(\Omega)}.
\end{align*}
Taking the supremum over all
$\mathbf{h}\in\mathbb{G}_{s,q}(\Omega)$ satisfying $\|\mathbf{h}\|_{\mathbb{G}_{s,q}(\Omega)}\leq1$, we obtain $\|\mathbf Z_{N,\ell}-\mathbf V_N\|_{\mathbb G_{s,q}(\Omega)}
\leq
\frac1\ell$.
For every $\ell\in\mathbb N$, we have $
\mathbf Z_{N,\ell}
\in
\operatorname{co}
\left(
\{\mathbf0\}
\cup
\{\mathbf W_\zeta:\zeta\in\Theta_N\}
\right)$.
Since $\mathbf Z_{N,\ell}
\rightarrow
\mathbf V_N$ in $\mathbb G_{s,q}(\Omega)$, it follows from the definition of the closed convex hull that
\begin{align*}
\mathbf V_N
\in
\overline{\operatorname{co}}^{\,\mathbb G_{s,q}(\Omega)}
\left(
\{\mathbf0\}
\cup
\{\mathbf W_\zeta:\zeta\in\Theta_N\}
\right)\subset
\overline{\operatorname{co}}^{\,\mathbb G_{s,q}(\Omega)}
\left(
\{\mathbf0\}
\cup
\{\mathbf W_\zeta:\zeta\in\Theta\}
\right).
\end{align*}
The set $
\overline{\operatorname{co}}^{\,\mathbb G_{s,q}(\Omega)}
\left(
\{\mathbf0\}
\cup
\{\mathbf W_\zeta:\zeta\in\Theta\}
\right)$ is closed in $\mathbb{G}_{s,q}(\Omega)$ and contains $\mathbf V_N$ for every $N$. 
Since $\mathbf V_N
\longrightarrow
(v,f,v_0)$ in $\mathbb{G}_{s,q}(\Omega)$ by \eqref{eq:pinn-truncated-graph-convergence}, we have \eqref{eq:pinn-closed-convex-hull-claim}.\medskip

\noindent\textbf{Step 3. Sampling and the final estimate.}

Having established
\eqref{eq:pinn-closed-convex-hull-claim}, we now apply an elementary
sampling argument to a finite convex approximation of the graph vector $\mathbf{v}:=(v,f,v_0)
\in\mathbb{G}_{s,q}(\Omega)$.
Fix $\varepsilon>0$. 
By \eqref{eq:pinn-closed-convex-hull-claim} and the definition of the
closed convex hull, there exist
$L\in\mathbb{N}$, $\zeta_1,\cdots,\zeta_L\in\mathbb{R}\times\mathbb{R}^d\times[0,2\pi]$, and coefficients $\lambda_0,\lambda_1,\cdots,\lambda_L\geq0$ such that, upon setting $\mathbf{Z}_0:=\mathbf0$, 
$\mathbf{Z}_k:=\mathbf{W}_{\zeta_k}$ for $k=1,\cdots,L$,
we have
\begin{equation}
\label{eq:pinn-finite-convex-approximation-v}
\left\|
\mathbf v-\overline{\mathbf Z}
\right\|_{\mathbb G_{s,q}(\Omega)}
<
\varepsilon,
\qquad
\overline{\mathbf Z}
:=
\sum_{k=0}^L\lambda_k\mathbf Z_k,\qquad \sum_{k=0}^L\lambda_k=1.
\end{equation}
For $k=0,\ldots,L$, define the centered vectors $\mathbf{Y}_k:=\mathbf{Z}_k-\overline{\mathbf{Z}}$.
Then
\begin{equation}
\label{eq:pinn-centered-finite-family}
\sum_{k=0}^L\lambda_k\mathbf Y_k
=
\sum_{k=0}^L\lambda_k\mathbf Z_k
-
\overline{\mathbf Z}
=
\mathbf0.
\end{equation}
Consider the finite weighted average
\begin{align}
\mathscr E_n
:=
\sum_{i_1,\ldots,i_n=0}^L
\lambda_{i_1}\cdots\lambda_{i_n}
\left\|
\frac1n
\sum_{j=1}^n
\mathbf Z_{i_j}
-
\overline{\mathbf Z}
\right\|_{\mathbb G_{s,q}(\Omega)}^2.
\label{eq:pinn-finite-sampling-average}
\end{align}
Since
$$
\frac1n
\sum_{j=1}^n
\mathbf Z_{i_j}
-
\overline{\mathbf Z}
=
\frac1n
\sum_{j=1}^n
\mathbf Y_{i_j},
$$
expanding the square gives
\begin{align*}
\mathscr E_n=\frac1{n^2}
\sum_{j=1}^n
\sum_{k=0}^L
\lambda_k
\|\mathbf Y_k\|_{\mathbb G_{s,q}(\Omega)}^2+
\frac1{n^2}
\sum_{\substack{j,r=1\\j\neq r}}^n
\left\langle
\sum_{k=0}^L\lambda_k\mathbf Y_k,
\sum_{\ell=0}^L\lambda_\ell\mathbf Y_\ell
\right\rangle_{\mathbb G_{s,q}(\Omega)}.
\end{align*}
The second term vanishes by
\eqref{eq:pinn-centered-finite-family}. Hence
\begin{align*}
\mathscr{E}_n=
\frac1n
\sum_{k=0}^L
\lambda_k
\|\mathbf Z_k-\overline{\mathbf Z}\|_
{\mathbb G_{s,q}(\Omega)}^2=
\frac1n
\left(
\sum_{k=0}^L
\lambda_k
\|\mathbf Z_k\|_{\mathbb G_{s,q}(\Omega)}^2
-
\|\overline{\mathbf Z}\|_{\mathbb G_{s,q}(\Omega)}^2
\right)
&\leq
\frac{M^2}{n},
\end{align*}
where the last inequality follows from
\eqref{eq:pinn-graph-atom-uniform-bound} and
$\sum_{k=0}^L\lambda_k=1$.

The coefficients $\lambda_{i_1}\cdots\lambda_{i_n}$ in \eqref{eq:pinn-finite-sampling-average} are nonnegative, and their sum over all index tuples is
$$
\sum_{i_1,\ldots,i_n=0}^L
\lambda_{i_1}\cdots\lambda_{i_n}
=
\left(
\sum_{k=0}^L\lambda_k
\right)^n
=
1.
$$
Therefore, at least one index tuple $(i_1,\ldots,i_n)\in\{0,\ldots,L\}^n$ satisfies
\begin{equation}
\label{eq:pinn-selected-sample}
\left\|
\frac1n
\sum_{j=1}^n
\mathbf Z_{i_j}
-
\overline{\mathbf Z}
\right\|_{\mathbb G_{s,q}(\Omega)}
\leq
\frac{M}{\sqrt n}.
\end{equation}

For $k=1,\ldots,L$, let $W_k:=W_{\zeta_k}$, and set $W_0:=0$.
Then $\mathbf{Z}_k
=
\left(W_k,
\mathcal A_{\bfb,c}W_k,
W_k(0,\cdot)\right)$ for $k=0,\cdots,L$.
Define
\begin{equation}
\label{eq:pinn-sampled-network}
\Phi_{n,\varepsilon}
:=
\frac1n
\sum_{j=1}^n W_{i_j}.
\end{equation}
By linearity of $\mathcal A_{\bfb,c}$ and of the initial trace,
\begin{equation}
\label{eq:pinn-sampled-graph-vector}
\frac1n\sum_{j=1}^n\mathbf Z_{i_j}
=
\left(
\Phi_{n,\varepsilon},
\mathcal A_{\bfb,c}\Phi_{n,\varepsilon},
\Phi_{n,\varepsilon}(0,\cdot)
\right).
\end{equation}
Moreover, every nonzero $W_{i_j}$ is a single cosine ridge atom.
Consequently,
$$
\Phi_{n,\varepsilon}(t,x)
=
\sum_{\substack{1\leq j\leq n\\ i_j\neq0}}
\frac{a_{i_j}}{n}
\cos\bigl(
\tau_{i_j}t+\xi_{i_j}\cdot x+\theta_{i_j}
\bigr),
$$
and hence $\Phi_{n,\varepsilon}$ is a cosine network of width at most $n$. 
Therefore, $\Phi_{n,\varepsilon}\in\Sigma_n^{\cos}$.
Combining \eqref{eq:pinn-finite-convex-approximation-v}, \eqref{eq:pinn-selected-sample}, and \eqref{eq:pinn-sampled-graph-vector}, we obtain
\begin{align*}
\left\|
\left(
\Phi_{n,\varepsilon},
\mathcal A_{\bfb,c}\Phi_{n,\varepsilon},
\Phi_{n,\varepsilon}(0,\cdot)
\right)
-
(v,f,v_0)
\right\|_{\mathbb G_{s,q}(\Omega)}
&\leq
\left\|
\frac1n
\sum_{j=1}^n
\mathbf Z_{i_j}
-
\overline{\mathbf Z}
\right\|_{\mathbb G_{s,q}(\Omega)}
+
\|\overline{\mathbf Z}-\mathbf v\|_{\mathbb G_{s,q}(\Omega)}<
\frac{M}{\sqrt n}
+
\varepsilon.
\end{align*}
By the definition of the graph norm, this is equivalent to
\begin{align*}
\|v-\Phi_{n,\varepsilon}\|_
{\mathcal H_{\alpha,\beta}(\Omega)}^2
+
\|\mathcal A_{\bfb,c}\Phi_{n,\varepsilon}-f\|_{L^2(Q_T)}^2
+
\|\Phi_{n,\varepsilon}(0,\cdot)-v_0\|_{H^q(\Omega)}^2
<
\left(
\frac{M}{\sqrt n}
+
\varepsilon
\right)^2.
\end{align*}
Since this holds for every $\varepsilon>0$, taking the infimum over
$\Phi\in\Sigma_n^{\cos}$ and then letting $\varepsilon\downarrow0$
gives
\begin{align}
\label{eq:pinn-sampled-graph-bound}
\inf_{\Phi\in\Sigma_n^{\cos}}
\Bigl[
\|v-\Phi\|_{\mathcal H_{\alpha,\beta}(\Omega)}^2
+
\|\mathcal A_{\bfb,c}\Phi-f\|_{L^2(Q_T)}^2
+
\|\Phi(0,\cdot)-v_0\|_{H^q(\Omega)}^2
\Bigr]
\leq
\frac{M^2}{n}.
\end{align}
By \eqref{eq:pinn-graph-atom-uniform-bound}, \eqref{eq:pinn-I-bound}, and \eqref{eq:pinn-regularity-input},
\begin{align}
\label{eq:pinn-final-graph-bound}
\begin{aligned}
\frac{M^2}{n}
&\lesssim_{T,\Omega,d,s,\gamma,q,\|\bfb\|,\|c\|}
\frac{1}{n}
\|v\|_{\mathcal B^{\alpha,\beta}(T)}^2
\\
&\lesssim_{T,\Omega,d,s,\gamma,q,\|\bfb\|,\|c\|}
\frac{1}{n}
\left(
\|v_0\|_{\mathcal B^\beta(\mathbb R^d)}
+
\|f\|_{\mathcal B^{s/\gamma,s}(T)}
\right)^2\\
&\lesssim_{T,\Omega,d,s,\gamma,q,\|\bfb\|,\|c\|}
\frac{R^2}{n}.
\end{aligned}
\end{align}
Finally, by the definition of the population loss,
\begin{align*}
\mathcal J^{(q)}_{\Omega;v_0,f}(\Phi)^2
:=
\|\mathcal A_{\bfb,c}\Phi-f\|_{L^2(Q_T)}^2
+
\|\Phi(0,\cdot)-v_0\|_{H^q(\Omega)}^2.
\end{align*}
Estimates \eqref{eq:pinn-sampled-graph-bound} and \eqref{eq:pinn-final-graph-bound} yield
\begin{align}
\label{eq:pinn-fixed-data-population-bound}
\begin{aligned}
&\inf_{\Phi\in\Sigma_n^{\cos}}
\left[
\|\mathrm{S}_{\bfb,c}(v_0,f)-\Phi\|_
{\mathcal H_{1+\frac{s}{\gamma},\gamma+s}(\Omega)}^2
+
\mathcal J^{(q)}_{\Omega;v_0,f}(\Phi)^2
\right]\lesssim_{T,\Omega,d,s,\gamma,q,\|\bfb\|,\|c\|}
\frac{R^2}{n}.
\end{aligned}
\end{align}
Taking the supremum over
$(v_0,f)\in\mathcal D_R$ proves
\eqref{eq:pde-solution-class}.
\end{proof}

\section*{Acknowledgements}
\pdfbookmark{Acknowledgements}{Acknowledgements}
H. Lim has been supported by a KIAS Individual Grant (AP109701) via the Center for Artificial Intelligence and Natural Sciences at Korea Institute for Advanced Study (KIAS). 
J. Seo has been supported by a KIAS Individual Grant (MG095802) at Korea Institute for Advanced Study and by the National Research Foundation of Korea (NRF) grant funded by the Korea government (Ministry of Science and ICT) (No. RS-2026-25477288).
Y.-J. Sim has been supported by a KIAS Individual Grant (MG108801) at Korea Institute for Advanced Study.
C. Song has been supported by Basic Science Research Programs through the National Research Foundation of Korea (NRF) funded by the Ministry of Education (RS-2024-00462755).
\vspace{1em}

\section*{Declaration of interest}
The authors declare no competing interests.

\end{document}